%% file: main.tex
\documentclass[twoside]{amsart}
\usepackage{pacchetto}

\author{Ludovica Buelli}
\address{Dipartimento di Matematica; Università di Genova;
Via Dodecaneso 35, 16146 Genova, Italia}
\email{ludovica.buelli@edu.unige.it, ludovica.buelli@hotmail.com}

\title[Locally trivial monodromy of moduli spaces]%
{Locally trivial monodromy of moduli spaces\\ of sheaves on Abelian surfaces}

\begin{document}
\begin{abstract}The aim of this work is to give a description of the locally trivial monodromy group of irreducible symplectic varieties arising from moduli spaces of semi-\linebreak[4]stable sheaves on Abelian surfaces with non-primitive Mukai vector. The outcome 
is that the locally trivial monodromy group of a singular moduli space of this type is isomorphic to the monodromy group of a smooth moduli space, extending Markman’s and Mongardi's description to the non-primitive case. As a consequence, we also prove the SYZ conjecture for any singular moduli space of this type.
\end{abstract}

\maketitle
\vspace{-5ex}

\tableofcontents
\vspace{-5ex}
\section*{Introduction}
\input{intro}

\section{Preliminaries}
\input{sec1}
\section{An injective morphism from \texorpdfstring{$\monlt(K_v(S,H))$}{Mon2lt(Kv(S,H))}  to \texorpdfstring{$\mon^{2}(K_w(S,H))$}{Mon2(Kw(S,H))}}
\input{sec2}
\section{A groupoid representation}
\input{sec3}
\section{The locally trivial monodromy group}
\input{sec4}
\section{An application: the SYZ conjecture}
\input{sec5}
\appendix
\section{Some lattice theory results}
\input{appendix}

\bibliographystyle{alpha}


\end{document}

%% file: intro.tex
\renewcommand{\theequation}{\Roman{equation}}

\setcounter{equation}{0}
This work takes place within the broader problem of the bimeromorphic classification of singular symplectic varieties. Among these, primitive and irreducible symplectic varieties have drawn increasing interest after several recent developments in this research area. For instance, a generalization of Bogomolov Decomposition Theorem to the setting of projective varieties with canonical singularities was provided by \cite{GKP11}, \cite{Dr18}, \cite{DG18}, \cite{GGK19} and \cite{HP17}, identifying irreducible symplectic varieties as fundamental building blocks, as well as natural generalizations of the notion of irreducible holomorphic symplectic manifolds to the singular setting. Furthermore, the recent formulation of Global Torelli Theorem for primitive symplectic varieties (\cite{BL18}) has carefully prescribed the fundamental steps to follow in order to address their classification: firstly, up to (locally trivial) deformation equivalence, then, up to bimeromorphism. 

In simple terms, the geometry of such varieties is completely encoded in the lattice and weight two Hodge structure of their second integral cohomology group. The key tool to decode such information lies in the description of those isometries that are \textit{(locally trivial) monodromy operators}. For any irreducible symplectic variety $X$, this class of isometries defines the \textit{(locally trivial) monodromy group} $\monlt(X)$ of $X$, which is a finite index subgroup of the orthogonal group $\O(\H^2(X,\Z))$ of the lattice $(\H^2(X,\Z),q_X)$ and represents the fundamental tool to detect bimeromorphism classes inside a fixed (locally trivial) deformation class.

The above-described procedure faithfully parallels the classification machinery for irreducible holomorphic symplectic manifolds - i.e. simply connected compact K\"{a}hler manifolds equipped with a unique (up to scalar) holomorphic symplectic form - developed by \cite{Ver09}, \cite{Huy97} and \cite{Mar11}. In that case, a description of the monodromy group has been achieved for any known deformation class by \cite{Mar08}, \linebreak[4]\cite{Mar10}, \cite{Mon14}, \cite{Mar18}, \cite{MR19} and \cite{Ono22}. In the singular setting, the problem has been addressed only partially, mainly due to the existence of a wide list of new examples. \\

\vspace{-0.2ex}
Moduli spaces of semistable sheaves on symplectic surfaces are a successful way to produce infinitely many (locally trivial) deformation classes of irreducible symplectic varieties and these will be the primary focus of this paper. A remarkable result by \cite{PR18} asserts that, if $S$ is either a projective K3 or an Abelian surface, $v$ is a non-primitive Mukai vector and $H$ is a $v-$generic polarization, then the moduli spaces $M_v(S,H)$ and $\Kv(S,H)$, respectively, of $H-$semistable sheaves on $S$ with Mukai vector $v$ are irreducible symplectic varieties and their deformation class is determined by the vector $v$.

The first explicit computations of the locally trivial monodromy group of a singular symplectic variety $X$ have been achieved precisely in the case of moduli spaces of sheaves, by \cite{MR19} and \cite{Ono22} for deformations of the two singular O'Grady examples, and by \cite{OPR23}, in the case in which $X$ is locally trivial deformation equivalent to any moduli space $\Mv(S,H)$ as above, with $S$ a projective K3 surface. In particular, in this case, Theorem A.1 of \cite{OPR23} provides the following identification: $$\monlt(X)=\W(\H^2(X,\Z)),$$ where $\W(\H^2(X,\Z))$ is the group of orientation preserving isometries of $\H^2(X,\Z)$ acting as $\pm\id$ on its discriminant group. \\ 

\vspace{-0.2ex}
The aim of this work is to complete the computation of the locally trivial monodromy group of moduli spaces of sheaves, by computing the remaining case in which $X$ is locally trivial deformation equivalent to $\Kv(S,H)$, with $S$ an Abelian surface. Indeed, it turns out that the techniques developed by \cite{OPR23} can be adapted to the Abelian case to achieve an analogous description, involving the index two subgroup $\Nn(\H^2(X,\Z))$ of $\W(\H^2(X,\Z))$ defined by isometries in the kernel of the character $\det\cdot \disc$. The latter is the main result of this paper.
\begin{thmA}[Theorem \ref{mainthm}, Corollary \ref{maincor}] \label{thmA1}
Let $X$ be an irreducible symplectic variety that is locally trivial deformation equivalent to a moduli space $\Kv(S,H)$, where $S$ is an Abelian surface, $v=mw$ is a Mukai vector with $m>1$ and $w$ primitive and such that $w^2>4$, and $H$ is a $v-$generic polarization. Then $$\monlt(X)=\Nn(\H^2(X,\Z)).$$
\end{thmA}
The outcome is that, regardless of the smoothness of the moduli space of sheaves, the monodromy description is the same. Indeed, in the case of an irreducible holomorphic symplectic manifold $X$ deformation equivalent to a moduli space $\Kv(S,H)$ as above, with $v$ primitive and of square greater than $4$, the combination of \cite{Mar18} and \cite{Mon14} provided the same description: \begin{equation}
    \label{marmon}\mon^2(X)=\Nn(\H^2(X,\Z)).
\end{equation}
Furthermore, Theorem \ref{thmA1} allows us to compute the index of $\Nn(\H^2(X,\Z))$ as a subgroup of the orientation preserving isometries group $\O^+(\H^2(X,\Z))$ of $\H^2(X,\Z)$ (Corollary \ref{corindex}, see also \cite[Lemma 4.1]{Mar10}). In particular, we deduce that the locally trivial monodromy group $\monlt(X)$ of $X$ is never maximal, theoretically providing an infinite number of new counterexamples to the conjectural Classic bimeromorphic Torelli Theorem (see \cite{Nam02b}).\\

\vspace{-0.2ex}
The relation between monodromy groups of smooth and singular moduli spaces of sheaves is deeper and more geometric than an equivalent lattice-theoretic description. To make it explicit, we present an equivalent formulation of Theorem \ref{thmA1}. 
 If $X$ is locally trivial deformation equivalent to a moduli space $\Kv(S,H)$, where $v=mw$, with $m>1$ and $w$ a primitive Mukai vector such that $w^2>4$, then each connected component $Z$ of its most singular locus is a smooth irreducible holomorphic symplectic manifold deformation equivalent to the moduli space $\Kw(S,H)$. Any locally trivial deformation $p\colon \XX \to T$ of $X$ induces, by restriction, a smooth deformation $p_{|\ZZ}\colon \ZZ \to T$ of $Z$. Consequently, any monodromy operator on $X$ defined by a loop $\gamma$ in the locally trivial family $p$ can be associated to the monodromy operator on $Z$ defined by the same loop $\gamma$ in the restricted smooth deformation $p_{|\ZZ}$. This association turns out to be well defined and an isomorphism on the respective monodromy groups, as stated in the following result.
\begin{thmB}[Corollary \ref{corinjmon}, Corollary \ref{maincor}]\label{thmB1}
Let $X$ be an irreducible symplectic variety locally trivial deformation equivalent to a moduli space $\Kv(S,H)$ as above and let $Z$ be a connected component of the most singular locus of $X$. Then its closed embedding $i_{Z,X}\colon Z \to X$ induces an isomorphism of groups $$i^\sharp_{Z,X}\colon \monlt(X) \longrightarrow \mon^2(Z).$$\end{thmB}


The outcome can be summarized by saying that the locally trivial monodromy group of a singular moduli space as above can completely be recovered by the monodromy group of the associated smooth moduli space and an analogue philosophy applies to any locally trivial deformation of these. In fact, the same phenomenon arose in the K3 case dealt in \cite{OPR23}, where the strategy above-mentioned is developed for a connected most singular locus.

We conclude by pointing out that the hypothesis $w^2>4$ is crucial and necessary (see Example \ref{exogsing}) and will be discussed in Section \ref{seck12}. The remaining cases $w^2=2$ and $w^2=4$ need a different approach and will be treated in a subsequent work.\\

In conclusion, we use the monodromy description given by Theorem \ref{thmB1} to provide a proof of the SYZ conjecture for any irreducible symplectic variety $X$ locally trivial deformation equivalent to a moduli space $\Kv(S,H)$ as above. According to this conjecture, nef isotropic line bundles on symplectic varieties are expected to define Lagrangian fibrations on the latter. This is precisely the content of the following Theorem. \begin{thmA}{(Theorem \ref{SYZAb})} Let $X$ be an irreducible symplectic variety locally trivial deformation equivalent to a moduli space $\Kv(S,H)$ as above. If $L$ is a nef line bundle on $X$ such that $q_X(L)=0$, then there exists a Lagrangian fibration $f\colon X \to B$ such that $L=f\Star\OO_B(1)$.
\end{thmA}
Also in this case, the outcome is analogous to that of moduli spaces of sheaves $\Mv(S,H)$, with $S$ is a K3 surface, which has been recently achieved by \cite{OO25}. 
\subsection*{Outline of the proof} We start by reducing to deal with $X=\Kv(S,H)$, where $S$ is an Abelian surface, $v=mw$ is a Mukai vector such that $m>1$ and $w$ is primitive of square $w^2>4$, and $H$ is a $v-$generic polarization, as the locally trivial monodromy group is a locally trivial deformation invariant.

The first step consists of the construction of an injective morphism \begin{equation}
    \label{iwmsharp}i_{w,m}^\sharp\colon \monlt(\Kv(S,H)) \longrightarrow \mon^2(\Kw(S,H)).
\end{equation} The latter is induced - via conjugation on isometries - by the pull-back action on the respective second integral cohomology groups of the closed embedding $$i_{w,m}\colon \Kw(S,H)\to \Kv(S,H)$$ as one of the finite connected components of the most singular locus of $\Kv(S,H)$. This is the content of Corollary \ref{corinjmon} (see also Proposition \ref{propinjPT}). Proposition \ref{istar} and Lemma \ref{lemisharp} guarantee the compatibility of the action of $i_{w,m}^\sharp$ with the lattice-theoretic descriptions of the second integral cohomology groups of the moduli spaces involved and of their respective orthogonal groups. Indeed, the Mukai vector $v$ belongs to the \textit{Mukai lattice} $\Htilde(S,\Z)$ of $S$, and the study of its orthogonal complement $v\ort$ in the latter yields the following non-trivial result, due to \cite{Yos01a} in the primitive case and to \cite{PR20} in the non-primitive case: there exists a chain of isometries - and obvious identifications - \begin{equation}
    \label{lambdaa}\begin{tikzcd}[column sep=large]
        \H^{2}(\Kv(S,H),\Z) \arrow[r, "\lambda_{\SvH}^{-1}"] & v\ort \arrow[r,equal] & w \ort \arrow[r,"\lambda_{\SwH}"] & \H^{2}(\Kw(S,H),\Z).
    \end{tikzcd}
\end{equation}Lemma \ref{lemisharp} asserts that the conjugation action of the composition (\ref{lambdaa}) coincides with $i_{w,m}^\sharp$. 

As a second step, we proceed to show the inclusion of groups $$\Nn(\H^2(\Kv(S,H),\Z))\subseteq \monlt(\Kv(S,H))$$ (Theorem \ref{thmNmon}). Paralleling \cite{Mar18}, we prove the inclusion by describing the generators of $\Nn(\H^2(\Kv(S,H),\Z))$ as monodromy operators induced by \begin{itemize}
    \item monodromy operators of the base surface $S$ 
    (Section \ref{secGmkdef} and Section \ref{secsurfacemon})
    \item pushforwards of isomorphisms of moduli spaces induced by some Fourier-Mukai equivalences on $\Db(S)$ (Section \ref{secGmkFM} and Section \ref{secsurj}).
\end{itemize} This procedure applies to the primitive case as well, and differs from the one in \cite{Mar18} only by a slight different choice of generators (Definition \ref{defgmkFM}) and the avoidance of non-projective families (Section \ref{secsurfacemon}). The compatibility proven in Lemma \ref{lemisharp} implies the equality $$i_{w,m}^\sharp(\Nn(\H^2(\Kv(S,H),\Z)))=\Nn(\H^2(\Kw(S,H),\Z)),$$ which, combined with (\ref{marmon}), forces surjectivity of $i^\sharp_{w,m}$ and simultaneously concludes the proof of both Theorem \ref{thmA1} and Theorem \ref{thmB1}, as summarized in the following commutative diagram. 
\begin{center}
      \begin{tikzcd}[ampersand replacement=\&, row sep=large]
          \monlt(\Kv(S,H))\arrow[rrr, hook, "i^{\sharp}_{w,m}", "\text{Section \ref{secinj}}" swap, dashed] \&\&\& \mon^{2}(\Kw(S,H))\arrow[d, equal, "\text{\cite{Mar18}, \cite{Mon14} }" swap]\\
          \Nn(\H^2(\Kv(S,H),\Z))\arrow[u, "\text{ Section \ref{secsurj}}" swap, " \rotatebox{90}{\(\subseteq\)}", dashed, hook]\arrow[rrr, bend right=15, "i^{\sharp}_{w,m}"]\arrow[r, "\sim" swap, "(\lambda_{\SvH}^{\sharp})^{-1}"] \& \Nn(v\ort) \arrow[r,equal]\& \Nn(w\ort)\arrow[r, "\sim" swap, "\lambda_{\SwH}^{\sharp}"] \&  \Nn(\H^2(\Kw(S,H),\Z))
      \end{tikzcd}
  \end{center}

\subsection*{Plan of the paper} Section \ref{secpreliminaries} is devoted to recall basic results concerning singular symplectic varieties, moduli spaces of sheaves and their deformation theory.

In Section \ref{secinjectivemorph} we relate the locally trivial monodromy group of a singular moduli space with the monodromy group of a smooth moduli space by means of an injective morphism of groups which arises from an explicit geometric construction.

In Section \ref{secgroupoid} we introduce a groupoid representation designed to encode the monodromy information of a moduli space in a natural way, extending \cite[Section 9]{Mar18}. This is the technical tool that will allow us to construct monodromy operators in the next Section.

In Section \ref{secmonodromy} the computation of the locally trivial monodromy group is completed, by showing surjectivity of the morphism introduced in Section \ref{secinjectivemorph}. This is the main result of the paper. Its extension to the whole locally trivial deformation class and few corollaries conclude the Section.

In Section \ref{secSYZ}, the SYZ conjecture for singular moduli spaces of sheaves on Abelian surfaces is proven, as a geometric application of the main results of this paper.

Appendix \ref{appendix} is meant to collect the lattice-theoretic material that is used throughout the paper, including the notion of orientation and the definition of the groups of isometries involved in the characterization of the monodromy groups.

\subsection*{Acknowledgements} This work is the outcome of my PhD project. I am deeply grateful to my advisor, Arvid Perego, for suggesting this problem, for carefully supervising its development and for the constant and precious support provided during the research. I wish to thank Antonio Rapagnetta and Claudio Onorati for their valuable guidance and advice, which greatly helped me to carry out this work, and Giovanni Mongardi for his bright comments and suggestions. I thank again Claudio Onorati, together with Ángel David Ríos Ortiz, for drawing my attention to an interesting application of this result. I am especially grateful to Simone Billi for his significant help in the final stage of the project, particularly for suggesting the proof of Lemma \ref{lemsimo}. Together with him, I deeply and sincerely thank Valeria Bertini, Lucas Li Bassi, Alessandro Frassineti and Filippo Papallo for the insightful and fruitful discussions and for their cheerful support. I am grateful to the Department of Mathematics of the University of Genoa - and, in particular, to the Algebraic Geometry research group - for providing a stimulating and friendly environment in which to carry out my research. 
The author is member of the INDAM-GNSAGA.

\numberwithin{equation}{section}

\setcounter{tocdepth}{1}

%% file: sec1.tex
\label{secpreliminaries}This Section is meant to collect some notions and results that will build the background for this work. In Section \ref{secsingularsymplectc} some basic material concerning primitive and irreducible symplectic varieties will be introduced, while Section \ref{ab} will be a quick overview of the theory of moduli spaces of sheaves on Abelian surfaces that will be needed throughout the paper. Eventually, in Section \ref{secdefomon}, a suitable deformation theory for primitive and irreducible symplectic varieties will be discussed, together with the definition and first properties of their locally trivial monodromy groups.

\subsection{Singular symplectic varieties}\label{secsingularsymplectc}
The first part of theory we are going to review is the one concerning symplectic varieties, with particular attention to those notions that provide the most meaningful generalization of the theory of irreducible holomorphic symplectic manifolds to the singular setting. For a more detailed discussion, we refer to \cite[Section 2, Section 3]{BL18} and \cite[Section 1.1]{OPR23}.

Let $X$ be a normal complex analytic variety, let $X_\reg$ be its smooth locus and let $j\colon X_\reg \to X$ be the corresponding open embedding. For any $0\leq p \leq \dim(X)$, we define the \textit{sheaf of reflexive $p-$forms} as $$\Omega_X^{[p]}:= j\sstar \Omega_{X_\reg}^p=(\wedge^p\Omega_X)^{\ast\ast}$$ and we call a global section $\sigma \in \H^0(X,\Omega_X^{[p]})$ a \textit{reflexive $p-$form} on $X$, which corresponds naturally to a holomorphic $p-$form on $X_\reg$. 

Throughout the paper we will also assume that $X$ admits a K\"{a}hler form $\omega$ and we will refer to \cite[Section 2.3]{BL18} and \cite[II, Section 2.1, Section 2.2]{Var89} for its precise definition and main properties. In the context of interest for this work, by \cite[Corollary 3.5]{BL18}, the torsion free part $\H^2(X,\Z)_\tf$ of $\H^2(X,\Z)$ will admit a pure weight $2$ Hodge structure, hence the class of a K\"{a}hler form $\omega$ will define an element $$[\omega] \in \H^{1,1}(X,\R)=\H^1(X,\Omega_X^{[1]})\cap \H^2(X,\R)$$ and the set of all classes obtained in this way will form an open cone in $\H^{1,1}(X,\R)$. A normal complex analytic variety $X$ admitting a K\"{a}hler form will be called a \textit{K\"{a}hler space}. If $X$ is smooth, by \cite[II, Proposition 1.3.1(ii)]{Var89}, it is a K\"{a}hler manifold with respect to the usual definition.

\begin{defn}\label{defsv}
    Let $X$ be a compact K\"{a}hler space. \begin{enumerate}
        \item A \textit{symplectic form} on $X$ is a closed reflexive $2-$form $\sigma \in \H^0(X,\Omega_X^{[2]})$ which is non-degenerate at each point of $X_\reg$.
        \item A \textit{symplectic variety} is a pair $(X,\sigma)$ where $\sigma$ is a symplectic form on $X$, such that for every resolution $\rho\colon \widetilde{X}\to X$, the form $\rho\Star(\sigma_{|X_\reg})$ extends to a holomorphic $2-$form on $X$. In particular, if the latter is a holomorphic symplectic form, the resolution $\rho$ is called a \textit{symplectic resolution}.
        \item A \textit{primitive symplectic variety} is a symplectic variety $(X,\sigma)$ such that \linebreak[4]$\H^1(X,\Ox)=0$ and $\H^0(X,\Omega_X^{[2]})=\C\sigma$.
        \item An \textit{irreducible symplectic variety} is a symplectic variety $(X,\sigma)$ such that, for every finite quasi-étale morphism $f\colon Y \to X$, the exterior algebra of reflexive forms on $Y$ is generated by $f^{[\ast]}\sigma$.
    \end{enumerate} With a slight abuse of notation, in the cases above we will refer to $X$ as a symplectic (respectively, primitive symplectic and irreducible symplectic) variety.
\end{defn}

\begin{rmk}\label{rmkPSV}
    We highlight some important features of primitive and irreducible symplectic varieties that will be useful in the following. \begin{enumerate}
        \item An irreducible symplectic variety is primitive symplectic, whilst the converse is, in general, false (see \cite[Example 1.5]{PR18}). For a more detailed discussion on the interplay of different notions of sympletic varieties, we refer to \cite{Per19}.
        \item An irreducible symplectic variety is smooth if and only if it is an irreducible holomorphic symplectic manifold. This is a consequence of Bogomolov Decomposition Theorem, together with the fact that, due to \cite[Corollary 13.3]{GGK19}, irreducible symplectic varieties are simply connected.
        \item If $X$ is a primitive symplectic variety, then $\H^2(X,\Z)_\tf$ admits a non-degenerate symmetric bilinear form $q_X$ of signature $(3,b_2(X)-3)$ (see \cite[Section 1.5, Lemma 5.7]{BL18}), which will be called \textit{Beauville-Bogomolov-Fujiki form} of $X$. We will refer to the pair $(\H^2(X,\Z)_\tf, q_X)$ as the \textit{Beauville-Bogomolov-Fujiki lattice} of $X$. This holds, in particular, for the free $\Z-$module $\H^2(X,\Z)$ in the case in which $X$ is an irreducible symplectic variety, by simple connectedness of the latter.
    \end{enumerate}
\end{rmk} The last remarkable property of symplectic varieties we include in this Subsection is related to their singular locus, which we recall is a closed subvariety of codimension at least $2$. 
\begin{prop}{(\cite[Theorem 2.3]{Kal06}, see also \cite[Theorem 3.4.(2)]{BL18})}\label{kalstratification}
    Let $X$ be a symplectic variety. Then there exists a finite stratification by closed subvarieties $$X=X_{0} \supseteq X_{1} \supseteq \cdots \supseteq X_l,$$ such that, for every $i=0,\dots m-1$, the stratum $X_{i+1}$ is the singular locus with reduced structure $(X_i^{sing})_{red}$ of $X_i$ and the normalisation of each irreducible component of $X_{i}$ is a symplectic variety.
\end{prop}
The stratification of Proposition \ref{kalstratification} will be called \textit{the stratification of singularities of} $X$ and its smaller stratum $X_l$ will be called \textit{the most singular locus $X^{ms}$ of} $X$.

In the next Subsection we will provide a wide class of examples of irreducible symplectic varieties, which will be the main focus of this work.
\subsection{Moduli spaces of sheaves on Abelian surfaces}\label{ab} In this Subsection we will provide a brief overview of the theory of moduli spaces of sheaves on Abelian surfaces that will be needed throughout the paper. For a deeper discussion, we refer to \cite{PR18} and \cite{PR20}.

Let $S$ be an Abelian surface and let us denote by $\Htilde(S,\Z)$ its \textit{Mukai lattice}, defined by the $\Z-$module $\H\ev(S,\Z):= \oplus_{i=0}^2\H^{2i}(S,\Z)$, equipped with the \textit{Mukai pairing}, i.e. the non-degenerate integral symmetric bilinear form $\cdot$ defined by $$(r_1,\xi_1,a_1)\cdot (r_2,\xi_2,a_2):= \xi_1\cdot \xi_2 - r_1 a_2 - r_2 a_1,$$ for any $(r_i,\xi_i,a_i)\in \H\ev(S,\Z)$, for $i=1,2$. Notice that the latter is isometric to $U^{\oplus 4}$, where $U$ is the rank $2$ unimodular hyperbolic lattice. Furthermore, the Mukai lattice of $S$ inherits a pure weight two Hodge structure from the one of $S$, by setting
\begin{align}
    \label{hodgestructure} \Htilde^{2,0}(S):= \H^{2,0}(S), \phantom{++}\Htilde^{0,2}(S):= \H^{0,2}(S),\\
    \Htilde^{1,1}(S) := \H^0(S,\C) \oplus \H^{1,1}(S) \oplus \H^4(S,\C).\nonumber
\end{align}
Among elements of type $(1,1)$ with respect to the Hodge decomposition (\ref{hodgestructure}), we define a \textit{Mukai vector} as an element $v=(r,\xi,a)\in \Htilde(S,\Z)$ such that $r\geq 0$ and $\xi \in \NS(S)$ and, if $r=0$ one of the two following conditions holds: \begin{enumerate}
    \item[(a)] $\xi$ is the first Chern class of a strictly effective divisor;
    \item[(b)] $\xi=0$ and $a>0$.
\end{enumerate}
For any Mukai vector $v\in \Htilde(S,\Z)$, there always exists a coherent sheaf $F$ on $S$ such that $v(F):= \ch(F) \cdot \sqrt{\td(S)}=v.$ The latter will be called the \textit{Mukai vector of $F$}. Notice that, as $\td(S)=(1,0,0)$, the Mukai vector of a coherent sheaf on an Abelian surface is nothing but its total Chern character.
\begin{defn}{(\cite[Section 2.1.2]{PR18})} \label{defvgen} Let $v=(r,\xi,a)\in \Htilde(S,\Z)$ be a Mukai vector. An ample line bundle - or polarization - $H$ on $S$ is called \textit{$v-$generic} if it satisfies one of the following conditions: \begin{enumerate}
    \item If $r>0$, then for every $\mu_H-$semistable sheaf $E$ such that $v(E)=v$ and for every $0\neq F \subseteq E$ such that $\mu_H(E)=\mu_H(F)$, the equality $\c_1(F)/\rk(F) = \c_1(E)/\rk(E)$ holds.
    \item If $r=0$, then for every $H-$semistable sheaf $E$ such that $v(E)=v$ and for every $0\neq F \subseteq E$ such that $\chi(E)/(\c_1(E)\cdot \c_1(H))= \chi(F)/(\c_1(F)\cdot \c_1(H))$, it holds $v(F)\in \Q v$.
\end{enumerate}
\end{defn}

Let $S$ be an Abelian surface, $v\in \Htilde(S,\Z)$ a Mukai vector such that $v^2>0$ and $H$ a $v-$generic polarization.  Let us denote by $M_v(S,H)$ the moduli space of Gieseker $H$-semistable sheaves on $S$ with Mukai vector $v$ and recall that, under these assumptions, $M_v(S,H)$ is a nonempty, irreducible, normal projective variety of dimension $2km^2+2$, whose smooth locus is the moduli space $M_v^s(S,H)$ parametrizing stable sheaves (see \cite[Theorem 4.4.]{KLS06}) and which is symplectic, by \cite{Mu84}.

\begin{rmk}\label{yoshfibrisotr} We quickly review the construction of the moduli space which will be the main focus of this work, following \cite[Section 2.2.1]{PR18}, \cite{Yos99a} and \cite{Yos01a}. Let us assume that $v^2>0$ and let us fix a coherent sheaf $F_0$ on $S$ such that $v(F_0)=v$. Let us denote by $\hat{S}=\Pic^0(S)$ the dual surface of $S$, let $\PP\in \Db(S\times \hat{S})$ be the Poincaré line bundle on $S \times \hat{S}$ and let $\FM_\PP\colon \Db(S) \to \Db(\hat{S})$ be the Fourier-Mukai transform with kernel $\PP$ (see Section \ref{secpoincaré}). Then the map \begin{align*}
    a_v \colon M_v(S,H) & \to S \times \hat{S}\\
    F & \mapsto (\det(\FM_\PP(F))\otimes \det(\FM_\PP(F_0))\dual, \det(F)\otimes \det(F_0)\dual)
\end{align*} is well defined - under the canonical isomorphism $\hat{\hat{S}}\simeq S$ - and an isotrivial fibration (see \cite[Lemma 2.15]{PR18}). We will denote by $\Kv(S,H)$ the fiber $a_v^{-1}(0_S,\Os)$ of $a_v$ over the $0-$point of the group $S \times \hat{S}$ and, with a slight abuse of notation, we will refer to the latter as the \textit{moduli space of sheaves on the surface $S$}. 

We point out that the isomorphism class of $\Kv(S,H)$ does not depend on the choice of the coherent sheaf $F_0$.
\end{rmk}

\begin{rmk}\label{rmkvgener}We will now compare Definition \ref{defvgen} with Definition 2.1 of \cite{PR18}. Let $S$ be an Abelian surface and let $v=\rxia$ be a Mukai vector. If $S$ has Picard rank $1$, then the ample generator is $v-$generic with respect to both the definitions, hence we can assume that $S$ has Picard rank at least $2$. In that case, if $r\neq 0$ or $r=0$ and $a\neq 0$, the Mukai vector $v$ induces a decomposition of the ample cone of $S$ in $v-$walls and open subcones called $v-$chambers (see \cite[Section 2.1.1]{PR18}). By \cite[Lemma 2.9]{PR18}, any polarization lying in a $v-$chamber - i.e. $v-$generic according to \cite[Definition 2.1]{PR18} - is $v-$generic according to Definition \ref{defvgen}. 

On the other hand, a $v-$generic polarization may lie on a $v-$wall, but Lemma 2.10 of \cite{PR18} ensures that, for any pair of polarizations $H, H'$ belonging to the closure of the same $v-$chamber, there are identifications of moduli spaces of sheaves $M_v(S,H)=M_v(S,H')$ and $K_v(S,H)=K_v(S,H')$ (see also Remark 2.13 of \cite{PR18}). The case $(0,\xi,0)$ is eventually included by \cite[Lemma 2.12]{PR18}. 
\end{rmk}

To keep notation concise, we introduce the following definition.

\begin{defn}\label{defmktriple}Let $m,k \in \N\setminus\{0\}$ be two strictly positive integers. A triple $\SvH$ will be called an \textit{$(m,k)-$triple} if $S$ is an Abelian surface, $v$ is a Mukai vector on $S$ of the form $v=mw$, where $w$ is a primitive Mukai vector such that $w^2=2k$ and $H$ is a $v-$generic polarization on $S$.
\end{defn}
\begin{rmk} Exactly as in \cite[Remark 1.22]{OPR23}, it can easily be checked that, if $\SvH$ is an $(m,k)-$triple with $v=mw$, then $\SwH$ is a $(1,k)-$triple.
\end{rmk}

For later use, we introduce an equivalence relation designed to formalize the identifications of moduli spaces mentioned in Remark \ref{rmkvgener}.

\begin{defn}\label{defncongruent}
    Two $(m,k)-$triples $\SvHuno$ and $\SvHdue$ are said to be \textit{congruent} if $S_1=S_2=:S$, $v_1=v_2=:v$ and, for every coherent sheaf $F$ on $S$ such that $v(F)=v$, we have that $F$ is $H_1-$semistable if and only if $F$ is $H_2-$semistable.
\end{defn}

The last property yields identifications $M_v(S,H_1) = M_v(S,H_2)$ and $K_v(S,H_1) = K_v(S,H_2)$ (see \cite[Lemma 2.16]{PR18}) and we denote by \begin{equation} \label{idcongruent}
    \chi_{H_1,H_2}\colon \Kv(S,H_1) \longrightarrow \Kv(S,H_2)
\end{equation}  the identity morphism.

\begin{rmk}It is straightforward form Remark \ref{rmkvgener} that, if $H_1$ and $H_2$ are two $v-$generic polarizations lying in the closure of the same $v-$chamber, then the triples $(S,v,H_1)$ and $(S,v,H_2)$ are congruent.
\end{rmk}

The following is one of the two foundational results that lay the groundwork for this paper.

\begin{thm}{(\cite[Theorem 1.10]{PR18})} \label{pr18thm1.10} Let $m,k \in \N\setminus \{0\}$ and let $\SvH$ be an $(m,k)-$triple. If $(m,k)\neq (1,1)$, then $\Kv(S,H)$ is an irreducible symplectic variety of dimension $2km^{2}-2$.
\end{thm}
\begin{rmk}\label{rmksmoothcase} We give a quick overview of the classification - up to deformation (see Section \ref{secdefo}) - of the moduli spaces $\Kv(S,H)$ according to the values of the pair $(m,k)$, referring to \cite[Section 1.1]{PR18} and \cite[Section 1.1]{PR20} for a more detailed discussion, including complete references. 
\begin{table}[H]
    \centering \small
    \begin{tabularx}{\textwidth}{cccc}
        \hline
        $m$ & $k$ & $\Kv(S,H)$  & References \\
        \hline
        $\geq 1$ & $\leq  0$ & empty or a PSV not ISV & \cite{Yos01a}, \cite{Mu84}, \cite{KLS06}\\
        $1$ & $1$ &  point & \cite{Yos01a}, \cite{KLS06}\\
        $1$ & $2$ &  K3 surface & \cite{Yos99b}\\
        $1$ & $\geq 3$ & $Kum^{k-1}(S)$ & \cite{Mu84}, \cite{Bea83}, \cite{Yos01a} \\
        $2$ & $1$ & ISV with symplectic resolution $OG6$ & \cite{PR10}, \cite{OG00}\\
        $\geq 2$ & $\geq 1$; $> 1$ if $m=2$ & ISV without symplectic resolution & \cite{PR18}, \cite{KLS06}\\
        \hline
    \end{tabularx}
\end{table} 
In particular, we will be interested only in $(m,k)-$triples with $k>0$. We furthermore point out that $\Kv(S,H)$ is smooth if and only if $m=1$, in which case $\Kv(S,H)$ is an irreducible holomorphic symplectic manifold.
\end{rmk}
The second fundamental result concerning moduli spaces of sheaves provides an explicit description of the lattice structure of their second integral cohomology groups, as introduced in Remark \ref{rmkPSV} (3).
\begin{thm}{(\cite[Theorem 1.6, Lemma 5.1 and Proposition 6.2]{PR20})} \label{pr20thm1.6} Let $m,k \in \N\setminus \{0\}$ and let $\SvH$ be an $(m,k)-$triple. If $(m,k)\neq (1,1)$, then there exists an injective Hodge isometry
    $$\lambda_{\SvH}\colon v\ort \longrightarrow \H^{2}(\Kv(S,H)),$$ where the lattice and Hodge structures of $v\ort$ are inherited from those of $\Htilde(S,\Z)$. Moreover, if $(m,k)\neq (1,2)$ the morphism $\lambda_{\SvH}$ is an isomorphism of $\Z-$modules.
\end{thm}
\subsection{Locally trivial deformations and monodromy operators}\label{secdefomon} In this Subsection we collect some notions and results that will build the framework for a deformation theory for primitive and irreducible symplectic varieties, in the case of our interest. This will allow us to introduce locally trivial monodromy operators and study their fundamental properties.
\subsubsection{Locally trivial deformations of symplectic varieties}\label{secdefo} We start by introducing an efficient notion of deformation, designed to preserve the essential properties of a variety, in a context in which we allow singularities. For further details, we refer to \cite[Section 4]{BL18} and \cite[Section 1.2]{OPR23}.
\begin{defn}\label{defltfamily}\begin{enumerate}
    \item A \textit{locally trivial family} is a proper morphism $f\colon \XX \to T$ of complex analytic spaces such that the base $T$ is connected and, for every point $x\in \XX$, there exist open neighborhoods $V_x\subseteq \XX$ of $x$ and $V_{f(x)}\subseteq T$ of $f(x)$ and an open subset $U_x \subseteq f^{-1}(f(x))$ such that $$V_x \simeq U_x \times V_{f(x)},$$ where the isomorphism is an isomorphism of complex analytic spaces commuting with the projections over $T$.
    \item If $X$ is a complex analytic variety, a \textit{locally trivial deformation of $X$} is a locally trivial family $f\colon \XX \to T$ for which there is $t\in T$ such that $\XX_t:= f^{-1}(t) \simeq X$.
    \item A \textit{locally trivial family of primitive (resp. irreducible) symplectic varieties} is a locally trivial family whose fibers are all primitive (resp. irreducible) symplectic varieties.
    \item Two primitive (resp. irreducible) symplectic varieties are said to be \textit{locally trivial deformation equivalent} if there exists a locally trivial family of primitive (resp. irreducible) symplectic varieties having both of them as fibers.
\end{enumerate}
\end{defn}

In the following, given a complex analytic space $X$, we will deal only with \textit{small} locally trivial deformations of $X$, i.e. locally trivial families with base an analytic open neighborhood $U$ of $X$ in the base of a universal deformation of $X$.

    By \cite[Corollary 4.11]{BL18}, any small locally trivial deformation of a primitive symplectic variety is a locally trivial deformation of primitive symplectic varieties, according to Definition \ref{defltfamily} (3). In the case of irreducible symplectic varieties, some sufficient criteria are provided in \cite[Section 1.2]{OPR23} in the case of our interest. The following is a straightforward application of Proposition 1.7, Proposition 1.8 and Proposition 1.9 of \cite{OPR23}.

    \begin{prop}
        Let $m$ and $k$ be two strictly positive integers and let $\SvH$ be an $(m,k)-$triple. If $(m,k)\neq (2,1)$, any small locally trivial deformation $f\colon \XX \to T$ of $\Kv(S,H)$ is an irreducible symplectic variety. If $(m,k)=(2,1)$, the same holds, provided that $f\colon \XX \to T$ is projective and $T$ is quasi-projective.
    \end{prop}

    \begin{proof}
        The first part of the statement immediately follows from Proposition 1.7 or Proposition 1.8 of \cite{OPR23}, as, under the assumption $(m,k)\neq (2,1)$, the moduli space $\Kv(S,H)$ admits at most terminal singularities, by \cite[Theorem A]{KLS06} and \cite[Corollary 1]{Nam01}, and its smooth locus $\Kv^s(S,H)$ is simply connected, by \cite[Theorem 3.6]{PR18}. If $(m,k)=(2,1)$, again by \cite[Theorem 3.6]{PR18}, it holds $\pi_1(\Kv^s(S,H))=\Z/2\Z$ and, by \cite[Theorem 1.4]{OG00}, \cite[Corollary 1]{Nam01} and \cite[Theorem 1.1]{Per09}, the moduli space $\Kv(S,H)$ has canonical singularities that are not terminal. Hence, the only criterion that can be applied in this case is Proposition 1.9 of \cite{OPR23}, for which projectivity is a requirement.
    \end{proof}
Consistently with the theory above-introduced and with a view towards bimeromorphic classification of irreducible symplectic varieties (see \cite[Theorem 6.16]{BL18}), a natural starting point is the study of the locally trivial deformation equivalence class (see Definition \ref{defltfamily} (4)) of a moduli space of sheaves on an Abelian surface, consisting of irreducible symplectic varieties. More details concerning such deformation classes are provided, for instance, by Theorem \ref{PR18defolt}.

Furthermore, as smooth deformations are locally trivial, the notions introduced so far define a consistent generalization of the classical deformation theory for irreducible holomorphic symplectic manifolds.\\

We conclude this Subsection by dealing with the behavior of the singular locus of a symplectic varieties along a locally trivial deformation, as it will be a crucial point in Section \ref{secemb}.

\begin{rmk}\label{relstrat} \begin{enumerate}
    \item We point out that, if $f \colon \XX \to T$ is a locally trivial family of symplectic varieties, then there is a natural relative stratification $$\XX= \XX_0 \supseteq \XX_1 \supseteq \cdots \supseteq \XX_l,$$ such that, for every $t\in T$, set $\XX_{i,t}:= \XX_i \cap \XX_t$ for $i=0, \dots, l$, the stratification $$\XX_t \supseteq \XX_{1,t} \supseteq \cdots \supseteq \XX_{l,t}$$ is the stratification of singularities of $\XX_t$ as in Proposition \ref{kalstratification} and the restriction $f_i\colon \XX_i \to T$ is again a locally trivial family. 
    \item We also remark that, by construction, the restriction $f_l\colon \XX_l\to T$ of $f$ defines a smooth deformation of smooth symplectic manifolds. Since connected components are preserved along locally trivial deformations, the restriction $f_l$ to each connected component of $\XX_l$ defines a smooth deformation of irreducible holomorphic symplectic manifolds.
\end{enumerate}
\end{rmk}

\subsubsection{Locally trivial monodromy operators} \label{secmon} Aim of this Subsection is to define the locally trivial monodromy group of a primitive symplectic variety, following \cite{BL18} and review its fundamental properties (see \cite[Section 1.3]{OPR23}).

As recalled in Remark \ref{rmkPSV} (3), if $X$ is a primitive symplectic variety, then its Beauville-Bogomolov-Fujiki lattice $(\H^2(X,\Z)_\tf,q_X)$ is an indefinite lattice of signature $(3,b_2(X)-3)$. By \cite[Lemma 5.7]{BL18}, the latter is a locally trivial deformation invariant, meaning that, if $X_1$ and $X_2$ are two locally trivial deformation equivalent primitive symplectic varieties, then there exists an isometry $$g\in \O(\H^2(X_1,\Z)_\tf, \H^2(X_2,\Z)_\tf)$$ (see Appendix \ref{appendix} for notations). Among isometries defined in this way, we identify the following special class:

\begin{defn}\label{defpt}Let $X$, $X_1$ and $X_2$ be three primitive symplectic varieties. \begin{enumerate}
    \item An isometry $g \in \O(\H^2(X_1,\Z)_\tf, \H^2(X_2,\Z)_\tf)$ is a \textit{locally trivial parallel transport operator from $X_1$ to $X_2$} if there exist a locally trivial family of primitive symplectic varieties $p\colon \XX \to T$, two points $t_1,t_2\in T$ such that $\XX_{t_i}\simeq X_i$, for $i=1,2$, and a continuous path $\gamma$ in $T$ from $t_1$ to $t_2$ such that $g$ is the parallel transport $\PT_p(\gamma)$ along $\gamma$ in the local system $R^2p\sstar \Z$.
    \item An isometry $g\in \O(\H^2(X,\Z)_\tf)$ is a \textit{locally trivial monodromy operator on $X$} if it is a locally trivial parallel transport operator from $X$ to itself.
\end{enumerate}
\end{defn}

If $X_1$ and $X_2$ are two primitive symplectic varieties, we will denote by $$\PT_{\lt}^2(X_1,X_2)\subseteq \O(\H^2(X_1,\Z)_\tf, \H^2(X_2,\Z)_\tf)$$ the set of locally trivial parallel transport operators from $X_1$ to $X_2$. If $X_1=X_2=:X$, we will denote by \begin{equation}
    \label{mon}\monlt(X):=\PT_{\lt}^2(X,X)\subseteq \O(\H^2(X,\Z)_\tf)
\end{equation} the set of locally trivial monodromy operators of $X$. 
\begin{rmk}\label{rmkcomposition}
    As already remarked in \cite[Lemma 1.12]{OPR23}, arguing as in the smooth case in \cite[footnote 3]{Mar11}, one can easily prove that the set $\monlt(X)$ of locally trivial monodromy operators of a primitive symplectic variety $X$ is actually a subgroup of $\O(\H^2(X,\Z)_{tf})$. 
    
    Indeed, if $f, g \in \monlt(X)$, there exist two locally trivial families $p\colon \XX \to T$ and $p'\colon \XX'\to T'$, two points $t\in T$ and $t'\in T'$ such that $\XX_t\simeq X \simeq \XX'_{t'}$ and two loops $\gamma$ in $T$ and $\gamma'$ in $T'$ centered, respectively, in $t$ and $t'$, such that $f$ is the locally trivial monodromy operator in the family $p$ associated to $\gamma$ and $g$ is the locally trivial monodromy operator in the family $p'$ associated to $\gamma'$. We can therefore define a new locally trivial deformation $p''\colon \XX'' \to T''$ of $X$ by gluing $\XX$ and $\XX'$ via the isomorphism $\XX_t\simeq X \simeq \XX'_{t'}$ and by gluing $T$ and $T'$ by means of the relation $t=t'=:t''$. The loop $\gamma'':= \gamma\ast \gamma'$ in the new reducible base is centered in $t''$, by construction, and we define $g\circ f$ as the locally trivial monodromy operator associated to $\gamma''$ in the family $p''$.
\end{rmk}
\begin{defn}\label{defmon}Let $X$ be a primitive symplectic variety. The group $\monlt(X)$ defined in (\ref{mon}) is called the \textit{locally trivial monodromy group of $X$}.
\end{defn}
\begin{rmk}
    If $X$ is smooth, then $\monlt(X)=\mon^2(X)$, according to the classical definition (see \cite[Definition 1.1]{Mar11}), as smooth deformations of $X$ are locally trivial.
\end{rmk}

In order to refine the inclusion of groups of Remark \ref{rmkcomposition}, we introduce an \textit{orientation} on $\H^2(X,\Z)_{tf}$, when $X$ is a primitive symplectic variety, referring the reader to Appendix \ref{rmkorient} for general definitions and basic facts, or to \cite[Section 4]{Mar11} and \cite[Section 4.1]{Mar08} for a more detailed discussion.

\begin{rmk}\label{exorientPSV} We recall that, if $X$ is a primitive symplectic variety, the Beauville-Bogomolov-Fujiki lattice $(\H^2(X,\Z)_{tf},q_X)$ has signature $(3,b_2(X)-3)$ (see Remark \ref{rmkPSV} (3)). Hence, in order to find an orientation for the big positive cone $$\widetilde{\CC_X}:=\{\alpha \in \H^2(X,\R) \colon q_X(\alpha)>0\}$$ of $X$, we may choose a basis of a $3-$dimensional positive space.  As explained in the proof of \cite[Lemma 1.13]{OPR23}, following \cite[Corollary 8]{Nam02a}, if $\omega \in \H^1(X,\Omega_X^{[1]})$ is a K\"ahler class and $\sigma \in \H^0(X,\Omega_X^{[2]})$ is a reflexive $2-$form, then the basis $\{\omega, \Re(\sigma), \\ \Im(\sigma)\}$ determines an orientation of $\H^2(X,\Z)_{tf}$. Moreover, such an orientation does not depend on the choice of the K\"ahler class and of the symplectic form and it is preserved by locally trivial parallel transport operators in families of primitive symplectic varieties.
\end{rmk}

Remark \ref{exorientPSV} is one of the key ingredients needed to prove the following

\begin{lem}{(\cite[Theorem 1.1 (1)]{BL18}, \cite[Lemma 1.13]{OPR23})}\label{lemmonltorpres}
    Let $X$ be a primitive symplectic variety. Then $\monlt(X)$ is a subgroup of finite order of $\O^+(\H^2(X,\Z)_{tf})$.
\end{lem}
For later use, we make few comments that will be helpful in order to give an efficient characterization of orientation preserving Hodge isometries.
\begin{rmk}\label{exorientPSV2}
    If $X$ and $Y$ are two primitive symplectic varieties and $$g\colon \H^2(X,\Z)_{tf} \to \H^2(Y,\Z)_{tf}$$ is a Hodge isometry, we may choose generators $\sigma_X \in \H^0(X,\Omega_X^{[2]})$ and $\sigma_Y \in \H^0(Y,\Omega_Y^{[2]})$ such that $g(\sigma_X)=\sigma_Y$. If $\omega_X$ is a K\"ahler class on $X$, as $g$ is an isometry, its image $g(\omega_X)$ must be orthogonal to $\sigma_Y$, namely, $g(\omega_X)\in \H^1(X,\Omega_Y^{[1]})$. Hence, we can choose two natural orientations for $X$ and $Y$, given, respectively, by $\{\omega_X, \Re(\sigma_X), \Im(\sigma_X)\}$ and $\{\omega_Y, \Re(\sigma_Y), \Im(\sigma_Y)\}$, where $\omega_Y$ is a K\"ahler class on $Y$, and deduce that $g$ preserves the given orientations if and only if the orientation of the positive cone $$\CC'_X=\{\alpha \in \H^{1,1}(X,\R)\colon \alpha\cdot \alpha>0\}$$ is preserved. According to Appendix \ref{rmkorient} (a) this happens if and only if $g$ maps the connected component $\CC_X$ of $\CC'_X$ containing $\omega_X$ to the connected component $\CC_Y$ of $\CC'_Y$ containing $\omega_Y$. 
\end{rmk}

An example of locally trivial parallel transport operator is given by pushforwards of birationalities between projective primitive symplectic varieties. In the case of irreducible holomorphic symplectic manifolds, this is the content of \cite[Theorem 2.5, Corollary 2.7]{Huy03a}, made explicit in terms of parallel transport operators by \cite[Section 3.1]{Mar11}. A generalization to the singular setting is given by a non-trivial result by \cite{BL18}.

\begin{prop}{(\cite[Theorem 6.16]{BL18})}\label{proppushforward} Let $X$ and $Y$ be two projective primitive symplectic varieties, and let $f\colon X \to Y$ be a birational morphism. Suppose that $f$ is an isomorphism in codimension $1$ and that
$f\sstar : \Pic(X)_\Q \to \Pic(Y)_\Q$ is well-defined and an isomorphism. Then the pushforward $$f\sstar\colon \H^2(X,\Z)_{tf} \longrightarrow \H^2(Y,\Z)_{tf}$$ is well defined and a locally trivial parallel transport operator.\end{prop}

In order to complete the picture and to provide additional motivation for the study of the locally trivial monodromy group of a primitive - respectively, irreducible - symplectic variety, we include the Hodge-theoretic version of Global Torelli Theorem for primitive symplectic varieties, which is the main result of \cite{BL18}.

\begin{thm}{(\cite[Theorem 1.1]{BL18})}\label{thmtorelli} Let $X$ and $Y$ be two locally trivial deformation equivalent primitive symplectic varieties and assume that $b_2(X)\geq 5$. Then $X$ and $Y$ are bimeromorphic if and only if there exists a Hodge isometry $g\colon \H^2(X,\Z)_\tf \to \H^2(Y,\Z)_\tf$ that is a locally trivial parallel transport operator.\end{thm}

In light of Theorem \ref{thmtorelli}, in order to perform a bimeromorphic classification of a primitive (or irreducible) symplectic variety $X$ whose locally trivial deformation class is fixed, the study of its locally trivial monodromy group $\monlt(X)$ becomes crucial. The index $[\O^+(\H^2(X,\Z)_\tf)\colon \monlt(X)]$ carries already an important geometric meaning: it counts the number of distinct bimeromorphism classes corresponding to each Hodge-isometry class, in the fixed deformation class.

\subsubsection{The monodromy group of smooth moduli spaces of sheaves on Abelian surfaces}\label{secmonkummer}We conclude the Section with a brief description of the monodromy group of moduli spaces of sheaves on Abelian surfaces whose Mukai vector is primitive and whose square is strictly greater than $4$. In this case, the latter is smooth and it is an irreducible holomorphic symplectic manifold deformation equivalent to a Kummer manifold (see Remark \ref{rmksmoothcase}).

Let $\SwH$ a $(1,k)-$triple, with $k>2$ and let $X=\Kw(S,H)$. Recall that, by \cite{Yos01a} (see also Theorem \ref{pr20thm1.6}), the Beauville-Bogomolov-Fujiki lattice $\H^2(X,\Z)$ of $X$ is abstractly isometric to the even lattice $U^{\oplus 3}\oplus <-2k>$ of signature $(3,4)$. In Appendix \ref{weylgroups} is recalled the construction of the Weyl group $\W(\H^2(X,\Z))$ generated by reflections around $(\pm 2)-$vectors in $\H^2(X,\Z)$, which we will denote, for simplicity, by $\W(X)$.

\begin{ex}\label{exmonk3}
    The group $\W(X)$ plays already a remarkable role in the context of moduli spaces of sheaves on K3 surfaces, both smooth and singular. Indeed, if $\SvH$ is an $(m,k)-$triple, with $S$ a K3 surface and $k>1$, and we let $X=M_v(S,H)$ be the moduli space of Gieseker $H-$semistable sheaves on $S$ with Mukai vector $v$, then one can analogously define the group $\W(X)$ and the following equality holds: $$\monlt(X)=\W(X).$$ The latter is the principal result of \cite{Mar08} and \cite{Mar10} in the primitive case and of \cite{OPR23} in the non-primitive case. 
\end{ex}

In the primitive Abelian case, for $X=\Kw(S,H)$, the same role is played by the order two subgroup $\Nn(X)$ of $\W(X)$ given by the kernel of the character $\det\cdot \disc$ (see (\ref{defNn})).

\begin{thm}{(\cite[Theorem 1.4]{Mar18}, \cite[Theorem 2.3]{Mon14})}\label{thmmonkummer} Let $X$ be an irreducible holomorphic symplectic manifold deformation equivalent to a moduli space $\Kw(S,H)$, with $\SwH$ a $(1,k)-$triple, with $k>2$. Then $$\mon^2(X)= \Nn(X).$$
\end{thm}
\begin{rmk}\label{indexN}
    By \cite[Lemma 4.1]{Mar10}, the index of $\Nn(X)$ inside $\O^+(\H^2(X,\Z))$ is $2^{\rho(k)}$, where $\rho(k)$ is the Euler number of $k$, counting the number of distinct primes occurring in its factorization. We deduce that, for irreducible holomorphic symplectic manifolds \textit{of Kummer type}, the monodromy group is never maximal.   
\end{rmk}

%% file: sec2.tex
\label{secinjectivemorph}In this section we relate locally trivial monodromy operators on a singular moduli space $\Kv(S,H)$ of semistable sheaves on an Abelian surface to monodromy operators on a smooth moduli space $\Kw(S,H)$ of the same kind. The key point is the study of the most singular locus of $\Kv(S,H)$ and the action on the second integral cohomology group of its closed embedding, that will be done, respectively, in Sections \ref{secemb} and \ref{secicohomology}. The latter will be used in Section \ref{secinj} to produce an injective morphism on the respective locally trivial monodromy groups, which will allow us to include $\monlt(\Kv(S,H))$ in $\mon^2(\Kw(S,H))$ as a subgroup.

\subsection{The embedding \texorpdfstring{$\Kw(S,H) \to \Kv(S,H)$}{Kw(S,H) → Kv(S,H)}}\label{secemb} Let $m$ and $k$ be two positive integers and let $\SvH$ be an $(m,k)-$triple, in the sense of Definition \ref{defmktriple}. We recall that, by Theorem \ref{pr18thm1.10}, the moduli space $\Kv(S,H)$ is an irreducible symplectic variety and, by Proposition \ref{kalstratification}, it admits a stratification of singularities $$\Kv(S,H) =X_0 \supset X_1\supset\cdots \supset X_l=\Kv(S,H)^{ms}.$$ Aim of this Subsection is to show that the most singular locus $\Kv(S,H)^{ms}$ admits a finite number of connected components, each of which is isomorphic to the associated smooth moduli space $\Kw(S,H)$, allowing us to define a natural closed embedding $$i_{w,m}\colon \Kw(S,H) \longrightarrow \Kv(S,H)$$ that deform in locally trivial families.\\

Let us start by considering the moduli space $M_v(S,H)$ of $H$-semistable sheaves on $S$ with Mukai vector $v$. As in \cite[Section 1.4.]{OPR23}, we recall the main ideas of the proof of \cite[Theorem 4.4.]{KLS06} that provide a description of the most singular locus $M_v(S,H)^{ms}$ of $M_v(S,H)$. The singular locus $M_v(S,H)^{sing}$ of $M_v(S,H)$ coincides with the variety parametrizing strictly semistable sheaves. Each point of $M_v(S,H)^{sing}$ then corresponds to the class $[F]$ of a semistable sheaf $F$ that, admitting a non-trivial Jordan-Hölder decomposition, is $S-$equivalent to a sheaf of the form $F_{1}\oplus F_{2}$, where $[F_{i}]\in M_{m_iw}(S,H)$ and $m_1+m_2=m$ (see \cite[Section 1.5]{HL97}). More precisely, it belongs to the image of the summation morphism \begin{align}\label{sum}
    \sigma_{m_1,m_2}\colon M_{m_1w}(S,H) \times M_{m_2w}(S,H) & \longrightarrow M_v(S,H) \nonumber\\
    ([F_1],[F_2]) & \longmapsto [F_1\oplus F_2],
\end{align} which is an irreducible component of the singular locus. The intersection of all such components is the locus $$\Sigma_m:= \{[E_1 \oplus\cdots \oplus E_m] \colon [E_i] \in M_w(S,H)\} \simeq \Sym^m(M_w(S,H)),$$ whose most singular locus - and hence, the most singular locus of $M_v(S,H)$ - is \begin{equation}\label{mvsingmax}
    Y := \{[E^{\oplus m}] \in M_v(S,H) \colon [E] \in M_w(S,H)\},
\end{equation}  which is naturally isomorphic to $M_w(S,H)$.\\

We now check the compatibility of this description with the fibers of the respective Yoshioka fibrations involved.

\begin{prop}\label{propKvms}Let $m$ and $k$ be two positive integers and let $\SvH$ be an $(m,k)-$triple. Then $$\Kv(S,H)^{ms}= M_v(S,H)^{ms}\cap \Kv(S,H).$$
\end{prop}
\begin{proof}
As recalled in Remark \ref{yoshfibrisotr}, the Yoshioka fibration $a_v\colon M_v(S,H) \longrightarrow S \times \hat{S}$ is an isotrivial fibration. Therefore, by Grauer-Fischer Theorem (\cite{FG65}), it is also an analytically locally trivial family, in the sense of Definition \ref{defltfamily}. Hence, by Remark \ref{relstrat}, it induces a relative stratification $$M_v(S,H)= \XX_0 \supseteq \XX_1 \supseteq \cdots \supseteq \XX_l= M_v(S,H)^{ms},$$ such that, for every $\alpha \in S\times \hat{S}$, $$a_v^{-1}(\alpha)= \XX_{0,\alpha} \supseteq \cdots \supseteq \XX_{l,\alpha}$$ is a stratification of singularities of $a_v^{-1}(\alpha)$. In particular, for $\alpha=(0,\Os)$, we get a stratification of singularities of $\Kv(S,H)$, whose last stratum $\XX_{l,0}= \Kv(S,H)^{ms}$ is precisely, by definition, \[\XX_l \cap a_v^{-1}(O,\Os)= M_v(S,H)^{ms}\cap \Kv(S,H).  \qedhere\]
\end{proof}
Proposition \ref{propKvms} yields a more explicit description of the most singular locus of $\Kv(S,H)$. In the following, for every positive integer $m\in \N\setminus\{0\}$, we will denote by $S[m] \times \hat{S}[m]$ the $m-$torsion points in $S\times\hat{S}$.
\begin{prop}\label{kvms}
    Let $m$ and $k$ be two positive integers and let $\SvH$ be an $(m,k)-$triple. Then $$\Kv(S,H)^{ms}\simeq \bigcup_{(p,L)\in S[m]\times \hat{S}[m]} a_w^{-1}(p,L).$$ 
\end{prop}
\begin{proof} From Proposition \ref{propKvms} and characterization (\ref{mvsingmax}), we get $$\Kv(S,H)^{ms}=\{[E^{\oplus m}] \in K_v(S,H) \colon [E] \in M_w(S,H)\}.$$ A straightforward computation shows that the summation map (\ref{sum}) is compatible with both the Yoshioka fibrations $a_{m_{i}w}$ for $i=1,2$ and, in particular, for every $E \in M_w(S,H)$, it holds $a_v(E^{\oplus m})=a_w(E)^{m}$, with respect to the group structure on $S\times \hat{S}$. It follows that, given $[E] \in M_w(S,H)$, the point $[E^{\oplus m}]$ belongs to $\Kv(S,H)^{ms}$ if and only if $$a_w(E)^{m}=(0,\Os) \in S\times \hat{S},$$ from which the claim follows. 
\end{proof}
\begin{cor}\label{corembedding}Let $m$ and $k$ be two positive integers and let $\SvH$ be an $(m,k)-$triple. For every $(p,L)\in S[m]\times \hat{S}[m]$, there is a closed embedding \begin{align*}
    \tau_{(p,L)}^m \colon \Kw(S,H) &\longrightarrow \Kv(S,H)^{ms}\\
    E & \longmapsto (\tau_{p,\ast}(E) \otimes L)^{\oplus m},
\end{align*} where $\tau_p\colon S \longrightarrow S$ is the automorphism given by translation by the point $p\in S$.
\end{cor}
\begin{proof}
    The claim follows from Proposition \ref{kvms} and isotriviality of the Yoshioka fibration $a_w$. In particular, the isomorphism $E \mapsto \tau_{p,\ast}(E) \otimes L$ naturally identifies the fibers $K_w(S,H)$ and $a_w^{-1}(p,L)$ for every $(p,L)\in S\times \hat{S}$ (see \cite[Section 2.2.1]{PR18} and \cite[Section 4.2]{Yos01a}).
\end{proof}
From finiteness of torsion points on Abelian surfaces, we deduce that we have a finite number of choices - namely, $m^{8}$ - to embed $\Kw(S,H)$ in $\Kv(S,H)$ as one of the connected components of its most singular locus.
\begin{rmk}We wish to point out that, as a consequence of Remark \ref{relstrat}, the closed embedding in Corollary \ref{corembedding} behaves well in locally trivial families. This means that, if $f\colon \XX \to T$ is a locally trivial deformation of a moduli space $\Kv(S,H)$ as above and $\tau^m_{(p,L)}\colon \Kw(S,H) \to \Kv(S,H)$ is a closed embedding as in Corollary \ref{corembedding}, the latter induces a relative closed embedding $\iota \colon \ZZ \hookrightarrow \XX$ of analytic varieties such that $f_{|\ZZ} \colon \ZZ \to T$ is a smooth deformation of $\Kw(S,H)$.
    
\end{rmk}

\subsection{Action on the second integral cohomology} \label{secicohomology} We now focus on the action of the closed embeddings defined in the previous Subsection on the second integral cohomology, in order to relate integral isometries - and, in particular, monodromy operators - defined on $\Kv(S,H)$ and on $\Kw(S,H)$. 

Let $m,k$ be two strictly positive integers, let $\SvH$ be an $(m,k)-$triple and set $i_{w,m}:= \tau_{(0,\Os)}^m$, where the latter is one of the closed embeddings defined in Corollary \ref{corembedding}. We now compare its pullback $$i_{w,m}\Star\colon \H^{2}(\Kv(S,H),\Z) \longrightarrow \H^{2}(\Kw(S,H),\Z)$$ with the composition \begin{center}
    \begin{tikzcd}
        \H^{2}(\Kv(S,H),\Z) \arrow[r, "\lambda_{\SvH}^{-1}"] & v\ort \arrow[r,equal] & w \ort \arrow[r,"\lambda_{\SwH}"] & \H^{2}(\Kw(S,H),\Z),
    \end{tikzcd}
\end{center}where the arrows on the sides are given by the isometries of Theorem \ref{pr20thm1.6}.  
\begin{prop}\label{istar}
    Let $m,k$ be two positive integers, with $k>1$, and let $\SvH$ be an $(m,k)-$triple. The morphism $i_{w,m}\Star\colon \H^{2}(\Kv(S,H),\Z) \to \H^{2}(\Kw(S,H),\Z)$ is a similitude of lattices satisfying $$i\Star_{w,m}= m\lambda_{\SwH} \circ \lambda_{\SvH}^{-1}.$$
\end{prop}
\begin{proof}
The proof follows as in Proposition 1.28 of \cite{OPR23}, assuming that the surface $S$ is Abelian and replacing the moduli space $M_v(S,H)$ with $\Kv(S,H)$, under the current hypothesis on the $(m,k)-$triple $\SvH$. We sketch the main ideas and we provide the references needed to approach this specific case. 

For every $p>0$, we will use the shortened notation $K_{pw}$ for $K_{pw}(S,H)$ and $\lambda_{pw}$ for $\lambda_{(S,pw,H)}$. By \cite[Section 4.20]{PR20} and, again, by naturality of Yoshioka fibrations with respect to direct sums, for every $p>0$ we have a morphism \begin{align*}
    f_p\colon K_w^p & \longrightarrow K_{pw}\\
    (F_1,\dots, F_p) & \longmapsto F_1 \oplus \cdots \oplus F_p,
\end{align*} that fits in the following diagram: \begin{equation}
        \label{cd1} \begin{tikzcd}[column sep=large]
            (pw)\ort = w\ort \arrow[d,"(\lambda_{w}\text{,}\dots \text{,}\lambda_{w})" swap]\arrow[r,"\lambda_{pw}"] & \H^{2}(K_{pw},\Z) \arrow[d,"f_{p}^{\ast}"]\\
            \oplus_{i=1}^{p}\H^{2}(K_{pw},\Z)\arrow[r,"\sum_{i=1}^{p}\pi_{i,p}\Star" swap] & \H^{2}(K_{w}^{p},\Z),
        \end{tikzcd} 
    \end{equation} where $\pi_{i,p}\colon \Kw^{p}\longrightarrow \Kw$ is the projection onto the $i-$th factor. Proceeding by induction on $p$ as in \cite[Proposition 1.28]{OPR23} and applying Proposition 4.12 (2) of \cite{PR20}, we get commutativity of diagram (\ref{cd1}), which yields, for $p=m$, to the identity \begin{equation}
        \label{flambda}f_{m}^{\ast}(\lambda_{v}(a))=\sum_{i=1}^{m}\pi_{i,m}\Star(\lambda_{w}(a),\dots,\lambda_{w}(a)),
    \end{equation} for every $a \in v\ort=w\ort$. We now consider the diagonal morphism $\delta_m \colon \Kw \longrightarrow \Kw^{m}$ and we notice that, by definition, the equality $i_{w,m}=f_m\circ \delta_m$ holds. Therefore, by (\ref{flambda}) we deduce, for every $a \in v\ort$, \begin{align*}
        (i_{w,m}\Star\circ \lambda_{v})(a) &= (\delta_m\Star \circ f^{\ast}_{m} \circ \lambda_{v})(a)=\delta_m\Star(\sum_{i=1}^{m}\pi_{i,m}\Star(\lambda_{w}(a),\dots,\lambda_{w}(a)))=\\
        &=\sum_{i=1}^{m}\lambda_{w}(a)= m\lambda_{w}(a),
    \end{align*} concluding the proof.
\end{proof}
In the following, we will use this notation: if $f\colon A \longrightarrow B$ is an isomorphism of groups, we will denote by $f^\sharp$ the map induced on the corresponding $\Hom$ sets by conjugation, i.e. \begin{align}\label{notconj}
    f^\sharp \colon \Hom(A,A) & \longrightarrow \Hom(B,B)\\
    g & \longmapsto f \circ g \circ f^{-1}. \nonumber
\end{align} From now on, we will work under the additional hypothesis that $k>2$, for reasons that will be made clear in Section \ref{seck12}. 

Studying the conjugation action of the $\Q-$linear extension of $i\Star_{w,m}$ restricted to integral isometries, applying Proposition \ref{istar} and working verbatim as in \cite[Lemma 1.30]{OPR23}, we can prove the following 
\begin{lem}\label{lemisharp} Let $m,k$ be two positive integers, with $k>2$, and let $\SvH$ be an \linebreak[4]$(m,k)-$triple. The morphism $i_{w,m}\Star$ induces and isomorphism of groups $$i^{\sharp}_{w,m}:= (i\Star_{w,m})^{\sharp}\colon \O(\H^{2}(\Kv(S,H),\Z)) \longrightarrow \O(\H^{2}(\Kw(S,H),\Z))$$  satisfying the identity $i^{\sharp}_{w,m}=\lambda^\sharp_{\SwH}\circ (\lambda^\sharp_{\SvH})^{-1}$, under the identification $v\ort=w\ort$.
    \end{lem}
    Arguing in a similar fashion, we can extend Lemma \ref{lemisharp} to the case in which we have two $(m,k)-$triples $\SvHuno$ and $\SvHdue$, with $v_i=mw_i$ for $i=1,2$,  and compare the cohomology action of the embeddings $i_{m,w_1}$ and $i_{m,w_2}$. 
    
    In the following, we will write $K_{v_{i}}$ and $K_{w_{i}}$ for the moduli spaces $K_{v_{i}}(S_i,H_i)$ and $K_{w_{i}}(S_i,H_i)$, respectively, for $i=1,2$. In addition, let us denote by $$i\Star_{w_i,m,\Q}\colon \H^2(K_{v_i},\Q) \longrightarrow \H^2(K_{w_i},\Q)$$ the $\Q-$linear extension of $i\Star_{w_i,m}$ and let us consider the following bijective map \begin{align*}
    i_{w_{1},w_{2},m,\Q}^{\sharp}\colon \O(\H^{2}(K_{v_{1}},\Q), \H^{2}(K_{v_{2}},\Q))&\longrightarrow  \O(\H^{2}(K_{w_{1}},\Q), \H^{2}(K_{w_{2}},\Q))\\
    g&\longmapsto  i_{w_{2},m,\Q}\Star \circ g \circ ( i_{w_{1},m,\Q}\Star)^{-1}.
\end{align*} 
    \begin{lem}\label{lemisharp2}
        Let $m,k$ be two positive integers, with $k>2$, and let $\SvHuno$ and $\SvHdue$ be two $(m,k)-$triples, with $v_i=mw_i$, for $i=1,2$. Then the bijection $i^\sharp_{w_1,w_2,m,\Q}$ restricts to a bijection $$ i_{w_{1},w_{2},m}^{\sharp}\colon \O(\H^{2}(K_{v_{1}},\Z), \H^{2}(K_{v_{2}},\Z))\longrightarrow  \O(\H^{2}(K_{w_{1}},\Z), \H^{2}(K_{w_{2}},\Z))$$ satisfying, for every $g\in \O(\H^{2}(K_{v_{1}},\Z), \H^{2}(K_{v_{2}},\Z))$,
$$i_{w_{1},w_{2},m}^{\sharp}(g)= (\lambda_{(S_{2},w_{2},H_{2})}\circ \lambda_{(S_{2},v_{2},H_{2})}^{-1})\circ g \circ (\lambda_{(S_{1},w_{1},H_{1})}\circ \lambda_{(S_{1},v_{1},H_{1})}^{-1})^{-1}.$$
    \end{lem}
    It is immediate to notice that, if $(S_i, v_i, H_i)=\SvH$ for $i=2$, then the morphism $i_{w_{1},w_{2},m}^{\sharp}$ described in Lemma \ref{lemisharp2} coincides with the morphism $i_{w,m}^\sharp$ of Lemma \ref{lemisharp}.
\subsection{Action on monodromy operators}\label{secinj} The next Proposition, which is the main result of the Section, shows that the morphism $i_{w_{1},w_{2},m}^{\sharp}$ introduced in Lemma \ref{lemisharp2} sends locally trivial parallel transport operators to parallel transport operators. The outcome is that its restriction to the set of locally trivial parallel transport operators provides an injection of sets.
\begin{prop}\label{propinjPT}Let $m,k$ be two positive integers, with $k>2$, and let $\SvHuno$ and $\SvHdue$ be two $(m,k)-$triples, with $v_i=mw_i$, for $i=1,2$. Then the bijection $i_{w_{1},w_{2},m}^{\sharp}$ restricts to an injective function
    $$ i_{w_{1},w_{2},m}^{\sharp}\colon \PT_{lt}^{2}(K_{v_{1}}(S_{1},H_{1}), K_{v_{2}}(S_{2},H_{2}))\longrightarrow  \PT^{2}(K_{w_{1}}(S_{1},H_{1}), K_{w_{2}}(S_{2},H_{2})).$$
\end{prop}
\begin{proof}
    Let $p\colon \XX \to T$ be a locally trivial family of primitive symplectic varieties and let $t_1, t_2 \in T$ such that $\XX_{t_{i}}\simeq K_{v_i}(S_i,H_i)$. By Remark \ref{relstrat} (1) there exists a relative stratification $$\XX= \XX_0 \supseteq \XX_1 \supseteq \cdots \supseteq \XX_l=:\YY$$ that, restricted to each fiber $\XX_t$ of $p$, corresponds to the stratification of singularities in Proposition \ref{kalstratification}. In particular, for every $t\in T$, we have $\YY_t\simeq \XX_t^{ms}$ and, for $i=1,2$, $$\YY_{t_i}\simeq K_{v_i}(S_i,H_i)^{ms}\simeq \bigcup_{
         p\in S_{i}[m], L \in \widehat{S_{i}}[m]}a_{w_{i}}^{-1}(p,L),$$ where the last isomorphism follows from Proposition \ref{kvms}. 
By isotriviality of the Yoshioka fibrations $a_{w_i}$ for $i=1,2$, we can suppose that $K_{w_1}(S_1,H_1)$ and $K_{w_2}(S_2,H_2)$ belong to the same connected component $\ZZ$ of $\YY$. By Remark \ref{relstrat} (2), the restriction $q:=p_{|\ZZ}\colon \ZZ \to T$ of $p$ defines a smooth deformation of irreducible holomorphic symplectic manifolds such that $\ZZ_{t_i}\simeq K_{w_i}(S_i,H_i)$.

The proof now follows analogously to the K3 case (\cite[Proposition 4.7]{OPR23}), but we include it for the reader's convenience.
Let $$\PT_p\colon \Omega(T,t_1,t_2)\longrightarrow \O(\H^2(K_{v_1}(S_1,H_1),\Z), \H^2(K_{v_2}(S_2,H_2),\Z))$$ be the monodromy representation of the family $p$, sending the homotopy class of a continuous path $\gamma$ from $t_1$ to $t_2$ to the corresponding locally trivial parallel transport operator $\PT_p(\gamma)$ in the family $p$, and, in the same way, we denote by $\PT_q$ the monodromy representation of the restricted family $q$. Our claim is that the isomorphism $i_{w_1,w_2,m}^\sharp$ of Lemma \ref{lemisharp2} sends a locally trivial parallel transport operator $g=\PT_p(\gamma)$ in the family $p$ along a path $\gamma\in \Omega(T,t_1,t_2)$ in the parallel transport operator $\PT_q(\gamma)$ in the family $q$ along the same path $\gamma$. Equivalently, we claim that the diagram \begin{equation}\label{cdPT}
    \begin{tikzcd}[row sep=large]
        \Omega(T,t_1,t_2)\arrow[r, "\PT_p"] \arrow[rd, "\PT_q" swap] & \O(\H^2(K_{v_1}(S_1,H_1),\Z), \H^2(K_{v_2}(S_2,H_2),\Z))\arrow[d, "i_{w_1,w_2,m}^\sharp", "\rotatebox{90}{\(\sim\)}" swap]\\
        & \O(\H^2(K_{w_1}(S_1,H_1),\Z), \H^2(K_{w_2}(S_2,H_2),\Z))
    \end{tikzcd}
\end{equation} is commutative. 

Let $g=\PT_p(\gamma)$ be as above, and let $g_\Q$ be its $\Q-$linear extension, i.e. the locally trivial parallel transport operator $\PT^\Q_p(\gamma)$ along $\gamma$ in the local system $R^2p\sstar\Q$. By Lemma \ref{lemisharp2}, it is sufficient to show that $i_{w_1,w_2,m,\Q}^\sharp(g_\Q)$ is the parallel transport $\PT_q^\Q(\gamma)$ along $\gamma$ in the local system $R^2q\sstar\Q$, i.e. that diagram (\ref{cdPT}) commutes when replacing the morphisms with their respective $\Q-$linear extensions.

Since $p\colon \XX \to T$ is locally trivial, the embedding $\iota \colon \ZZ \to \XX$ induces a morphism of local systems $$\iota\Star_\Q\colon R^2p\sstar \Q \longrightarrow R^2q\sstar \Q$$ such that $\iota\Star_{\Q,t_i}=i\Star_{w_i,m,\Q}$. From Proposition \ref{istar} it follows that $\iota_{\Q,t_1}\Star$ is an isometry, hence $\iota_\Q\Star$ is an isomorphism of local systems. These two facts lead to the following equalities \begin{align*} \label{PTq}
    i\Star_{w_{2},m,\Q}\circ \PT_{p}(\gamma)_{\Q} \circ ( i\Star_{w_{1},m,\Q})^{-1}= \iota\Star_{\Q,t_{1}}\circ \circ \PT_{p}^{\Q}(\gamma) \circ (\iota\Star_{\Q,t_{2}})^{-1} = \PT_{q}^{\Q}(\gamma),
\end{align*} whose left hand side is exactly $i^{\sharp}_{w_{1},w_{2},m,\Q}(\PT_{p}(\gamma)_{\Q})$, by definition of $i_{w_1,w_2,m,\Q}^\sharp$.
\end{proof}
In the particular case of $\SvHuno=\SvHdue=:\SvH$ we get the following
\begin{cor}\label{corinjmon}Let $m,k$ be two positive integers, with $k>2$, and let $\SvH$ be an $(m,k)-$triple. Then the isomorphism $i^\sharp_{w,m}$ restricts to an injective morphism of groups $$i^{\sharp}_{w,m}\colon \monlt(\Kv(S,H)) \longrightarrow \mon^{2}(\Kw(S,H))$$ satisfying the identity $$i^{\sharp}_{w,m}(\PT_{p}(\gamma))=\PT_{q}(\gamma)$$ for every $\gamma \in \pi_1(T, p(\Kv(S,H)))$ and for every locally trivial deformation $p\colon \XX \to T$ of $\Kv(S,H)$ and every restriction $q$ of $p$ to any connected component of $\XX^{ms}$.
\end{cor}

We conclude the Subsection by extending the last result to any irreducible symplectic variety $X$ locally trivial deformation equivalent to a moduli space $\Kv(S,H)$ for some $(m,k)-$triple $\SvH$ with $k>2$.

\begin{cor}\label{corinjmondefo}
    Let $X$ be an irreducible symplectic variety, which is locally trivial deformation equivalent to $\Kv(S,H)$, where $\SvH$ is an $(m,k)-$triple with $k>2$, and let $Z$ be a connected component of the most singular locus of $X$. Then $Z$ is an irreducible holomorphic symplectic manifold deformation equivalent to $\Kw(S,H)$ and its closed embedding $i_{Z,X}\colon Z \to X$ induces an injective morphism $$i^\sharp_{Z,X}\colon \monlt(X) \longrightarrow \mon^2(Z).$$
\end{cor}
\begin{proof} Let $p\colon \XX \to T$ a locally trivial deformation of irreducible symplectic varieties and let $t_1,t_2\in T$ such that $\XX_{t_1}\simeq X$ and $\XX_{t_2}\simeq \Kv(S,H)$. 

The first assertion was already essentially contained in the proof of Proposition \ref{propinjPT} and it is a straightforward consequence of Remark \ref{relstrat} (2). In fact, for any connected component $\ZZ$ of the most singular locus of $\XX$, the restriction $q:=p_{|\ZZ}\colon \ZZ \to T$ defines a smooth deformation of irreducible holomorphic symplectic manifolds such that $\ZZ_{t_1}$ is isomorphic to a connected component $Z$ of $X^{ms}$ and $\ZZ_{t_2}\simeq \Kw(S,H)$.

For the second assumption, being the locally trivial monodromy group a locally trivial deformation invariant, it is sufficient to define the morphism $i^\sharp_{Z,X}$ in such a way that its action on locally trivial monodromy operators is natural with respect to the action of the morphism $i_{w,m}^\sharp$ described in Corollary \ref{corinjmon}. 

For every $g\in \monlt(X)$, there exists a locally trivial family $p'\colon \XX'\to T'$, a point $t'\in T'$ such that $\XX'_{t'}\simeq X$ and a loop $\gamma$ centered in $t'$ such that $g=\PT_{p'}(\gamma)$. As before, we let $q'\colon \ZZ'\to T'$ be the smooth deformation given by restriction of $p'$ to the connected component $\ZZ'$ of $(\XX')^{ms}$ such that $\ZZ'_{t'} \simeq Z\simeq  \ZZ_{t_1}$. We set $$i^\sharp_{Z,X}(g):=\PT_{q'}(\gamma)\in \mon^2(Z)$$ and we observe that this assignment is well defined and compatible with the composition law defined in Remark \ref{rmkcomposition} and, by construction, it satisfies the following identity \begin{equation}\label{isharp1}
    i^\sharp_{Z,X}(g)\circ i_{Z,X}\Star = i_{Z,X}\Star \circ g \colon \H^2(X,\Z) \to \H^2(Z,\Z),
\end{equation} where $i_{Z,X}$ is the closed embedding of $Z$ in $X$. Moreover, for every smooth path $\delta$ in $T$ from $t_1$ to $t_2$, the locally trivial parallel transport operators $\PT_p(\delta)$ and $\PT_q(\delta)$ in the two families defined above fit in the following commutation identity by construction: \begin{equation}\label{isharp2}
    i\Star_{w,m}\circ \PT_p(\delta) = \PT_q(\delta)\circ i\Star_{Z,X}\colon \H^2(X,\Z) \to \H^2(\Kw(S,H),\Z).
\end{equation} Putting identities (\ref{isharp1}) and (\ref{isharp2}) together, we deduce that $i^\sharp_{Z,X}$ fits in the following commutative diagram: \begin{equation}\label{cdmonX} \begin{tikzcd}
        \monlt(X) \arrow[d, "\PT_p(\delta)^\sharp" swap] \arrow[r, "i^\sharp_{Z,X}"] & \mon^2(Z)\arrow[d, "\PT_q(\delta)^\sharp"]\\
        \monlt(\Kv(S,H)) \arrow[r, "i^\sharp_{w,m}"] & \mon^2(\Kw(S,H)),
    \end{tikzcd}
\end{equation}where we are using notation (\ref{notconj}) for the isomorphisms $\PT_p(\delta)$ and $\PT_q(\delta)$. Since $i^\sharp_{w,m}$ is injective, the claim follows.
\end{proof}

As a consequence of Remark \ref{indexN}, we get the following
\begin{cor}\label{cormaximal} Let $X$ be an irreducible symplectic variety locally trivial deformation equivalent to $\Kv(S,H)$, where $\SvH$ is an $(m,k)-$triple with $k>2$. Then its locally trivial monodromy group $\monlt(X)$
is always a proper subgroup of $\O^+(\H^2(X,\Z))$.    
\end{cor}

In Example \ref{k=1}, we will see that the statement in Corollary \ref{cormaximal} is no longer true for $k\leq 2$, due to the existence of singular moduli spaces of sheaves on Abelian surfaces with maximal locally trivial monodromy group. More generally, in the next Section, we will provide some evidence of the fact that the hypothesis $k>2$ is crucial for the previous constructions. 

\subsection{The cases $k=1$ and $k=2$}\label{seck12} We conclude the Section giving some motivations for the hypothesis $k>2$ on the $(m,k)-$triples considered along Sections \ref{secicohomology}-\ref{secinj}.

We recall that, by Corollary \ref{corembedding}, for any $(m,k)-$triple $\SvH$ with $v=mw$ a primitive Mukai vector and $m,k$ two any positive integers, we can embed $\Kw(S,H)$ in $ \Kv(S,H)$ as a connected component of its most singular locus. This still holds, in particular, for $k=1,2$, but in these cases the associated primitive moduli spaces $\Kw(S,H)$ and the corresponding maps $\lambda_w\colon w\ort \to \H^2(\Kw(S,H),\Z)$ have an exceptional behavior.
\begin{rmk}[The case $k=1$]\label{k=1} If $k=1$, the moduli space $\Kw(S,H)$ is isomorphic to a point and the map $\lambda_w$ is the $0-$map. An analogue of Proposition \ref{istar} still holds, leading to the trivial statement $i_{w,m}\Star=0$, which suggests that no information about the locally trivial monodromy group of $\Kv(S,H)$ can be recovered from the monodromy group of its most singular locus (which is, in fact, trivial), as motivated by the following example.
\end{rmk} 
    \begin{ex}\label{exogsing} If $(m,k)=(2,1)$, by \cite[Theorem 1.6(2)]{PR10}, the moduli space $\Kv(S,H)$ is an irreducible symplectic variety that admits a symplectic resolution $\widetilde{\Kv(S,H)}$ deformation equivalent to the O'Grady's $6-$dimensional example $K_6$ (see \cite[Theorem 1.4]{OG00}). The locally trivial monodromy group of $\Kv(S,H)$ has been proven to be maximal, by \cite[Proposition 4.14, Corollary 4.12]{MR19}: $$\monlt(\Kv(S,H))\simeq \O^+(\H^2(\Kv(S,H),\Z)).$$ In particular, by Remark \ref{k=1}, the only admissible group morphism $$\monlt(\Kv(S,H)) \longrightarrow \mon^2(\Kw(S,H))\simeq \{0\}$$ is the $0-$map.\end{ex}
\begin{rmk}
    [The case $k=2$] \label{k=2} If $k=2$, then $\Kw(S,H)$ is a $K3$ surface and, in particular, it is a Kummer surface, by \cite[Theorem 3.2]{Yos99b}. Hence, by \cite{Bor72}, $$\mon^2(\Kw(S,H))\simeq \O^+(\H^2(\Kw(S,H),\Z)),$$ but there is no natural way of comparing the orthogonal groups $\O(\H^2(\Kv(S,H),\Z))$ and $\O(\H^2(\Kw(S,H),\Z))$ using the isometries provided by Theorem \ref{pr20thm1.6}, as explained in the following. Proposition \ref{istar} provides the identity $$i\Star_{w,m}= m\lambda_{\SwH} \circ \lambda_{\SvH}^{-1},$$ but in this case, invertibility of $\lambda_{\SwH}$ fails. In fact, by Theorem \ref{pr20thm1.6}, the isometry $\lambda_{\SwH}\colon w\ort \to \H^2(\Kw(S,H))$ is just injective and cannot be surjective, as $\rk(w\ort)=7$ and $\rk(\H^2(\Kw(S,H)))=22$.
\end{rmk}

Remarks \ref{k=1} and \ref{k=2} suggest that a description of the locally trivial monodromy group of moduli spaces $\Kv(S,H)$ given by $(m,k)-$triples with $k\in\{1,2\}$ may be achieved using different techniques. 

Furthermore, the explicit description of the monodromy group of smooth moduli spaces of sheaves on Abelian surfaces given by Theorem \ref{thmmonkummer} only works for $(1,k)-$triples with $k>2$, i.e. for moduli spaces \textit{of Kummer type}, and this result will play a fundamental role in the completion of the computation of the locally trivial monodromy group in the non primitive case. For this reasons, the cases $k=1$ and $k=2$ will be excluded from the present work.

%% file: sec3.tex
\label{secgroupoid}This Section is devoted to the development of the technical tools  that will provide an embedding of the group $\Nn(\Kv(S,H))$, described in Section \ref{secmonkummer}, in the locally trivial monodromy group $\monlt(\Kv(S,H))$ in a natural way, allowing us to complete the computation of the latter, regardless of the choice of the $(m,k)-$triple $\SvH$. In Sections \ref{secGmkdef} and \ref{secGmkFM} we will define a groupoid $\Gmk$ of $(m,k)-$triples, whose morphisms come from deformations of the triple itself, or from Fourier-Mukai equivalences. In Section \ref{secrepr} we will construct two different representations of $\Gmk$ with values in groupoids of free $\Z-$modules, designed to encode the monodromy information of the moduli spaces. 

We wish to point out that this Section parallels Section 2 of \cite{OPR23}, to which we refer for further details. The aim of this part of the work is, indeed, to review the main results and constructions, check their compatibility with the Abelian case and replace them with a suitable analogue if any difference occurs.\\ 

Let us start with a quick overview of the main definitions and constructions concerning groupoids that we will use in the following, referring the reader to the lecture notes \cite{Hig71} for a deeper discussion of this theory. 

\begin{defn}\label{defgroupoid}A \textit{groupoid} $\GG$ is a small category whose morphisms are all isomorphisms. \begin{enumerate}
    \item If $\GG$ is a groupoid and $x$ is an object of $\GG$, we define its \textit{isotropy group} as $$\Aut_\GG(x):= \Hom_\GG(x,x).$$ If $F\colon \GG \to \HH$ is a functor between two groupoids, for every $x\in Ob(\GG)$ we will denote its action on the respective isotropy groups as $$F_x\colon \Aut_\GG(x) \to \Aut_\HH(F(x)).$$
    \item If $\GG$ is a groupoid, a \textit{representation of $\GG$} is a functor $F\colon \GG \to \AA$, where $\AA$ is a suitable groupoid of $\Z-$modules.
    \item If $\GG$ and $\HH$ are two groupoids having the same objects, we define the \textit{free product of $\GG$ and $\HH$} as the groupoid $\GG \ast \HH$ whose objects are the objects of $\GG$ (and $\HH$) and whose morphisms are defined in the following way: if $x,y \in \GG\ast \HH$ are two objects, a morphism $f\in \Hom_{\GG\ast \HH}(x,y)$ is a formal combination (with usual cancellation properties) $$f=f_1\ast \cdots \ast f_l,$$ where $f_i$ is a morphism from $x_i \in Ob(\GG\ast \HH)$ to $x_{i+1}\in Ob(\GG\ast \HH)$ either in $\GG$ or $\HH$ for every $i=1,\dots, l$ and such that $x_1=x$ and $x_{l+1}=y$.
\end{enumerate}
    
\end{defn}

\subsection{Deformations of $(m,k)-$triples and their groupoid}\label{secGmkdef}
The content of this section can be seen as a categorical translation of the main ideas and tools that are involved in the following result.
\begin{thm}{(\cite[Theorem 1.7]{PR18})}\label{PR18defolt} Let $m,k \in \N\setminus\{0\}$ and let $\SvHuno$ and $\SvHdue$ be two $(m,k)-$triples. Then $K_{v_1}(S_1,H_1)$ and $K_{v_2}(S_2,H_2)$ are locally trivial deformation equivalent.\end{thm} Indeed, one of the key points of Theorem \ref{PR18defolt} is the construction of a deformation of a moduli space of sheaves induced by a deformation of the base surface, which we will recall in the following. 

Let $\SvH$ be an $(m,k)-$triple, with $v=(r,\xi,a)$, and let $L\in \Pic(S)$ be a line bundle such that $\c_1(L)=\xi$. If $f\colon \XX \to T$ is a morphism and $\LL \in \Pic(\XX)$, we will denote $\LL_t:=\LL_{|\XX_t}$.

\begin{defn}\label{defdefotripla}
    Let $\SvH$ be an $(m,k)-$triple as above and $T$ a smooth and connected algebraic variety. A \textit{deformation of $\SvH$ along $T$} is a triple $(f\colon \SS \to T, \LL, \HH)$, where \begin{enumerate}
        \item $f\colon \SS \to T$ is a smooth, projective deformation of $S$, with $0\in T$ the point such that $\SS_0\simeq S$;
        \item $\LL$ is a line bundle on $\SS$ such that $\LL_0\simeq L$;
        \item  $\HH$ is a line bundle on $\SS$ such that $\HH_t$ is a $v_t-$generic polarization on $\SS_t$ for every $t\in T$ and such that $\HH_0\simeq H$,
    \end{enumerate} where, for every $t \in T$, we set $v_t:= (r,\c_1(\LL_t),a)$.
\end{defn}
\begin{rmk}\label{vgenopen}For later use, we remark that $v-$genericity is an open property in the Zariski topology, as shown in Proposition 2.14 of \cite{PR18}. In fact, up to replace the base $T$ with the complement of the Zariski closed subset $$Z:= \{t \in T \text{ such that }\HH_t \text{ is not }v_t-\text{generic}\},$$ we can find a deformation as in Definition \ref{defdefotripla} such that property (3) is satisfied. Notice that non-emptiness of $T\setminus Z$ is a consequence of the assumption that $(S,v,H)$ is an $(m,k)-$triple and $\HH_0\simeq H$.
\end{rmk}

\begin{defn}\label{defndefopath}
    Let $\SvHuno$ and $\SvHdue$ be two $(m,k)-$triples. A \textit{deformation path from $\SvHuno$ to $\SvHdue$} is a $6-$tuple $$\alpha:= (f\colon \SS \to T, \LL, \HH, t_1,t_2,\gamma),$$ where \begin{enumerate}
        \item the triple $(f\colon \SS \to T, \LL, \HH)$ is a deformation of both $\SvHuno$ and $\SvHdue$;
        \item for $i=1,2$, the point $t_i\in T$ is such that $(\SS_{t_i},v_{t_i},\HH_{t_i})= (S_i,v_i,H_i)$;
        \item $\gamma$ is a continuous path in $T$ from $t_1$ to $t_2$.
    \end{enumerate}
\end{defn}

We will now explain, with a few remarks, how deformation paths provide a tool to produce locally trivial monodromy operators in a natural way. We start by recalling the following
\begin{defn}\label{defnpttilde}Let $S_1$ and $S_2$ be two Abelian surfaces. An isometry $g\colon \Htilde(\SS_{t_1},\Z) \to \Htilde(\SS_{t_2},\Z)$ is a parallel trasport operator if there exists a smooth family $f\colon \SS \to T$ of Abelian surfaces, $t_1, t_2\in T$ such that $\SS_{t_i}\simeq S_i$ for $i=1,2$ and a continuous path $\gamma$ in $T$ from $t_1$ to $t_2$ such that $g$ is the parallel transport $\PT_f(\gamma)$ along $\gamma$ inside the local system $R\ev f\sstar \Z= \bigoplus_{i=0}^2R^{2i}f\sstar \Z$.
\end{defn}

\begin{rmk}\label{rmkrelmod}
    If $\SvH$ is an $(m,k)-$triple and $(f\colon \SS \to T, \LL, \HH)$ is a deformation of $\SvH$ along $T$, as explained in Section 2.3 of \cite{PR18}, it induces a locally trivial deformation of the moduli space $\Kv(S,H)$ as follows: we consider the relative moduli space of semistable sheaves $\phi\colon \MM \to T$, so that for every $t\in T$ we have $\MM_t=M_{v_t}(\SS_t,\HH_t)$, and we let $\widehat{\SS} \to T$ be the dual family of $\SS \to T$, so that $\widehat{\SS}_t$ is the dual surface of $\SS_t$ for every $t\in T$. By \cite[Remark 2.19]{PR18}, a relative Yoshioka fibration $$a\colon \MM \to \SS \times_T \widehat{\SS}$$ is defined and its restriction on each fiber $\MM_t$ of $\phi$ coincides with the Yoshioka fibration $a_{v_t}\colon M_{v_t}(\SS_t,\HH_t) \to \SS_t \times \widehat{\SS_t}$ defined in Remark \ref{yoshfibrisotr}. If we set $$Z:= \{(0_{\SS_t},\OO_{\SS_t})\in \SS_t \times_T \widehat{\SS}_t |\text{ } t \in T\} \subseteq \SS_t \times \widehat{\SS_t},$$ the restriction of $\phi$ to $\KK:= a^{-1}(Z)$ provides us a morphism $p\colon \KK \to T$, whose fiber over $t \in T$ is $\KK_t= K_{v_t}(\SS_t,\HH_t)$. By Lemma 2.21 of \cite{PR18}, the family $p\colon \KK \to T$ is a locally trivial deformation of $\Kv(S,H)$ along $T$.
\end{rmk}

In light of Remark \ref{rmkrelmod}, if $\alpha=(f\colon \SS \to T, \LL, \HH, t_1,t_2,\gamma)$ is a deformation path between two $(m,k)-$triples $\SvHuno$ and $\SvHdue$ and $p\colon \KK \to T$ is the relative moduli space associated to the deformation $(f\colon \SS \to T, \LL, \HH)$, then $\alpha$ induces two natural locally trivial parallel transport operators: \begin{enumerate}
    \item[(a)] the parallel transport operator \begin{equation}
        \label{palpha}p_\alpha:= \PT_f(\gamma)\colon \Htilde(S_1,\Z)\longrightarrow \Htilde(S_2,\Z)
    \end{equation}  along $\gamma$ in the local system $R\ev f\sstar\Z$; 
    \item[(b)] the locally trivial parallel transport operator  \begin{equation}
        \label{galpha}g_\alpha := \PT_p(\gamma)\colon \H^2(K_{v_1}(S_1,H_1)) \longrightarrow \H^{2}(K_{v_2}(S_2,H_2))
    \end{equation} along $\gamma$ in the local system $R^2p\sstar\Z$.
\end{enumerate}
\begin{rmk}\label{rmklocsystem} In order to clarify the relation between the two locally trivial parallel transport operators $p_\alpha$ and $g_\alpha$, we compare the two local systems involved. \begin{enumerate}
    \item Due to condition (2) in Definition \ref{defdefotripla}, the local system $R\ev f\sstar \Z$ admits a flat section $\v$ such that, for every $t\in T$, we have $\v_t= v_t\in \Htilde(\SS_t,\Z)$. This allows us to consider the sub-local system $$\v\ort \subseteq R \ev f\sstar \Z.$$
    \item If $\SvH$ is an $(m,k)-$triple with $k>2$, then the isometry $\lambda_{\SvH}\colon v\ort \to \H^2(\Kv(S,H),\Z)$ in Theorem \ref{pr20thm1.6} extends to an isometry of local systems $$\lambda \colon \v\ort \longrightarrow R^2p\sstar\Z.$$ Indeed, the relative moduli space of stable sheaves $p^s\colon \KK^s \to T$ induced by $(f\colon \SS \to T, \LL, \HH)$ naturally includes in $p\colon \KK \to T$ via a relative open embedding, inducing an isomorphism $\iota\colon R^2p^s\sstar \Z \to R^2 p \sstar \Z$ of local systems, by \cite[Corollary 3.3(2), Proposition 3.5(3)]{PR20}. By Proposition 4.4(2) and Definition 4.6(2) of \cite{PR20}, a quasi-universal family of $\MM^s \to T$ induces an isomorphism $\lambda^s\colon v\ort \to R^2p^s\sstar \Z$ of local systems such that the composition $$\lambda:= \iota \circ \lambda^s \colon v\ort \longrightarrow R^2 p\sstar \Z$$ coincides, over every $t \in T$, with $\lambda_{(\SS_t,v_t, \HH_t)}$. In other words, the isometry $\lambda_{\SvH}$ behaves well in deformations of $(m,k)-$triples.
    \item By definition, the parallel transport operator $p_\alpha$ is constant along $\v$, i.e. \linebreak[4]$p_\alpha(v_{t_1})=v_{t_2}$ for every $t_1, t_2 \in T$, hence its restriction $p_{\alpha|v_{t_1}\ort}$ defines a parallel transport operator in the local system $\v\ort$, which is isomorphic, due to point (2), to $R^2p\sstar \Z$. Since $g_\alpha$ and $p_{\alpha|\v\ort}$ are parallel transport operators over the same path in isomorphic local systems, we deduce that $p_\alpha$ uniquely determines $g_\alpha$.
\end{enumerate}   
\end{rmk}

From now on, we will consider only $(m,k)-$triples with $k>2$. In light of the last Remark and in order to give a well defined composition law for deformation paths, we introduce the following 

\begin{defn}\label{defequivalence}
    Let $\SvHuno$ and $\SvHdue$ be two $(m,k)-$triples. Two deformation paths $\alpha$ and $\alpha'$ from $\SvHuno$ to $\SvHdue$ are \textit{equivalent} if $p_\alpha = p_\alpha'$. We will denote the equivalence class of $\alpha$ by $\overline{\alpha}$.
\end{defn}

We immediately notice that, by point (3) of Remark \ref{rmklocsystem}, if two deformation paths $\alpha$ and $\alpha'$ are equivalent, then also $g_\alpha=g_{\alpha'}$.\\

Let $\SvHuno$, $\SvHdue$ and $(S_3,v_3,H_3)$ be three $(m,k)-$triples and let $\alpha= (f\colon \SS \to T,\LL,\HH,t_1,t_2,\gamma)$ be a deformation path from $\SvHuno$ to $\SvHdue$ and $\alpha'=(f'\colon \SS' \to T',\LL',\HH',t_1',t_2',\gamma')$ a deformation path from $\SvHdue$ to $(S_3,v_3,H_3)$.

\begin{defn}
    The \textit{concatenation of $\alpha$ with $\alpha'$} is the $6-$tuple $$\alpha\star\alpha':= (f''\colon \SS'' \to T'',\LL'',\HH'',t_1'',t_2'',\gamma''),$$ where \begin{itemize}
        \item $T''$ is obtained by gluing $T$ and $T'$ along $t_2$ and $t_1'$;
        \item $\SS''$ is obtained by gluing $\SS$ and $\SS'$ along $\SS_{t_2}$ and $\SS'_{t_1'}$;
        \item $f''$ is obtained by gluing $f$ and $f'$ along $\SS_{t_2}$ and $\SS'_{t_1'}$;
        \item $\LL''$ is obtained by gluing $\LL$ and $\LL'$ along $\SS_{t_2}$ and $\SS'_{t_1'}$;
        \item $\HH''$ is obtained by gluing $\HH$ and $\HH'$ along $\SS_{t_2}$ and $\SS'_{t_1'}$;
        \item $t_1''$ is the image of $t_1$ in $T''$;
        \item $t_2''$ is the image of $t_3$ in $T''$;
        \item $\gamma''$ is the concatenation of the image of the path $\gamma$ in $T''$ with the image of the path $\gamma'$ in $T''$.
    \end{itemize}
\end{defn}

\begin{rmk}\label{rmkconcat}We remark that, in general, the concatenation $\alpha\star\alpha'$ does not define a deformation path in the sense of Definition \ref{defndefopath}, since the base $T''$ obtained by gluing might not be smooth. Nonetheless, if we set \begin{center}
    $p_{\alpha \star\alpha'}:= p_{\alpha'}\circ p_\alpha$ and $g_{\alpha \star\alpha'}:= g_{\alpha'}\circ g_\alpha$,
\end{center} we get two well defined locally trivial parallel transport operators, by Remark \ref{rmkcomposition}. We can then extend the equivalence relation in Definition \ref{defequivalence} to concatenations of $(m,k)-$triples and notice that, if $\alpha$ is equivalent to $\beta$ and $\alpha'$ is equivalent to $\beta'$, then $\alpha\star\alpha'$ is equivalent to $\beta\star\beta'$. As before, we will denote the equivalence class of $\alpha\star\alpha'$ with $\overline{\alpha\star\alpha'}$.
\end{rmk}
Thanks to the previous definitions and remarks, we are now able to introduce the first of the two groupoids needed to define the groupoid $\Gmkdef$ of deformations of $(m,k)-$triples.
\begin{defn}Let $m,k\in \N\setminus\{0\}$, with $k>2$. The groupoid $\widetilde{\GG}^{m,k}_{def}$ is defined as follows: \begin{itemize}
    \item the objects of $\widetilde{\GG}^{m,k}_{def}$ are the $(m,k)-$triples;
    \item if $\SvHuno$ and $\SvHdue$ are two $(m,k)-$triples, an \textit{elementary morphism} between them is an equivalence class of deformation paths from $\SvHuno$ to $\SvHdue$;
    \item if $\SvHuno$ and $\SvHdue$ are two $(m,k)-$triples, a morphism in $\widetilde{\GG}^{m,k}_{def}$ from $\SvHuno$ to $\SvHdue$ is a formal concatenation $$\overline{\alpha_1} \ast \cdots \ast \overline{\alpha_l} $$ of elementary morphisms and their formal inverses, subject to usual cancellation rules, prescribed by \begin{equation}\label{concat}
        \overline{\alpha_i} \ast \overline{\alpha_{i+1}}:= \overline{\alpha_i\star \alpha_{i+1}},
    \end{equation} for every $i=1,\dots,l-1$.
\end{itemize}
\end{defn}

\begin{rmk}
    By Remark \ref{rmkconcat}, we get the good definition of the composition law (\ref{concat}) and an explicit description of \begin{itemize}
        \item the \textit{identity morphism} $$\id_{\SvH}:= \overline{(S \to \{p\}, L,H, p,p, k_p)},$$ where $k_p$ is the constant path in $p$, for any object $\SvH$ in $\widetilde{\GG}^{m,k}_{def}$;
        \item the \textit{formal inverse} of any elementary morphism $\overline{\alpha}= \overline{\defopath}$, defined by $\overline{\alpha}^{-1}:= \overline{\alpha^{-1}}$, where $$\alpha^{-1}:= (f\colon \SS \to T, \LL, \HH, t_2,t_1,  \gamma^{-1}).$$
    \end{itemize}  
\end{rmk}

We now complete the setting with the groupoid $\Pmk$ of congruent $(m,k)-$triples, according to Definition \ref{defncongruent}.

\begin{defn}
    Let $m,k\in \N\setminus\{0\}$, with $k>2$. The groupoid $\Pmk$ is defined as follows: \begin{itemize}
        \item the objects of $\Pmk$ are the $(m,k)-$triples;
        \item if $\SvHuno$ and $\SvHdue$ are two $(m,k)-$triples, we set $$\Hom_{\Pmk}(\SvHuno,\SvHdue):= \{\chi_{H_1,H_2}\}$$ as in (\ref{idcongruent}), if $\SvHuno$ and $\SvHdue$ are congruent, and otherwise $$\Hom_{\Pmk}(\SvHuno,\SvHdue):=\emptyset.$$
    \end{itemize}
\end{defn}
We have now collected all the material needed to define the groupoid $\Gmkdef$.

\begin{defn}
     Let $m,k\in \N\setminus\{0\}$, with $k>2$. We define the groupoid $$\Gmkdef:= \widetilde{\GG}^{m,k}_{def} \ast \Pmk$$ as the free product of $\widetilde{\GG}^{m,k}_{def}$ and $\Pmk$, according to Definition \ref{defgroupoid} (3).
\end{defn}

\subsection{Fourier-Mukai equivalences and their groupoid} \label{secGmkFM} In this Section we will introduce the groupoid $\GmkFM$, whose morphisms are defined by some Fourier-Mukai equivalences on $\Db(S)$ inducing isomorphisms on the moduli spaces $\Kv(S,H)$. 
\subsubsection{Tensorization with line bundles} Let $S$ be an Abelian surface, let $L\in \Pic(S)$ be a line bundle and let us consider the derived equivalence \begin{align} \label{L}
    \L\colon \Db(S)  \longrightarrow \Db(S), \phantom{++}
    F  \longmapsto F \otimes L.
\end{align} If $F$ is a sheaf on $S$ with Mukai vector $v(F)=v$, then \begin{equation}
    \label{vtensor}v(\L(F))=v(F)\cdot \ch(L)=:v_L
\end{equation} and the equivalence (\ref{L}) induces a morphism $$\L\colon \Kv(S,H) \longrightarrow K_{v_L}(S,H).$$ The next result provides a sufficient criterion for the morphism $\L$ to be an isomorphism.
\begin{lem}{\cite[Lemma 2.24]{PR18}}\label{lemtensor} Let $S$ be an Abelian surface, $v=(r,\xi,a)$ a Mukai vector and $H$ an ample line bundle on $S$.
\begin{enumerate}
    \item For any $d\in \Z$, the morphism $\d\HFM\colon \Kv(S,H) \to K_{v_{dH}}(S,H)$ is an isomorphism.
    \item If $r>0$ and $H$ is $v-$generic, the morphism $\L\colon \Kv(S,H) \to K_{v_L}(S,H)$ is an isomorphism.
\end{enumerate}
\end{lem}

\subsubsection{The Poincaré line bundle as kernel}\label{secpoincaré} Let $S$ be an Abelian surface, let $\PP \in \Db(S\times \hat{S})$ be the Poincaré line bundle on $S$. The Fourier-Mukai transform \begin{align*}
    \FM_{\PP}\colon \Db(S)  \longrightarrow \Db(\hat{S}), \phantom{++}
    F  \longmapsto R\pi_{\hat{S}\ast}(\pi_{S}\Star(F)\overset{L}{\otimes}\PP),
\end{align*}where $\pi_S\colon S \times \hat{S} \to S$ and $\pi_{\hat{S}}\colon S \times \hat{S} \to \hat{S}$ are the two natural projections, is known to be an equivalence (\cite{Mu81}, see also \cite[Proposition 9.19]{Huy03b}). For any $L \in \Pic(S)$, we set $\hat{L}:= \det(\FM_\PP(L))^{-1}\in \Pic(\hat{S})$ and, if $\c_1(L)=\xi$, we set $\hat{\xi}:= \c_1(\hat{L})$. If $(r,\xi,a)$ is a Mukai vector on $S$, then the cohomological action of $\FM_\PP$ (see \cite[Lemma 9.23]{Huy03b}) provides the equality
\begin{equation}\label{vtilde}
    v(\FM_\PP(r,\xi,a))= (a,-\hat{\xi},r) =: \tilde{v},
\end{equation} where the latter is a Mukai vector on $\hat{S}$. 

We now consider the derived duality operation \begin{align*}
    D_S\colon \Db(S)  \longrightarrow \Db(S), \phantom{++}
    F  \longmapsto F\dual,
\end{align*} which is a contravariant autoequivalence, as $S$ is smooth. 
Its cohomological action, restricted to $\Htilde(S,\Z)$, provides an involution and isometry satisfying \begin{equation}\label{vduale}
    v(D_S(r,\xi,a))=(r,-\xi,a),
\end{equation} for any Mukai vector $v=\rxia$ on $S$. Combining the two previous derived equivalences we can define 
\begin{align*}
    \FM_\PP\dual := D_{\hat{S}}\circ \FM_\PP\colon \Db(S) \longrightarrow \Db(\hat{S}) , \phantom{++}
    F  \longmapsto (R\pi_{\hat{S}\ast}(\pi_{S}\Star(F)\overset{L}{\otimes}\PP))\dual
\end{align*} and get, by (\ref{vtilde}) and (\ref{vduale}), for any Mukai vector $v=\rxia$, \begin{equation}\label{vcap}
    v(\FM_\PP\dual(r,\xi,a))= (a,\hat{\xi},r) =: \hat{v}.
\end{equation}
Again, we turn our attention to the morphisms induced on the corresponding moduli spaces by the previous derived functors. For this purpose, we recall that, if $H$ is an ample line bundle on $S$, then $\hat{H}$ is an ample line bundle on $\hat{S}$ (see \cite[Proposition 3.11]{Mu81}).
\begin{lem}{(\cite{Yos01b},\cite{PR18})}\label{lemFM}
    Let $S$ an Abelian surface, $H$ an ample line bundle on $S$ with $\c_1(H)=:h$ and $n,a \in \Z$.
    \begin{enumerate}
        \item Suppose that $\NS(S)=\Z h$ and $v=(r,nh,a)$, with $r>0$. Then there exists an integer $n_0\gg 0$ such that for every $n>n_0$ the Fourier-Mukai equivalence $\FM_\PP$ induces an isomorphism $$\FM_{\PP,v}\colon \Kv(S,H) \longrightarrow K_{\tilde{v}}(\hat{S},\hat{H}).$$
        \item Set $v=(0,\xi,a)$ and suppose that $H$ is $v-$generic and $\hat{H}$ is $\tilde{v}-$generic. Then there exists an integer $a_0\gg 0 $ such that for every $a>a_0$ the Fourier-Mukai equivalence $\FM_{\PP,v}$ induces an isomorphism $$\FM_\PP\colon \Kv(S,H)\longrightarrow K_{\tilde{v}}(\hat{S},\hat{H}).$$
        \item In any of the two cases above, the Fourier-Mukai equivalence $\FM_{\PP,v}\dual$ induces an isomorphism $$\FM_{\PP,v}\dual \colon \Kv(S,H)\longrightarrow K_{\hat{v}}(\hat{S},\hat{H}).$$
    \end{enumerate}
\end{lem}
\begin{proof}
    Exactly as in \cite[Lemma 2.15]{OPR23}, the claim is a straightforward application of Proposition 2.29, Proposition 2.33 and Lemma 2.28 of \cite{PR18}. See \cite{Yos01b}, Theorem 3.18 for a more general statement.
\end{proof}
\subsubsection{A relative Poincaré line bundle as kernel in the elliptic case}\label{secelliptic} In the following, we will let $p\colon S \to E$ be an elliptic Abelian surface, with $E$ an elliptic curve, and we will let $s\colon E \to S$ be a $0-$section. We will denote by $f:=\c_1(p\Star(\OO_E(1)))\in \NS(S)$ the class of a fiber of $p$ and by $e:=\c_1(s(E))\in \NS(S)$ the class of the section $s$. Then, one can easily check that $e^2=0=f^2$ and $e\cdot f=1$, so that, if we assume that $\NS(S)=\brakett{e,f}$, then $\NS(S)$ is isometric to the unimodular rank $2$ hyperbolic lattice $U$.\\

Let $H\in \Pic(S)$ such that $h:= \c_1(H)=e + tf$ and we assume $t\gg 0$, so that, by Nakai-Moishezon criterion, the class $h$ is ample. Moreover, it can be checked (see Lemma 2.12 of \cite{PR18}) that the polarization $H$ is generic with respect to the Mukai vector $(0,f,0)$. As explained in \cite[Section 4]{Bri98}, the moduli space $M_{(0,f,0)}(S,H)$ is fine and parametrizes stable sheaves $F$ of pure dimension $1$, supported on the fibers of $p$, where their restriction has rank $1$ and degree $\c_1(F)\cdot f=0$. Furthermore, it is equipped by a fibration $M_{(0,f,0)}(S,H) \to E$ induced by $p$, whose fibers are isomorphic to the fibers of $p$ by \cite[Theorem 7]{At57} and, indeed, \begin{equation}\label{isoS}
    M_{(0,f,0)}(S,H) \simeq S
\end{equation} as elliptic surfaces (see \cite[Section 4.2]{Bri98}). By Theorem 1.2 of \cite{Bri98}, there exist tautological sheaves $\PP$ on $M_{(0,f,0)}(S,H) \times S$ - regarded as coherent sheaves on $S\times S$ via the isomorphism (\ref{isoS}) - such that the Fourier-Mukai transform $\FM_\PP \colon \Db(S) \longrightarrow \Db(S)$ with kernel $\PP$ is an equivalence. We choose one of such tautological sheaves - often called \textit{relative Poincaré line bundles} - and we denote it by $\EE$, to avoid confusion with the actual Poincaré line bundle $\PP$ used in Section \ref{secpoincaré}. We consider the Fourier-Mukai equivalence \begin{align*}
    \FM_\EE \colon \Db(S)  \longrightarrow \Db(S), \phantom{++}
    F \longmapsto (R\pi_{2\ast}(\pi_{1}\Star(F)\overset{L}{\otimes}\EE))[1],
\end{align*} where $\pi\colon S\times S \to S$ is the projection on the $i-$th factor for $i=1,2$, and where the shift $[1]$ is motivated by Lemma 3.9 and Definition 2.10 of \cite{Yos01a}.

\begin{lem}{(\cite[Theorem 3.15, Proposition 4.9]{Yos01a})}\label{lemyosh} Let $S$ be an elliptic Abelian surface as above and let $H \in \Pic(S)$ such that $h=e+tf$, with $t\gg 0$. Then $\FM_\EE$ induces an isomorphism \begin{equation}
    \label{vell}\FM_{\EE} \colon K_{(m,0,-mk)}(S,H) \longrightarrow K_{(0,m(e+kf),m)}(S,H)
\end{equation} for every $m,k>0$.
\end{lem}

We can now introduce the groupoid $\GmkFM$.

\begin{defn}\label{defgmkFM}Let $m,k\in \N\setminus\{0\}$, with $k>2$. We define the groupoid $\GmkFM$ as follows: \begin{itemize}
    \item the objects of $\GmkFM$ are the $(m,k)-$triples;
    \item if $\SvHuno$ and $\SvHdue$ are two $(m,k)-$triples, an \textit{elementary morphism} between them is one of the following: \begin{itemize}
        \item the equivalence $\L$ if $\SvHuno=(S,v,H)$ and $\SvHdue=(S,v_L,H)$ are as in Lemma \ref{lemtensor};
        \item the Fourier-Mukai equivalence $\FM_\PP$ if $\SvHuno=\SvH$ and \linebreak[4]$\SvHdue=(\hat{S},\tilde{v},\hat{H})$ verify the hypotheses of Lemma \ref{lemFM} (1) or (2);
        \item the Fourier-Mukai equivalence $\FM_\PP\dual$ if $\SvHuno=\SvH$ and \linebreak[4]$\SvHdue=(\hat{S},\hat{v},\hat{H})$ verify the hypotheses of Lemma \ref{lemFM} (3)
        ;
        \item the Fourier-Mukai equivalence $\FM_\EE$ if $\SvHuno=(S, (m,0,-mk),H)$ and $\SvHdue=(S,(0, m(e+kf),m),H)$ verify the hypotheses of \linebreak[4]Lemma \ref{lemyosh};
    \end{itemize}
    \item if $\SvHuno$ and $\SvHdue$ are two $(m,k)-$triples, a morphism in $\GmkFM$ from $\SvHuno$ to $\SvHdue$ is a formal concatenation of elementary morphisms and their formal inverses, subject to usual cancellation rules.
\end{itemize}
\end{defn}

\subsection{The groupoid $\Gmk$ and its representations}\label{secrepr}
Thanks to the notions introduced in Sections \ref{secGmkdef} and \ref{secGmkFM}, we are finally in the position to define the groupoid $\Gmk$ of $(m,k)-$triples.

\begin{defn}
    Let $m,k\in \N\setminus\{0\}$, with $k>2$. We define the groupoid $$\Gmk:= \Gmkdef\ast \GmkFM$$ as the free product of the groupoids $\Gmkdef$ and $\GmkFM$. 
\end{defn}

\begin{rmk}\label{rmkhomgmknonempty}
    In the proof of \cite[Theorem 1.7]{PR18}, it is shown that, if $\SvHuno$ and $\SvHdue$ are two $(m,k)-$triples, then one can always construct a path from $\SvHuno$ to $\SvHdue$ only using the following elementary morphisms in $\Gmk$: \begin{itemize}
        \item equivalence classes of deformation paths of $(m,k)-$triples, inducing locally trivial deformations of moduli spaces;
        \item the identity morphism $\chi_{H_1,H_2}$ of $\Pmk$, inducing identifications of moduli spaces;
        \item derived equivalences of the form $\L$, where $L$ is a suitable line bundle;
        \item the Fourier-Mukai equivalence $FM_\PP$,
    \end{itemize}where the last two morphisms induce isomorphisms of moduli spaces. Consequently, for any pair of $(m,k)-$triples as above, we get $$\Hom_{\Gmk}(\SvHuno,\SvHdue)\neq \emptyset.$$ 
\end{rmk}
\subsubsection{The $\HHtilde-$representation $\Phitildemk$ of $\Gmk$} We start by recalling that, if $S$ is an Abelian surface, then its Mukai lattice $\Htilde(S,\Z)$ is isometric to $U^{\oplus 4}$, where $U$ is the unimodular rank $2$ hyperbolic lattice. Motivated by this, in the following we will let $\Lambdatilde$ be an even unimodular lattice  of signature $(4,4)$, which is isometric, by Milnor's Theorem (\cite{Mil58}), to the lattice $U^{\oplus 4}$, and we will endow it with an orientation $\epsilon$, as defined in Appendix \ref{rmkorient} (b). We will use this data to define a new groupoid of $\Z-$modules, as follows:
\begin{defn}\label{defHtildemk}Let $m,k\in \N\setminus\{0\}$, with $k>2$. We define the groupoid $\Htildemk$ as follows: \begin{itemize}
    \item the objects are triples $(\Lambdatilde, v, \epsilon)$, where $\Lambdatilde$ is an even unimodular lattice of signature $(4,4)$, $v \in \Lambdatilde$ is of the form $v=mw$, where $w\in \Lambdatilde$ is primitive and $w^2=2k$, and $\epsilon$ is an orientation on $\Lambdatilde$;
    \item if $(\Lambdatilde_1,v_1,\epsilon_1)$ and $(\Lambdatilde_2,v_2,\epsilon_2)$ are two objects in $\Htildemk$, then we set $$\Hom_{\Htildemk}((\Lambdatilde_1,v_1,\epsilon_1),(\Lambdatilde_2,v_2,\epsilon_2)):= \{g\in \O(\Lambdatilde_1,\Lambdatilde_2) \colon g(v_1)=v_2\}.$$
    \end{itemize}
\end{defn}
We notice that morphisms in $\Htildemk$ are not necessarily orientation preserving, in the sense of Appendix \ref{rmkorient}.

\begin{ex}\label{exorientS}As previously remarked, if $S$ is an Abelian surface, then its Mukai lattice $\Htilde(S,\Z)$ is an even unimodular lattice of signature $(4,4)$. We now define an orientation on $\Htilde(S,\Z)$ starting from an orientation on $\H^2(S,\Z)$, which is an even unimodular lattice of signature $(3,3)$. \begin{enumerate}
    \item We point out that, from the description of the period domain and of the period map of Abelian surfaces (\cite[Theorem II]{Shi78}) we get that any holomorphic $2-$form $\sigma \in \H^0(S,\Omega_S^2)$ satisfies $\sigma^2=0$, $\sigma\cdot \overline{\sigma}>0$ and $\sigma \perp \H^{1,1}(S)$, from which we deduce that $\Re(\sigma)$ and $\Im(\sigma)$ span a positive definite real subspace of $\H^2(S,\R)$, orthogonal to $\R\omega$, where $\omega$ is any K\"ahler form, which is again positive definite by Hodge Index Theorem. We conclude that the basis $\{\omega, \Re(\sigma), \Im(\sigma)\}$ defines an orientation on $\H^2(S,\Z)$, which, as in Remark \ref{exorientPSV}, does not depend on the choice of the symplectic form and of the K\"ahler form. 
    \item Finally, we recall that $\H^2(S,\Z)$ naturally embeds in $\Htilde(S,\Z)$ and its orthogonal complement, with respect to the Mukai pairing, coincides with $\H^0(S,\Z)\oplus \H^4(S,\Z)$. Hence, we can naturally identify $\O(\H^2(S,\Z))$ with the subgroup of $\O(\Htilde(S,\Z))$ of isometries acting as the identity on $\H^0(S,\Z)\oplus \H^4(S,\Z)$, and use this to extend any orientation of $\H^2(S,\Z)$ to an orientation of $\Htilde(S,\Z)$. Indeed, by adding the $2-$vector $(1,0,-1)\in \Htilde(S,\Z)$, orthogonal to $\H^2(S,\Z)$, to the previous orientation, we get an orientation $$\{\omega, \Re(\sigma), \Im(\sigma), (1,0,-1)\}$$ of $\Htilde(S,\Z)$, which we will denote by $\epsilon_S$.
\end{enumerate}
\end{ex}
The last Example allows us to define the representation $\Phitildemkdef\colon \Gmkdef \to \Htildemk$.
\begin{defn}Let $m,k\in \N\setminus\{0\}$, with $k>2$. We define the representation $$\Phitildemkdef\colon \Gmkdef \longrightarrow \Htildemk$$
    is defined as follows: \begin{itemize}
        \item if $\SvH$ is an object in $\Gmkdef$, then we set $\Phitildemkdef(\SvH):=(\Htilde(S,\Z),v,\epsilon_S)$, where $\epsilon_S$ is the orientation defined in Example \ref{exorientS};
        \item if $\SvHuno$ and $\SvHdue$ are two objects in $\Gmkdef$ and $\alpha=\defopath$ is a deformation path from $\SvHuno$ to $\SvHdue$, then $$\Phitildemkdef(\overline{\alpha}):=p_\alpha$$ is the parallel transport operator in the local system $R\ev f\sstar \Z$ along the path $\gamma$ (as in (\ref{palpha}));
        \item if $\SvHuno$ and $\SvHdue$ are congruent, then we set $$\Phitildemkdef(\chi_{H_1,H_2}):= \id_{\Htilde(S,\Z)}.$$
    \end{itemize}
\end{defn}
\begin{rmk}\label{rmkimphitildedef} Let $\alpha =\defopath$ be a deformation path from \linebreak[4]$(S_1,v_1,H_1)$ to $\SvHdue$. \begin{enumerate}
    \item As already noticed in Remark \ref{rmklocsystem} (1), the Mukai vectors $v_1$ and $v_2$ belong to the same flat section of the local system $R\ev f\sstar \Z$. Hence, the parallel transport $p_\alpha$ maps $v_1$ to $v_2$ and the representation $\Phitildemkdef$ is well defined.
    \item Moreover if $\epsilon_{S_i}=\{\omega_i, \Re(\sigma_i),\Im(\sigma_i), (1,0,-1)\}$ are the orientations for \linebreak[4]$\Htilde(S_i,\Z)$ defined in Example \ref{exorientS}, for $i=1,2$, then $p_\alpha$ maps $\epsilon_{S_1}$ to $\epsilon_{S_2}$. Indeed, the classes $\omega_1$ and $\omega_2$  extend to flat sections of the same local system and the same holds for $\sigma_1$ and $\sigma_2$. As we will discuss in more detail in Section \ref{secmonAb}, $$\mon^2(S)\subseteq \O^+(\H^2(S,\Z))$$ (by \cite{Shi78}, see also \cite{MR19}), from which we get that $p_\alpha$ is constant along such local systems. This implies, together with the fact that the vector $(1,0, -1)$ is constant along any locally trivial deformation, that morphisms in the image of $\Phitildemkdef$ are always orientation preserving.
\end{enumerate}\end{rmk}

We shall now proceed to define the representation $\PhitildemkFM\colon \GmkFM \to \Htildemk$.
\begin{defn}Let $m,k\in \N\setminus\{0\}$, with $k>2$. We define the representation $$\PhitildemkFM\colon \GmkFM \longrightarrow \Htildemk$$
    is defined as follows: \begin{itemize}
        \item if $\SvH$ is an object in $\GmkFM$, then we set $\PhitildemkFM(\SvH):=(\Htilde(S,\Z),v,\epsilon_S)$;
        \item if $\SvHuno$ and $\SvHdue$ are two objects in $\GmkFM$ and $\phi\colon \Db(S_1) \to \Db(S_2)$ is a morphism in $\GmkFM$ from $\SvHuno$ to $\SvHdue$, then $$\PhitildemkFM(\phi):=\phi^\H$$ is the isometry induced by $\phi$ on the respective Mukai lattices.
    \end{itemize}
\end{defn}

\begin{rmk}\label{rmkimphitildeFM}As recalled in Proposition 5.44 and Corollary 9.43 of \cite{Huy03b}, any Fourier-Mukai equivalence between Abelian surfaces induces an isomorphism on the respective integral cohomologies, isometric with respect to the Mukai pairing, which is also parity preserving. Hence, if $\phi= \L$, $\FM_\PP$ or $\FM_\EE$, then it induces an isometry $\phi^\H$ on the respective Mukai lattices, preserving the Mukai vectors by (\ref{vtensor}), (\ref{vtilde}) and (\ref{vell}). The claim follows for $\phi=\FM_\PP\dual$, as derived duality induces, via the Chern character, an isomorphism on the even cohomology prescribed by (\ref{vduale}), and by (\ref{vcap}). Hence, the representation $\PhitildemkFM$ is again well defined.
\end{rmk}

\begin{rmk}\label{rmkFMnotorient}
    As we will explain in more detail in Section \ref{seclambda}, we remark that morphisms in the image of $\PhitildemkFM$ are not necessarily orientation preserving.
\end{rmk}

\begin{defn}Let $m,k\in \N\setminus\{0\}$, with $k>2$. We define the representation $$\Phitildemk\colon \Gmk \longrightarrow \Htildemk$$ as the unique (see \cite[Remark 2.31]{OPR23}) representation restricting to $\Phitildemkdef$ on $\Gmkdef$ and to $\PhitildemkFM$ on $\GmkFM$.\end{defn}
    
\subsubsection{The $\AAk-$representation $\ptmk$ of $\Gmk$}\label{secpt}We now turn our attention to lattices of the same isometry class of $\HdueKv$, where $\SvH$ is an $(m,k)-$triple, with $k>2$. We recall that, by Theorem \ref{pr20thm1.6}, there is an isometry $\HdueKv\simeq v\ort$. The latter, being the orthogonal complement of the lattice generated by $v=mw$, with $w^2=2k>0$, in $\HtildeSZ\simeq U^{\oplus 4}$, is isometric to the even lattice
$$U^{\oplus 3} \oplus \brakett{-2k}$$ of rank $7$ and signature $(3,4)$.

\begin{defn}
    For any $k>2$, we define the groupoid $\AAk$ as follows: \begin{itemize}
        \item the objects of $\AAk$ are even lattices $\Lambda$ of signature $(3,4)$, isometric to the lattice $U^{\oplus 3} \oplus \brakett{-2k}$.
        \item if $\Lambda_1$ and $\Lambda_2$ are two objects, then we set $$\Hom_{\AAk}(\Lambda_1,\Lambda_2):= \O(\Lambda_1,\Lambda_2).$$
    \end{itemize}
\end{defn}
As in the previous Section, we start by defining the $\AAk-$representation for $\Gmkdef$.
\begin{defn}Let $m,k\in \N\setminus\{0\}$, with $k>2$. We define the representation $$\ptmkdef\colon \Gmkdef \longrightarrow \AAk$$ as follows: \begin{itemize}
    \item if $\SvH$ is an object in $\Gmkdef$, then we set $$\ptmkdef(\SvH):= \HdueKv;$$
    \item if $\SvHuno$ and $\SvHdue$ are two objects in $\Gmkdef$ and $\alpha= \defopath$ is a deformation path from $\SvHuno$ to $\SvHdue$, then $$\ptmkdef(\overline{\alpha}):= g_\alpha$$ is the locally trivial parallel transport in the local system $R^2p\sstar\Z$ along the path $\gamma$ (as in (\ref{galpha}));
    \item if $\SvHuno$ and $\SvHdue$ are congruent, then we set $$\ptmkdef(\chi_{H_1,H_2}):= \id_{\HdueKv}.$$
\end{itemize}
\end{defn}
We now define the $\AAk-$representation for $\GmkFM$.
\begin{defn}Let $m,k\in \N\setminus\{0\}$, with $k>2$. We define the representation $$\ptmkFM\colon \GmkFM \longrightarrow \AAk$$ as follows: \begin{itemize}
    \item if $\SvH$ is an object in $\GmkFM$, then we set $$\ptmkFM(\SvH):= \HdueKv;$$
    \item if $\SvHuno$ and $\SvHdue$ are two objects in $\GmkFM$ and $\phi\colon \Db(S_1) \to \Db(S_2)$ is an elementary morphism from $\SvHuno$ to $\SvHdue$, then, by definition, it induces an isomorphism $$\phi_v\colon \KvSHuno \longrightarrow \KvSHdue$$ of moduli spaces. We then set $$\ptmkFM(\phi):= \phi_{v\ast}\colon \HdueKvuno \longrightarrow \HdueKvdue;$$
    \item if $\phi= \phi_1\ast \cdots \ast \phi_l\in \Hom_{\GmkFM}(\SvHuno,\SvHdue)$ is a composition of elementary morphisms, we define $\ptmkFM(\phi)$ as the composition $$\ptmkFM(\phi_l)\circ \cdots \circ \ptmkFM(\phi_1)$$ of the corresponding isometries.
\end{itemize}
\end{defn}

Combining the two previous Defintions, as in the case of $\Phitildemk$, we get the following

\begin{defn}
    Let $m,k\in \N\setminus\{0\}$, with $k>2$. We define the representation $$\ptmk\colon \Gmk \longrightarrow \AAk$$ as the unique representation restricting to $\ptmkdef$ on $\Gmkdef$ and to $\ptmkFM$ on $\GmkFM$.
\end{defn}
The first straightforward, but nonetheless crucial, property satisfied by the representation just defined is stated in the following 
\begin{prop}
    Let $m,k\in \N\setminus\{0\}$, with $k>2$ and let $A_1:= \SvHuno$ and $A_2:= \SvHdue$ be two objects in $\Gmk$. Then there is an inclusion of sets $$\ptmk(\Hom_{\Gmk}(A_1,A_2))\subseteq \PT_{\lt}(\KvSHuno,\KvSHdue).$$
\end{prop}
\begin{proof}
    First of all, the inclusion $$\ptmkdef(\Hom_{\Gmkdef}(A_1,A_2))\subseteq \PT_{\lt}(\KvSHuno,\KvSHdue)$$ holds by definition. Moreover, since any elementary morphism in $\GmkFM$ induces isomorphisms of moduli spaces,  by Proposition \ref{proppushforward}, $$\ptmkFM(\Hom_{\GmkFM}(A_1,A_2))\subseteq \PT_{\lt}(\KvSHuno,\KvSHdue)$$ holds, concluding the proof.
\end{proof}

According to the notation introduced in Definition \ref{defgroupoid} (1), we can therefore relate, via the representation $\ptmk$, the isotropy group of an $(m,k)-$triple with the locally trivial monodromy group of the moduli space associated to the same triple:
\begin{cor}\label{corimpt}
     Let $m,k\in \N\setminus\{0\}$, with $k>2$ and let $\SvH$ be an object in $\Gmk$. Then $$\Im(\ptmk_{\SvH}\colon \Aut_{\Gmk}(\SvH)\longrightarrow \Aut_{\AAk}(\H^{2}(\Kv(S,H),\Z))) \subseteq \monlt(\Kv(S,H)).$$
\end{cor}

\subsubsection{Relation between the two representations $\Phitildemk$ and $\ptmk$.}\label{seclambda} In order to compare the two representations \begin{center}
    $\Phitildemk\colon \Gmk \to \Htildemk$ and $\ptmk\colon \Gmk \to \AAk$ 
\end{center}we shall now connect them by means of a functor $\Psimk\colon \Htildemk \to \AAk$.

\begin{defn}\label{defpsimk}
    For any $m,k\in \N\setminus\{0\}$, with $k>2$, we define the functor $$\Psimk\colon \Htildemk \to \AAk$$ as follows: \begin{itemize}
        \item if $(\Lambdatilde, v, \epsilon)$ is an object in $\Htildemk$, we set $$\Psimk(\Lambdatilde, v, \epsilon):= v\ort;$$
        \item if $(\Lambdatilde_1, v_1, \epsilon_1)$ and $(\Lambdatilde_2, v_2, \epsilon_2)$ are two objects in $\Htildemk$ and $g\colon \Lambdatilde_1 \to \Lambdatilde_2$ is an isometry such that $g(v_1)=v_2$, then we set $$\Psimk(g):= (-1)^{\ori(g)}g_{|v_1\ort}\colon v_1\ort \longrightarrow v_2\ort,$$ where $\ori \colon \O(\Lambdatilde_1,\Lambdatilde_2)\to \Z/2\Z$ is the orientation character (\ref{orientmap}). 
    \end{itemize}
\end{defn}

We are finally in the position to compare the two $\AAk-$representations \begin{center}
    $\Phimk:= \Psimk \circ \Phitildemk\colon \Gmk \longrightarrow \AAk$
\end{center} and $\ptmk$. For this purpose, we will translate naturality of the isometries $$\lambda_{\SvH}\colon v\ort \longrightarrow \HdueKv$$ of Theorem \ref{pr20thm1.6} into the existence of an isomorphism of functors $$\lambda\colon \Phimk \longrightarrow \ptmk.$$ The proof of this result, which is the main of the Section, will be addressed in two parts, involving, respectively, the $\AAk-$representations of $\Gmkdef$ and of $\GmkFM$.

Let us start by setting $$\Phimkdef:= \Psimk\circ \Phitildemkdef \colon \Gmkdef \longrightarrow \AAk.$$
\begin{prop}\label{lambdadef}
    For any $m,k\in \N\setminus\{0\}$, with $k>2$, there exists an isomorphism of functors $$\lambda_{\defo}\colon \Phimkdef \longrightarrow \ptmkdef.$$
\end{prop}
\begin{proof}
    For $i=1,2$, let $A_i:=\SvHi$ be two objects in $\Gmkdef$. By Theorem \ref{pr20thm1.6}, for $i=1,2$, there exists an isometry $$\lambda_{A_i}\colon v_i\ort \longrightarrow \HdueKvi.$$ In order to show that the assignment $\lambda(A_i):= \lambda_{A_i}$ defines an isomorphism of functors $\Phimkdef \to \ptmkdef$, we need to show that the following diagram \begin{equation} \label{cdnat}
         \begin{tikzcd}
            v_{1}\ort \arrow[d,"\Phimkdef(h)" swap] \arrow[r,"\lambda_{A_{1}}"]& \H^{2}(K_{v_{1}}(S_{1},H_{1}),\Z)\arrow[d,"\ptmkdef(h)"]\\
            v_{2}\ort \arrow[r,"\lambda_{A_{2}}"]& \H^{2}(K_{v_{1}}(S_{2},H_{2}),\Z).
        \end{tikzcd}
    \end{equation} is commutative, for any morphism $h\in \Hom_{\Gmkdef}(A_1,A_2)$.

    \textit{Case 1.} Let $h=\overline{\alpha}$ be the elementary morphism given by the class of a deformation path $\alpha=\defopath$ from $A_1$ to $A_2$. In this case, $$\Phimk_{def}(\alpha)=\Psimk(\Phitildemk_{def}(\alpha))= (-1)^{\ori(p_{\alpha})}p_{\alpha}|_{v_{1}\ort}=p_{\alpha}|_{v_{1}\ort},$$ as the parallel transport operator $p_\alpha$ in the local system $R\ev f\sstar \Z$ (see (\ref{palpha}) is orientation preserving by Remark \ref{rmkimphitildedef} (b). On the other hand, if $p\colon \KK \to T$ is the relative moduli space induced by the deformation $(f\colon \SS \to T, \LL,\HH)$, then $$\ptmkdef(h)=g_\alpha$$ is the locally trivial parallel transport operator in the local system $R^2p\sstar \Z$ (see (\ref{galpha})). By Remark \ref{rmklocsystem} (1), there exists a flat section $\v$ of $R\ev f\sstar\Z$ such that $\v_{t_i}=v_i$ for $i=1,2$ and, by Remark \ref{rmklocsystem} (2), the isometries $\lambda_{A_i}$, for $i=1,2$, fit in an isomorphism of local systems $$\lambda_{\v}\colon \v\ort \longrightarrow R^2p\sstar\Z.$$ We deduce then that $g_\alpha$ and $p_{\alpha|v_1\ort}$ are locally trivial parallel transport operators along the same path inside local systems that are isomorphic via $\lambda_\v$, by Remark \ref{rmklocsystem} (3), from which commutativity of diagram (\ref{cdnat}) follows.

    \textit{Case 2.} Suppose that $A_1$ and $A_2$ are congruent and that $h=\chi_{H_1,H_2}$. In this case, commutativity of diagram (\ref{cdnat}) is automatic, as \begin{center}
        $\Phimk_{def}(\chi_{H_{1},H_{2}})=\id_{v\ort}$ and $\ptmk_{def}(\chi_{H_{1},H_{2}})=\id_{\H^{2}(\Kv(S,H_{1}),\Z)}$
    \end{center} and we have an identification $\H^{2}(\Kv(S,H_{1}),\Z) =\H^{2}(\Kv(S,H_{2}),\Z)$ and an identification $\lambda_{A_{1}}=\lambda_{A_{2}}$.
\end{proof}

We now set $$\PhimkFM:= \Psimk\circ \PhitildemkFM \colon \GmkFM \longrightarrow \AAk.$$
\begin{prop}\label{lambdaFM}
    For any $m,k\in \N\setminus\{0\}$, with $k>2$, there exists an isomorphism of functors $$\lambda_{\FM}\colon \PhimkFM \longrightarrow \ptmkFM.$$
\end{prop}

As already noticed in Remark \ref{rmkFMnotorient}, morphisms in the image of $\PhitildemkFM$ are not necessarily orientation preserving. In order to prove the commutativity of a diagram as (\ref{cdnat}) for elementary morphisms in $\GmkFM$, we need to introduce some criteria to decide whether an isometry induced by a Fourier-Mukai equivalence is orientation preserving. The following is an adaptation of \cite[Proposition 5.3]{HS05} to the case in which the surface $S$ is Abelian.

\begin{lem}\label{lemmaorient}
Let $S_1$ and $S_2$ be two Abelian surfaces and let $\phi\colon \Htilde(S_1,\Z) \to \Htilde(S_2,\Z)$ be a Hodge isometry. Let $\omega \in \H^{1,1}(S_1,\R)$ be a K\"ahler class on $S_1$ and let $\CC_{S_i}\subseteq \H^{1,1}(S_i,\R)$ be the connected component of the cone of positive forms containing the K\"ahler cone of $S_i$, for $i=1,2$. Set \begin{center}
    $r:= \phi(0,0,1)_0, \phantom{.} \chi := \phi(1,0,0)_0,   \phantom{.} \chi_\omega:= \phi(0, \omega, -\omega^2/2)_0 \in \H^0(S_2,\Z)$
\end{center}
   and set $u_0:= (-r,0,\chi)$, $u_1:= (0,-r\omega, r\omega^2/2 + \chi_\omega) \in \Htilde(S_1,\Z).$ 
   \begin{enumerate}
       \item[(a)] If $r\neq 0$, then $\phi$ is orientation preserving if and only if $$\Big(\frac{\chi_\omega}{r} + \frac{\omega^2}{2}\Big)\phi(u_0)_2 - \Big(\frac{\chi}{r} - \frac{\omega^2}{2}\Big)\phi(u_1)_2 \in \CC_{S_2};$$
       \item[(b)] if $r= 0$, then $\phi$ is orientation preserving if and only if $$\chi \phi(0, \omega, -\omega^2/2)_2 - \chi_\omega \phi(1,0,0)_2 - \frac{\omega^2}{2}(\phi(u_0)_2-\phi(u_1)_2)\in \CC_{S_2}.$$
   \end{enumerate}
\end{lem}
\begin{proof}We refer to \cite[Section 5]{HS05} for a detailed proof in the original case in which the surfaces involved are K3 surfaces and we sketch here the idea of the proof in the Abelian case, omitting computations. As $\phi$ is a Hodge isometry, arguing as in Remark \ref{exorientPSV2}, we get that the orientation of $\Htilde(S_1,\Z)$ defined by $\{\omega, \Re(\sigma),\Im(\sigma), (1,0,-1)\}$ - where $\sigma$ is a holomorphic symplectic form on $S_1$ - is preserved by $\phi$ if and only if the orientation of the positive plane $\brakett{\omega, (1,0,-1)}= \brakett{(1,\omega, -\omega^2/2)}$ is preserved. The latter is completely encoded in the complex line spanned by $\exp(i\omega)$, whose image via the $\C-$linear extension $\phi_\C$ of $\phi$ must be of the form $$\phi_\C(\exp(i\omega))=\lambda(\exp(b+ia)),$$ with $\lambda \in \C\Star$ and $a,b\in \H^{1,1}(S_2,\R)$. We write $\phi_\C(\exp(i\omega))$ as a $\Z[i]-$linear combination of $\phi(1,0,0)$, $\phi(0,0,1)$ and $\phi(0, \omega, -\omega^2/2)$ and we use it to compute the scalar $$\lambda=\phi_\C(\exp(i\omega))_0= \chi- \frac{\omega^2}{2}r+ i(\chi_\omega + \frac{\omega^2}{2}r)$$ and the real $(1,1)-$form \begin{align*}
    \abs{\lambda}^2a &= \Im(\overline{\lambda}(\exp(b+ia)))_2 = \Im(\overline{\lambda} \phi_\C(\exp(i\omega)))_2 =\\
    &= \Bigg\{\begin{array}{ll}
      \Big(\frac{\chi_\omega}{r} + \frac{\omega^2}{2}\Big)\phi(u_0)_2 - \Big(\frac{\chi}{r} - \frac{\omega^2}{2}\Big)\phi(u_1)_2   & \text{if }r\neq 0 \\
       \chi \phi(0, \omega, -\omega^2/2)_2 - \chi_\omega \phi(1,0,0)_2 - \frac{\omega^2}{2}(\phi(u_0)_2-\phi(u_1)_2)  & \text{if }r= 0.
    \end{array}
\end{align*}The latter is, up to multiplication with the positive scalar $\abs{\lambda}^2$, the class \linebreak[4]$a\in \H^{1,1}(S_2,\R)$, which determines the orientation given by $\{\Re(\lambda\exp(b+ia), \linebreak[4]\Im(\lambda\exp(b+ia)\}$, or, equivalently the orientation defined by the real and imaginary part of $\exp(ia)$. If $\omega'$ is a K\"ahler class on $S_2$, then the real and imaginary part of $\exp(ia)$ induce the same orientation as the natural one given by $\exp(i\omega')$ if and only if $a$ and $\omega'$ belong to the same connected component of the cone of positive classes in $\H^{1,1}(S_2,\R)$, from which the claim follows.    
\end{proof}

\begin{cor}\label{cororientFM}Let $\SvHuno$ and $\SvHdue$ two objects in $\GmkFM$, let $\phi\colon \Db(S_1) \to \Db(S_2)$ an elementary morphism and let $\phi^\H\colon \Htilde(S_1,\Z)\to \Htilde(S_2,\Z)$ be the induced isometry on the respective Mukai lattices. \begin{enumerate}
    \item If $\phi= \L, \FM_\PP$ or $\FM_\EE$, then $\ori(\phi^\H)=0$;
    \item If $\phi=\FM_\PP\dual$, then $\ori(\phi^\H)=1$.
\end{enumerate} 
    
\end{cor}
\begin{proof}
We recall that any Fourier-Mukai equivalence $\phi\colon \Db(S_1) \to \Db(S_2)$ induces an isomorphism of weight two Hodge structures as defined in (\ref{hodgestructure}) (see \cite[Proposition 5.39]{Huy03b}), so we can apply Lemma \ref{lemmaorient}.

If $\phi=\L$, $\SvHuno=\SvH$ and $\SvHdue=(S,v_L, H)$ then, for any $\rxia \in \Htilde(S,\Z)$, $$\L^\H\rxia= (r,\xi + \c_1(L), a+ \xi\c_1(L)+ r\c_1(L)^2/2).$$ We get $\L^\H(0,0,1)_0=0$, hence we need to check the condition of case (b) of Lemma \ref{lemmaorient}, which is shown equivalent to $\omega\in\CC_S$, where $\omega$ is a K\"ahler class.

If $\phi=\FM_\PP$, $\SvHuno=\SvH$ and $\SvHdue=(\hat{S}, \tilde{v}, \hat{H})$, then by (\ref{vtilde}) we get $\FM_\PP^\H(0,0,1)_0= 1$. Condition of case (a) turns out to be $(\omega^2/2)\hat{\omega}\in \CC_{\hat{S}}$, which is satisfied, as $\omega \in \CC_S$. Analogously, the Fourier-Mukai equivalence $\phi=\FM_\PP\dual$ falls into case (a) and produces the opposite class $-(\omega^2/2)\hat{\omega} \notin \CC_{\hat{S}}$.

Lastly, if $\phi=\FM_\EE$, $\SvHuno=\SvH$ and $\SvHdue=(S,\FM_\EE^\H(v),H)$, with $\NS(S)=\brakett{e,f}$ as in Lemma \ref{lemyosh}, by \cite[Theorem 5.3]{Bri98} we get \linebreak[4]$\phi^\H(0,0,1)_0=0$. We therefore proceed to check condition (b), by using \cite[Theorem 5.3]{Bri98} again and \cite[Section 3.2, (3.14)]{Yos01a} for Mukai vectors of positive rank, which equals to $$(\omega \cdot f)(e + l\frac{\omega^2}{2}f) \in \CC_S,$$ for a suitable $l>0$, which is satisfied, as $\omega \in \CC_S$.
\end{proof}

\begin{proof}[Proof of Proposition \ref{lambdaFM}] As in the previous case, for $i=1,2$, let $A_i:=\SvHi$ be two objects in $\Gmkdef$ and let us show that the following diagram \begin{equation} \label{cdnat2}
         \begin{tikzcd}
            v_{1}\ort \arrow[d,"\Phimk_{\FM}(\phi)" swap] \arrow[r,"\lambda_{A_{1}}"]& \H^{2}(K_{v_{1}}(S_{1},H_{1}),\Z)\arrow[d,"\ptmkFM(\phi)"]\\
            v_{2}\ort \arrow[r,"\lambda_{A_{2}}"]& \H^{2}(K_{v_{1}}(S_{2},H_{2}),\Z).
        \end{tikzcd}
    \end{equation} is commutative, for any elementary morphism $\phi\in \Hom_{\Gmkdef}(A_1,A_2)$, i.e. for $\phi=\L, \FM_\PP, \FM_\PP\dual$ and $\FM_\EE$. 
    A straightforward computation shows that the respective kernels of such Fourier-Mukai equivalences satisfy the hypotheses (2.1) and (2.2) of \cite[Section 2.1]{Yos01a}, namely they are all flat and strongly simple on both factors. Moreover, any $E \in \KvSHuno$ satisfies the $WIT_0$ property with respect to $\phi$ (see \cite[Definition 2.25, Lemma 2.28]{PR18} and \cite[Lemma 3.9]{Yos01a}). 

    If $\phi= \L, \FM_\PP$ or $\FM_\EE$, Proposition 2.4 of \cite{Yos01a} produces the equality $$\lambda_{A_2}^{-1}\circ \Phi_{v_1,\ast}\circ \lambda_{A_1}=\phi^\H_{|v_1\ort},$$ which is exactly the commutativity condition for diagram (\ref{cdnat2}), by Corollary \ref{cororientFM} (1). If $\phi=\FM_\PP\dual$, Proposition 2.5 of \cite{Yos01a} provides the identity $$\lambda_{A_2}^{-1}\circ \Phi_{v_1,\ast}\circ \lambda_{A_1}=-\phi^\H_{|v_1\ort},$$ which concludes the proof, as $\phi^\H$ is orientation reversing by Corollary \ref{cororientFM} (2). 
\end{proof}

By defining $$\lambda\colon \Phimk \longrightarrow \ptmk$$ as the unique natural transformation of functors restricting to $\lambdadef$ on $\Gmkdef$ and to $\lambda_{FM}$ on $\GmkFM$, Proposition \ref{lambdadef} and Proposition \ref{lambdaFM} yield the following

\begin{cor}\label{corlambda}
    For any $m,k\in \N\setminus\{0\}$, with $k>2$, there exists an isomorphism of functors $$\lambda\colon \Phimk \longrightarrow \ptmk.$$
\end{cor}

%% file: sec4.tex
\label{secmonodromy}In this Section we conclude the computation of the locally trivial monodromy group of moduli spaces of the form $\Kv(S,H)$, where $\SvH$ is  an $(m,k)-$triple, by showing that the injective morphism $$i^\sharp_{w,m}\colon \monlt(\KvSH) \longrightarrow \mon^2(\KwSH)$$ of Corollary \ref{corinjmon} is indeed surjective. In Section \ref{secsurfacemon} we describe locally trivial monodromy operators induced by monodromy operators of the Abelian surface $S$ as morphisms in the image of the representation $\Phitildemkdef$. In Section \ref{secsurj} we include the group $\Nn(v\ort)$ of Theorem \ref{thmmonkummer} in $\monlt(\KvSH)$ by showing that its generators belong to the image of the representation $\Phimk$ and using the isomorphism of functors $\lambda$ defined in Section \ref{seclambda}. In Section \ref{secmainresults} we show that the previous inclusion is indeed an equality, by using the morphism $i^\sharp_{w,m}$ as a constraint. From this, we deduce surjectivity of $i^\sharp_{w,m}$ too and we extend the result to the whole locally trivial deformation class of $\KvSH$, refining Corollary \ref{corinjmondefo}. 

\subsection{Locally trivial monodromy operators of surface type}\label{secsurfacemon}
In this section we show that monodromy operators in projective families of polarized Abelian surfaces can be lifted to some moduli spaces of sheaves $\KvSH$ on the same surface. We will translate the lifting process in terms of action of the representation $\Phitildemkdef$, providing an injective morphism $$\mon^2(S)\longrightarrow \monlt(\KvSH)$$ of groups, under the isomorphism of functors $\lambda$. The locally trivial monodromy operators in the image of this map will be called \textit{of surface type}.
\subsubsection{The monodromy group of Abelian surfaces}\label{secmonAb}
We start by recalling that, as a consequence of \cite[Theorem 1 and 2]{Shi78}, if $S$ is an Abelian surface, then \begin{equation}
    \label{monab}\mon^2(S)=\SO^+(\H^2(S,\Z)),
\end{equation}where the latter is the group of orientation preserving isometries of $\H^2(S,\Z)$ of determinant $1$ (see Appendix \ref{det} and \ref{rmkorient}). In \cite[Section 3]{MR19}, a finer result is proven, showing that the monodromy group of an Abelian surface can be generated by monodromy operators in a finite number of projective families:
\begin{thm}{(\cite[Corollary 3.4, Corollary 3.6]{MR19})}\label{thmmonab} Let $S$ be an elliptic Abelian surface with $\NS(S)\simeq U^{\oplus 2}$. There exist four ample classes $h_1,\dots, h_4 \in \NS(S)$ such that 
   \begin{equation}
       \label{monab2}\mon^2(S)= \brakett{\SO^+(\H^2(S,\Z))_{h_i}\colon i=0,\dots,4}.
   \end{equation} Moreover, for any $i=0,\dots, 4$, there exists a projective family $f_i\colon \colon \SS_i \to T_i$ of polarized Abelian surfaces such that \begin{equation}\label{monfi}\SO^+(\H^2(S,\Z))_{h_i}=\mon^2_{f_i}(S),\end{equation} where the latter is the group generated by monodromy operators in the family $f_i$.
\end{thm}

Theorem \ref{thmmonab} allows us to consider only monodromy operators in projective families, which is a requirement in order to induce deformations of moduli spaces of the form $\KvSH$, using the theory introduced in the previous Sections. 

For later use, we give some more details concerning the projective families $f_i\colon \SS_i \to T_i$ and the ample classes $h_i$, for $i=0,\dots,4$, involved in Theorem \ref{thmmonab}. For a more precise discussion, we refer to \cite[Section 3]{MR19}.

\begin{rmk}\label{rmkprojfamilies}
    Let $(S,h)$ be a \textit{polarized Abelian surface of degree $2d \in \N\setminus\{0\}$}, i.e. a pair made up of an Abelian surface $S$ and an indivisible class $h\in \H^2(S,\Z)\cap \H^{1,1}(S)$ represented by an ample line bundle $H$ on $S$ such that $H^2=2d$. Any such polarized Abelian surface can be embedded in a fixed projective space $\proj^N$ of dimension $N=9d-1$ by means of the linear system $\abs{3H}$ associated to the ample line bundle $3H$. As a consequence, there exists a Zariski open subset $T_{2d}$ of the Hilbert scheme parametrizing closed subschemes of $\proj^N$ with Hilbert Polynomial $P(x)=9dx^2$, such that, for any $t \in T_{2d}$, the corresponding subscheme $S_t\subseteq\proj^N$ is an Abelian surface equipped with an ample line bundle $H_t$ such that $(S_t, \c_1(H_t))$ is a polarized Abelian surface of degree $2d$. Furthermore, for any polarized Abelian surface $(S,h)$ of degree $2d$ there exists $t\in T_{2d}$ such that $(S,h)\simeq (S_t, \c_1(H_t))$. By base change over $T_{2d}$ of the universal family of the above mentioned Hilbert scheme, we get a smooth and proper family $$f_{2d}\colon \AA_{2d}\longrightarrow T_{2d}.$$ By \cite[Remark 3.1]{MR19}, the base $T_{2d}$ is a smooth, connected, quasi-projective variety of dimension $(N+1)^2+2$. Moreover, as the moduli space of polarized Abelian surfaces has dimension $3$ (see \cite[Remark 8.10.4]{BL04}) and by its algebraic construction (see \cite[Sections 8.7, 8.10]{BL04}), we get that the ample line bundles $H_t$, for $t\in T_{2d}$, glue together in a relatively ample line bundle $\HH_{2d}$ on $\AA_{2d}$.
\end{rmk}

In light of Remark \ref{rmkprojfamilies}, we can explicitly describe the families of identity (\ref{monfi}) in Theorem \ref{thmmonab} as $$f_i:= f_{2d_i}\colon \AA_{2d_i}\longrightarrow T_{2d_i},$$ where $H_{i}^2=2d_i$ for $i=0,\dots, 4$. 

In order to identify the ample classes $h_i\in \NS(S)$, let $S\overset{p}{\to} E$ be an elliptic Abelian surface with $\NS(S)\simeq \brakett{e,f}$, where $e$ and $f$ denote, respectively, the class of a \linebreak[4]$0-$section and the class of a fiber, as in Section \ref{secelliptic}. Let us consider three postive integers $r,s,p>0$ and define the following ample classes \begin{equation}\label{hi}
    h_1=e+rf, \phantom{+} h_2=e+(r-1)f, \phantom{+} h_3=se+pf, \phantom{+} h_4=(s-1)e+pf.
\end{equation}

Then, the same proof of \cite[Lemma 3.5]{MR19}, which is an application of Eichler's criterion (\cite[Lemma 3.5]{GHS10}
) by using Eichler's transvections (\cite[Lemma 3.2, Proposition 3.3 (iii)]{GHS09}
), shows that the classes $h_1, \dots, h_4$ are such that their stabilizers satisfy \begin{equation}\label{genso+}
    \SO^+(\H^2(S,\Z))= \brakett{\SO^+(\H^2(S,\Z))_{h_i}\colon i=0,\dots,4}.
\end{equation} Equalities (\ref{monab}) and (\ref{genso+}) together yield equality (\ref{monab2}) of Theorem \ref{thmmonab}. \begin{rmk}\label{rmkmonfi}
    We conclude by pointing out that, by \cite[Corollary 3.4]{MR19}, equality (\ref{monfi}) holds for any projective deformation $f\colon \SS \to T$ of any polarized Abelian surface $(S,h)$.
\end{rmk}
 
\subsubsection{Lift of the monodromy of an Abelian surface to moduli spaces of sheaves}We will now show that the monodromy operators induced by the projective families introduced in the previous Subsection can be lifted to locally trivial monodromy operators on a moduli space $\KvSH$ for a suitable $(m,k)-$triple $\SvH$. This will be proven in two special cases, which will be the only two needed in the next Section.
 \begin{rmk}
     We recall that the monodromy group $\mon^2(S)$ of an Abelian surface $S$ is, by definition, a subgroup of $\O(\H^2(S,\Z))$, and that the latter can be identified as the subgroup of $\O(\Htilde(S,\Z))$ of isometries acting as the identity on $\H^0(S,\Z) \oplus \H^4(S,\Z)$. Another remarkable subgroup of $\O(\Htilde{(S,\Z)})$ is given by the following: let $\SvH$ be an $(m,k)-$triple and let us consider the homomorphism of groups $$\PhitildemkdefSvH\colon \Aut_{\Gmkdef}\SvH \longrightarrow \Aut_{\Htildemk}(\Htilde(S,\Z),v,\epsilon_S)$$ induced by the action of the representation $\Phitildemkdef$ on the respective isotropy groups (see Definition \ref{defgroupoid} (1)). By Definition \ref{defHtildemk} we have an inclusion of groups $$\Im(\PhitildemkdefSvH)\subseteq \O(\Htilde(S,\Z))_v,$$ where the latter is the subgroup of isometries fixing the vector $v$.
     
     On the other hand, by Corollary \ref{corimpt}, the representation $\ptmkdefSvH$ provides us an inclusion of groups $$\Im(\ptmkdefSvH)\subseteq \monlt(\KvSH).$$ The interplay between the two representations above, studied in Section \ref{seclambda}, will allow us to compare the groups $\mon^2(S)$ and $\monlt(\KvSH)$.
 \end{rmk}

\begin{lem}\label{lemsurfacemon1} Let $S$ be an Abelian surface and let $v=(m,0,-mk)$ be a Mukai vector, with $m>1$ and $k>2$. Let $h\in \NS(S)$ be the class of a $v-$generic polarization on $S$ and  $f\colon \SS \to T$ a projective deformation of the polarized Abelian surface $(S,h)$ as in Remark \ref{rmkprojfamilies}. Then there exists a $v-$generic polarization $H$ on $S$ such that $h=\c_1(H)$ and $$\mon^2_f(S)\subseteq \Im(\PhitildemkdefSvH).$$
\end{lem}
\begin{proof} The claim will follow as soon as we show that the smooth and projective deformation $f\colon \SS \to T$ can be used to produce a deformation path centered in the $(m,k)-$triple $(S,v,H)$, with $H$ a suitable $v-$generic polarization. Notice that, as $v$ is of the form $(m,0,-mk)$, it shall be constant along the deformation.

By Remark \ref{rmkprojfamilies}, there exists an ample line bundle $\HH$ on $\SS$ such that, for every $t\in T$, its restriction $H_t:= \HH_t$ is an ample line bundle on $\SS_t$ such that $(\SS_t,\c_1(H_t))$ is a polarized Abelian surface of degree $h^2$. Moreover, there exists a point $t_0$ such that $(\SS_{t_0}, \c_1(H_{t_0}))\simeq (S,h)$. Set $H:= H_{t_0}$ and observe that, as the notion of $v-$genericity of a polarization only depends on its first Chern class, the ample line bundle $H$ is $v-$generic. Hence, by Remark \ref{vgenopen}, the points $t\in T$ such that $H_t$ is not $v-$generic define a Zariski closed proper subset $Z$ of $T$. As in the proof of \cite[Lemma 3.6]{OPR23}, we get that for any loop $\gamma$ in $T$ centered in $t_0$, there exists a loop $\gamma'$ in $T':=T\setminus Z$ centered in $t_0$ and homotopic to $\gamma$. Hence, up to replacing $T$ with $T'$, we can assume that $\HH_t$ is $v-$generic for every $t \in T$.  

Consequently, if $g\in \mon^2(S)$ is the monodromy operator in the family $f$ along a path $\gamma$ centered in $t_0$, then the deformation path $$\alpha=(f\colon \SS \to T, \Os, \HH, t_0,t_0,\gamma)$$ defines a a morphism $\overline{\alpha}$ in $\Gmkdef$ such that $(\id, g, \id) =\PhitildemkdefSvH(\overline{\alpha})$, concluding the proof.\end{proof}
For later use, we provide an analogue of Lemma \ref{lemsurfacemon1} for another special class of $(m,k)-$triples.
\begin{lem}\label{lemsurfacemon2} Let $(S,h)$ be a polarized Abelian surface such that $\NS(S)=\Z h$, let $v=(r,mh,a)$ be a Mukai vector, and let $f\colon \SS \to T$ be a projective deformation of $(S,h)$ as in Remark \ref{rmkprojfamilies}.  Then there exists a $v-$generic polarization $H$ on $S$ such that $h=\c_1(H)$ and  $$\mon^2_f(S)\subseteq \Im(\PhitildemkdefSvH).$$
\end{lem}

\begin{proof}
    The proof runs exactly as the proof of Lemma \ref{lemsurfacemon1}, by using the deformation path 
        $\alpha=(f\colon \SS \to T, \HH^{\otimes m}, \HH, \gamma, t_0,t_0)$.
\end{proof}
 \begin{cor}\label{corsurfacemon}
     Let $\SvH$ be an $(m,k)-$triple and $f\colon \SS \to T$ a projective family as in Lemma \ref{lemsurfacemon1} or in Lemma \ref{lemsurfacemon2}. Then there exists an injective morphism $$\mon^2_f(S) \hookrightarrow \Im(\ptmkdefSvH)\subseteq \monlt(\KvSH).$$
 \end{cor}
 \begin{proof} Lemma \ref{lemsurfacemon1} and Lemma \ref{lemsurfacemon2} provide an inclusion of groups $$\mon^2_f(S)\subseteq \Im(\PhitildemkdefSvH).$$ As the latter is made of orientation preserving isometries by Remark \ref{rmkimphitildedef} (2), the morphism $\Psimk$ defined in \ref{defpsimk} provides an injective morphism \begin{align*}
     \mon^2_f(S) & \hookrightarrow \Im(\PhimkdefSvH)\\
     g & \mapsto (\id, g, \id)_{|v\ort}.
 \end{align*}The claim follows by applying the isomorphism of functors $\lambdadef \colon \Phimkdef \to \ptmkdef$ defined in Proposition \ref{lambdadef}.\end{proof}
We conclude the Section by showing that, in the elliptic case, the classes $h_1,\dots, h_4$ defined in (\ref{hi}) and generating $\mon^2(S)$ (see Theorem \ref{thmmonab}) can be chosen to belong to the same $v-$chamber. As a consequence, we get that only one $(m,k)-$triple $\SvH$ is needed to generate the whole monodromy group $\mon^2(S)$ using morphisms in $\Im(\PhitildemkdefSvH)$. 

In the following, we will assume that $S$ is an elliptic Abelian surface with $\NS(S)\simeq \brakett{e,f}$, where $e$ and $f$ are, respectively, the class of a $0-$section and the class of a fiber, as in Section \ref{secelliptic}.
\begin{prop}\label{propsurfacemon}
    Let $S$ be an elliptic Abelian surface with $\NS(S)\simeq \brakett{e,f}$ and let $v=(m,0,-mk)$. Then there exists an integer $t\gg 0$ and a $v-$generic polarization $H$ such that $\c_1(H)=e+tf$ and $$\mon^2(S) \subseteq \Im(\PhitildemkdefSvH).$$
\end{prop}
\begin{proof}By Theorem \ref{thmmonab}, there exist four ample classes $h_1,\dots, h_4$ on $S$ such that $$\mon^2(S)=\brakett{\mon^2_{f_1}(S),\dots, \mon^2_{f_4}(S)},$$ where $f_i:= f_{h_i^2}$ are projective deformations of the polarized Abelian surfaces $(S,h_i)$, for $i=1,\dots, 4$, as in Remark \ref{rmkprojfamilies}, and by Lemma \ref{lemsurfacemon1}, for any $i=1,\dots, 4$, we have an inclusion of groups $$\mon^2_{f_i}(S)\subseteq \Im(\Phitildemk_{\defo, (S,v,H_i)}),$$ where $H_i$ is a $v-$generic polarization such that $\c_1(H_i)=h_i$. As in (\ref{hi}), we can choose three positive integers $r,s,p>0$ such that $$
    h_1=e+rf, \phantom{+} h_2=e+(r-1)f, \phantom{+} h_3=se+pf, \phantom{+} h_4=(s-1)e+pf.$$ By Lemma 2.38 (see also Definition 2.37) of \cite{PR18}, if we choose those integers such that $r$ and the quotient $p/s$ are big enough, then all the classes $h_i$ belong to the unique $v-$chamber whose closure contains the class $f$. We can then choose a polarization $H$ on $S$ such that its class $h:=\c_1(H)$ lies in that $v-$chamber. Consequently, the $(m,k)-$triple $\SvH$ is congruent to $(S,v,H_i)$ for every $i=1,\dots,4$ and we get a collection of isomorphisms $$\chi_{H,H_i}^\sharp\colon \Aut_{\Gmkdef}\SvH \longrightarrow \Aut_{\Gmkdef}(S,v,H_i)$$ that induce, via $\Phitildemkdef$, the following identifications $$\Im(\Phitildemk_{\defo,(S,v,H_i)})= \Phitildemkdef(\chi_{H,H_i})^\sharp(\Im(\PhitildemkdefSvH))= \Im(\PhitildemkdefSvH),$$ where the last equality follows from the fact that $\Phitildemkdef(\chi_{H,H_i})$ is the identity morphism.
\end{proof}
By combining Proposition \ref{propsurfacemon} and Corollary \ref{corsurfacemon}, we get the following
\begin{cor}
    Let $S$ be an elliptic Abelian surface with $\NS(S)\simeq \brakett{e,f}$ and let $v=(m,0,-mk)$. If $H$ is a $v-$generic polarization on $S$ whose class is contained in the unique $v-$chamber whose closure contains the class $f$, then there is an injective morphism of groups $$\mon^2(S)\hookrightarrow \monlt(\KvSH).$$
\end{cor}

\subsection{The group $\Nn(v\ort)$ as subgroup of $\monlt(K_v(S,H))$}\label{secsurj} 
We start by recalling that, by Theorem \ref{thmmonkummer}, if $(S,w,H)$ is a $(1,k)-$triple with $k>2$, then $$\mon^2(\KwSH)= \Nn(\KwSH)\simeq \Nn(w\ort),$$ where $\Nn(\KwSH)$ is the index $2$ subgroup of the Weyl group of reflections defined in Appendix \ref{weylgroups} and the last isomorphism is given by conjugation via the isometry $\lambda_{(S,w,H)}$ of Theorem \ref{pr20thm1.6}. The main result of this section is the following
\begin{thm} \label{thmNmon} Let $\SvH$ be an $(m,k)-$triple, with $k>2$. Then $$\Nn(\KvSH)\subseteq \monlt(\KvSH).$$
\end{thm}

The proof of Theorem \ref{thmNmon} will be addressed by showing that the generators of the group $\Nn(\KvSH)$ are isometries belonging to the image of the group morphism $\ptmk_{\SvH}\colon \Aut_{\Gmk}\SvH\to \monlt(\KvSH)$ (see Corollary \ref{corimpt}), induced by the representation $\ptmk$ defined in Section \ref{secpt}, for a suitable $(m,k)-$triple $\SvH$. As the Hodge isometry $\lambda_{\SvH}\colon v\ort \to \H^2(\KvSH)$ naturally conjugates $\Nn(v\ort)$ to $\Nn(\KvSH)$, the proof of Theorem \ref{thmNmon} reduces to showing the following inclusion of groups: \begin{equation}
    \label{Nnimphi} \Nn(v\ort)\subseteq \Im(\Phimk_{\SvH}).
\end{equation} 
Paralleling Markman's proof of \cite[Theorem 1.4]{Mar18} in the primitive case, up to conjugation with a morphism $\Phimk(\eta)$, for any $\eta\in \Hom_{\Gmk}(\SvH, \SvHuno)$ and $\SvHuno\in \Gmk$, we can assume without loss of generality that $\SvH\in \Gmk$ is an $(m,k)-$triple with Mukai vector $v=m(1,0,-k)$. 
\begin{rmk}\label{rmkgenN2}
    As recalled in Remark \ref{rmkgenN} (2), the group $\Nn(v\ort)$ is generated by $\SO^+(\Htilde(S,\Z))_v$ and by the involution $R_{s}\circ R_{s_1}\colon v\ort \to v\ort$ given by the composition of the two reflections around the $(-2)-$vector $s=(1,0,1)$ and the $(+2)-$vector $s_1=(1,0,-1)$, which sends the vector $m(1,0,k)$ to its opposite and acts as the identity on $\lambda_{\SvH}(\H^2(S,\Z))$. More precisely, as its action on the discriminant $A_{v\ort}$ is trivial (see Corollary \ref{corextid}), the latter corresponds to the restriction of the product $R_s\circ R_{s_1}\colon \Htilde(S,\Z)\to \Htilde(S,\Z)$ of the two above-defined reflections on $\Htilde(S,\Z)$, whose action is precisely $-D_{S}^\H$, where $D_S^\H$ is the cohomological action of duality  defined in (\ref{vduale}), and which belongs to $\O(\Htilde(S,\Z))_v$ by definition.
\end{rmk}
Inclusion (\ref{Nnimphi}) will be proven on each of the generators described in Remark \ref{rmkgenN2}. The first and most technical part is given by the following:
\begin{prop}\label{propso+v}Let $m\geq 1$ and $k>2$ two positive integers and let $\SvH$ an $(m,k)-$triple, with $S$ an elliptic Abelian surface, $v=m(1,0,-k)$ and $h:=\c_1(H)=e+tf$, with $t\gg 0$, as in Proposition \ref{propsurfacemon}. Then $$\SO^+(\Htilde(S,\Z))_v\subseteq \Im(\Phimk_{\SvH}).$$
\end{prop}

For this purpose, we introduce the following result:
\begin{prop}
    \label{lemrefl2}Let $S$ be an elliptic Abelian surface with $\NS(S)\simeq\brakett{e,f}$ and let $v=m(1,0,-k)\in \Htilde(S,\Z)$ a Mukai vector. The stabilizer $\SO^+(\Htilde(S,\Z))_v$ is generated by $\SO^+(\H^2(S,\Z))$ and the products $R_{t_1}\circ R_{t_2}$ of reflections around $(-2)-$vectors \linebreak[4] $t_i=(1,\beta_i-f,k)$, with $\beta_i\in \NS(S)\ort$ a primitive class, for $i=1,2$.
\end{prop}

The starting point for the proof of Proposition \ref{lemrefl2} is the following 
\begin{lem}{\cite[Lemma 5.3 (1), (3), Lemma 5.4]{Mar18}}\label{lemrefl1} Let $S$ be an Abelian surface and let $v=m(1,0,-k)\in \Htilde(S,\Z)$ a Mukai vector. The stabilizer $\SO^+(\Htilde(S,\Z))_v$ is generated by $\SO^+(\H^2(S,\Z))$ and the products $R_{t_1}\circ R_{t_2}$ of reflections around $(-2)-$vectors $u_i=(1,\xi_i,k)$, with $\xi_i\in \H^2(S,\Z)$ a primitive class, for $i=1,2$.
\end{lem}

The key technical tool used to refine Lemma \ref{lemrefl1} is the following result, whose proof was suggested by Simone Billi, which I thank again for the precious help.

\begin{lem}\label{lemsimo}
    Let $S$ be an elliptic Abelian surface with $\NS(S)=\brakett{e,f}$ and let $\xi_1, \xi_2\in \H^2(S,\Z)$ be two primitive elements such that $\xi_1^2=\xi_2^2=2k-2$. Then there exist two primitive elements $\beta_1, \beta_2 \in \NS(S)\ort$ and an isometry $g \in \SO^+(\H^2(S,\Z))$ such that $g(\xi_i)=\beta_i-f$, for $i=1,2$.
\end{lem}

The proof of Lemma \ref{lemsimo} will be subdivided in two steps, namely Lemma \ref{lemlattice1} and Lemma \ref{lemlattice2}. In the following, we will let $S$ be an elliptic Abelian surface with $\NS(S)=\brakett{e,f} =:U_1$ and we will set $$L:= \H^2(S,\Z)\simeq U^{\oplus 3}=U_1\overset{\perp}{\oplus}U^{\oplus 2}.$$ Let $\xi_1, \xi_2\in L$ be two primitive elements such that $\xi_1^2=\xi_2^2=2k-2$, $\xi_2\neq \pm\xi_1$ and let us denote by $S_1:= \overline{<\xi_1,\xi_2>}$ the saturation of $\brakett{\xi_1,\xi_2}$ into $L$ (see Appendix \ref{secprimitiveemb}).

\begin{lem}\label{lemlattice1} There exist two primitive elements $\beta_1, \beta_2 \in U_1\ort$, a primitive sublattice $S_2$ of $L$ such that $\beta_1,\beta_2 \in S_2$ and an isometry $\pphi\colon S_1 \to S_2$ such that $\pphi(\xi_i)=\beta_i-f$ for $i=1,2$.
\end{lem}
\begin{proof} Let us set $\xi_1\cdot \xi_2=: l \in \Z$ and let $e_2,f_2,e_3,f_3 \in L$ be the isotropic generators of the two orthogonal copies $U^{\oplus2}$ of the hyperbolic plane, orthogonal to $U_1$, so that \begin{center}
    $e_i^2=f_i^2=e\cdot e_i=f\cdot e_i = f\cdot e_i=f\cdot f_i =0$ and $e_i\cdot f_i=1$ for $i=2,3$.
\end{center}
Let us set $$\beta_1:=e_2 + (k-1)f_2\textit{, }\phantom{+}\beta_2:= l f_2+ (k-1)e_3 + f_3 \in L$$ and notice that $\beta_1,\beta_2 \in U_1\ort$, $\beta_i^2=\xi_i^2=2k-2$ and $\beta_1\cdot \beta_2=l=\xi_1\cdot \xi_2$.  Let $x_i=\lambda_i \xi_1 + \mu_i \xi_2$, with $\lambda_i, \mu_i \in \Q$, for $i=1,2$, be two generators of $S_1$, set $$y_i:= \lambda_i \beta_1 + \mu_i \beta_2 - (\lambda_i+\mu_i)f\in L,$$ for $i=1,2$, and let $S_2:= <y_1,y_2>$. Then, a direct computation shows that the map \begin{align*}
    \pphi\colon S_1 &\longrightarrow S_2\\
    x_i & \longmapsto y_i
\end{align*} is an isometry between primitive sublattices of $L$ satisfying $\pphi(\xi_i)=\beta_i-f$ for $i=1,2$. In particular, the lattice $S_2$ is the saturation of $\brakett{\beta_1-f,\beta_2-f}$ inside $L$.\end{proof}

In particular, both $S_1$ and $S_2$ are abstractly isometric to a primitive sublattice $S\subseteq L$, which satisfies the following properties:
\begin{lem}\label{lemlattice2}Let $S$ be a primitive sublattice of $L$ as above and let $K:= S\ort$ its orthogonal complement in $L$. Then the following holds: \begin{enumerate}
    \item There exists an isomorphism of groups $\gamma\colon A_S \simeq A_K$ such that $\overline{q_S}=-\overline{q_K}\circ \gamma$.
    \item The lattice $K$ is unique up to isomorphism and the natural  map $\disc\colon \O(K) \to \O(A_K)$ is surjective.
    \item In particular, $K \simeq U \oplus S(-1)$, where $(S(-1),q_{S(-1)}):=(S,-q_{S})$.
\end{enumerate}
\end{lem}
\begin{proof}
    Part (1) is a straightforward application of Theorem \ref{thmantiso}, with $\gamma:=\gamma_{ST}$. Part (2) follows from Proposition \ref{propunique}. Indeed, as $K$ is a rank $4$ sublattice of $L$ and $sgn(L)=(3,3)$, it is indefinite. Moreover, by part (1), it holds $l(A_K)=l(A_S)\leq 2$, hence $\rk(K)=4 \leq l(A_K)+2$. Part (3) follows from the uniqueness up to isomorphism of $K$ shown in part (2), by noticing that $sgn(K)=sgn(U \oplus S(-1))$ and $$\overline{q_{U\oplus S(-1)}}=\overline{q_{S(-1)}}=-\overline{q_S}=\overline{q_K},$$ where the last equality follows from part (1).
\end{proof}
The two previous Lemmas allow us to address the proof of Lemma \ref{lemsimo}.
\begin{proof}[Proof of Lemma \ref{lemsimo}] With the notations above introduced, we consider the isometry $\pphi\colon S_1\to S_2$, whose existence is guaranteed by Lemma \ref{lemlattice1} and which satisfies $\pphi(\xi_i)=\beta_i-f$ for two suitable elements $\beta_i\in U_1\ort$, for $i=1,2$.

\textit{Extension of $\pphi$ to $\Tilde{\pphi}\in \O(L)$.} Let us denote by $K_i:=S_i\ort$, respectively, the orthogonal complements of $S_i$ in $L$ for $i=1,2$. By Lemma \ref{lemlattice2} (2), the latter are both isomorphic to $K$ and, by Lemma \ref{lemlattice2} (1), there exist two isomorphisms $\gamma_i\colon A_{S_i}\to A_{K_i}$ reversing the disriminant quadratic form, for $i=1,2$. Let us consider the isomorphism $$\overline{\psi}:= \gamma_2\circ \overline{\pphi}\circ \gamma_1^{-1}\colon A_{K_1}\to A_{K_2},$$ where $\overline{\pphi}:= \disc(\pphi)$ is induced by $\pphi$ on the discriminant groups. Notice that $\overline{\psi}\in \O(A_{K_1},A_{K_2})$, as $\overline{\pphi}$ preserves the discriminant form and both $\gamma_1$ and $\gamma_2$ reverse it. By Lemma \ref{lemlattice2} (2) we deduce the surjectivity of the dicriminant map $\disc\colon \O(K_1,K_2)\to \O(A_{K_1},A_{K_2})$, so that there exists an isometry $\psi \in \O(K_1,K_2)$ such that $\disc(\psi)=\overline{\psi}$. Consequently, by Proposition \ref{propext}, there exists an isometry $\Tilde{\pphi}\in \O(L)$ such that $\Tilde{\pphi}|_{S_1}=\pphi$.

\textit{Determinant:} By Lemma \ref{lemlattice2} (3), it holds $K_2\simeq U\oplus S(-1)$. We consider the isometry $\eta\in \O(K_2)$ interchanging the two isotropic generators of $U$ and acting as the identity on $S(-1)$. We notice that $\det(\eta)=-1$ and $\disc(\eta)=\id_{A_{K_2}}$. Hence, by Corollary \ref{corextid}, we can extend $\eta$ to $\Tilde{\eta}:=\eta \oplus \id_{S_2}\in \O(L)$ satisfying $\det(\Tilde{\eta})=\det(\eta)=-1$. Consequently, if $\det(\Tilde{\pphi})=-1$, we can replace $\Tilde{\pphi}$ with $\Tilde{\eta}\circ \Tilde{\pphi} \in \SO(L)$.

\textit{Orientation:} In conclusion, we can work in a similar fashion to make $\Tilde{\pphi}$ orientation preserving. Again by using the characterization of Lemma \ref{lemlattice2} (3), we can define an isometry $\theta:= -\id_{U}\oplus \id_{S(-1)}\in \O(K_2)$ such that $\ori(\theta)=1$, as $sgn(U)=(1,1)$, $\det(\theta)=1$, as $\rk(U)=2$, and $\disc(\theta)=\id_{A_{K_2}}$. By Corollary \ref{corextid}, we can extend $\theta$ to $\Tilde{\theta}:= \theta \oplus \id_{S_2}$ such that $\ori(\Tilde{\theta})=\ori(\theta)=1$ and $\det(\Tilde{\theta})=\det(\theta)=1$. Hence, the claim follows by setting $g:=\Tilde{\theta}^{\ori(\Tilde{\pphi})}\circ \Tilde{\pphi} \in \SO^+(L)$. \end{proof}

The proof of Proposition \ref{lemrefl2} now reduces to a combined application of Lemma \ref{lemrefl1} and Lemma \ref{lemsimo}.

\begin{proof}[Proof of Proposition \ref{lemrefl2}]In order to prove the claim, we need to show that any product $R_{u_1} \circ R_{u_2}$ of reflections around $(-2)-$vectors $u_i=(1,\xi_i,k)$, with $\xi_i\in \H^2(S,\Z)$ a primitive class, for $i=1,2$, as in Lemma \ref{lemrefl1} can be conjugated into the product $R_{t_1}\circ R_{t_2}$ of reflections around $(-2)-$vectors $t_i=(1,\beta_i-f,k)$, with $\beta_i\in \NS(S)\ort$ a primitive class, for $i=1,2$, as in Proposition \ref{lemrefl2}, via an isometry $\Tilde{g}\in SO^+(\Htilde(S,\Z))_v$ defined as $\tilde{g}:=\id_{\H^0(S,\Z)}\oplus g \oplus \id_{\H^4(S,\Z)}$, with $g\in \SO^+(\H^2(S,\Z))$.

Let $u_i=(1,\xi_i,k)$, with $\xi_i\in \H^2(S,\Z)$ primitive, for $i=1,2$, be two $(-2)-$vectors as above, so that $\xi_i^2=2k-2$ for $i=1,2$. By Lemma \ref{lemsimo}, there exist two elements $\beta_1,\beta_2 \in \NS(S)\ort$ and an isometry $g\in \SO^+(\H^2(S,\Z))$ such that $g(\xi_i)=\beta_i-f$. By extending $g$ to $\tilde{g} \in SO^+(\Htilde(S,\Z))_v$ via the identity as above, we get $\tilde{g}(u_i)=t_i$ for $i=1,2$, and a straightforward application of the identity  \begin{equation}\label{rifless}
        R_{\tilde{g}(u)}=\tilde{g}\circ R_u \circ \tilde{g}^{-1} \phantom{++} \text{for any }u\in \Htilde(S,\Z)
    \end{equation} leads to $$\tilde{g}\circ R_{u_1}\circ R_{u_2} \circ \tilde{g}^{-1} = R_{\tilde{g}(u_1)}\circ R_{\tilde{g}(u_2)}= R_{t_1}\circ R_{t_2},$$ concluding the proof. \end{proof}
    
The last result allows us to reduce the first part of the proof of Proposition \ref{propso+v} to the following
\begin{lem}\label{lemrefl3}
    Let $\SvH$ be an $(m,k)-$triple as above, with $\c_1(H)=e+tf$, with $t\gg0$, and let $t_i=(1,\xi_i,k) \in \Htilde(S,\Z)$ be a $(-2)-$vector such that $\xi_i=\beta_i -f$, with $\beta_i\in \brakett{e,f}\ort$ for $i=1,2$. Then $$R_{t_1}\circ R_{t_2}\in \Im(\Phitildemk_{\SvH}).$$ 
\end{lem}
\begin{proof}Under the current assumptions, Lemma \ref{lemyosh} guarantees that the derived equivalence $FM_\EE$ introduced in Section \ref{secelliptic} induces a morphism $$\FM_{\EE} \in \Hom_{\GmkFM}(\SvH, (S,\overline{v},H)),$$ where we set $\overline{v}:= (0, m(e+kf),m)$. Arguing verbatim as in the proof of \cite[Proposition 4.9]{MR19}, we get a description of the cohomological action of $\FM_\EE$ on the whole Mukai lattice $\Htilde(S,\Z)$. Consequently, for $i=1,2$, it holds $$\FM_\EE^\H(t_i)=\FM_\EE^\H(1,\beta_i-f,k)=(0,e-kf-\beta_i,0).$$ Setting $a_i:= e- kf-\beta_i$, for $i=1,2$,  by identity (\ref{rifless}) we get $$R_{t_1}\circ R_{t_2}=(\FM_\EE^{H})^{-1} \circ R_{(0,a_1,0)}\circ R_{(0,a_2,0)}\circ \FM_\EE^\H.$$ Hence, we can reduce the proof to showing that \begin{equation}\label{claim1}
    R_{(0,a_1,0)}\circ R_{(0,a_2,0)} \in \Im(\Phitildemk_{(S,\overline{v},H)}).
\end{equation}Let us consider a deformation path $\alpha=\defopath$, where \begin{center}
    $(\SS_{t_1},\LL_{t_1},\HH_{t_1})=(S,m(e+kf),H)$ and $(\SS_{t_2},\LL_{t_2},\HH_{t_2})=(S', mh', H')$, 
\end{center}with $H'\in \Pic(S')$, $h'=\c_1(H')$ and $\NS(S')=\Z h'$, and set $v':= v_{t_2}$. Then, up to conjugation with $\Phitildemkdef(\overline{\alpha})$, the statement in (\ref{claim1}) is equivalent to \begin{equation}
    \label{claim2} R_{(0,b_1,0)}\circ R_{(0,b_2,0)} \in \Im(\Phitildemk_{(S',v',H')}),
\end{equation} where $(0,b_i,0)=\Phitildemkdef(\overline{\alpha})(0,a_i,0)$ - as $\Phitildemkdef(\overline{\alpha})$ is a graded isomorphism of groups by definition - for $i=1,2$. Moreover, the latter satisfies the following properties: \begin{enumerate}
    \item $b_i^2=a_i^2=t_i^2=-2$;
    \item $b_i \cdot h'= t_i\cdot v=0$
\end{enumerate} for $i=1,2$. Hence, by identifying $R_{(0,b_i,0)}$ with the reflection $R_{b_i}$ in $\H^2(S',\Z)$ under the natural embedding  $\O(\H^2(S',\Z))\hookrightarrow \O(\Htilde(S',\Z))$, we get that $R_{b_1}\circ R_{b_2}\in \SO^+(\H^{2}(S',\Z))$, by point (1) and equality (\ref{char}), and that $$R_{b_1}\circ R_{b_2}\in \SO^+(\H^{2}(S',\Z))_{h'},$$ by point (2). Hence, by Remark \ref{rmkmonfi} and Lemma \ref{lemsurfacemon2}, we deduce $$R_{b_1}\circ R_{b_2} \in \Im(\Phitildemk_{\defo,(S',v',H')}),$$ proving claim (\ref{claim2}) and concluding the proof.
\end{proof}
We are finally in the position to quickly address the proof of Proposition \ref{propso+v}.
\begin{proof}[Proof of Proposition \ref{propso+v}]
    In order to show that $\SO^+(\Htilde(S,\Z))_v\subseteq \Im(\Phimk_{\SvH})$, by Lemma \ref{lemrefl1} and Proposition \ref{lemrefl2} it is enough to prove the inclusion $\SO^+(\H^2(S,\Z))\linebreak[4]\subseteq \Im(\Phimk_{\SvH})$ and that $R_{t_1}\circ R_{t_2} \in \Im(\Phimk_{\SvH})$ for any $t_i=(1,\beta_i-f,k)$, with $\beta_i\in \brakett{e,f}\ort$, as in Proposition \ref{lemrefl2}. The first assertion follows from Proposition \ref{propsurfacemon} and the fact that the representation $\Psimk$ acts trivially - namely, just as restriction to $v\ort$ - on orientation preserving isometries. The second part of the claim follows analogously from Lemma \ref{lemrefl3}. More precisely, the latter shows that the product of such reflections can be written as a composition of orientation preserving isometries of the form $\Phitildemk(\eta)$ for a suitable morphism $\eta \in \Hom_{\Gmk}(\SvHuno,\SvHdue)$, so the representation $\Psimk$ acts trivially on $R_{t_1}\circ R_{t_2}$ again, providing an element of $\Im(\Phimk_{\SvH})$.
\end{proof}

The last ingredient needed to complete the proof of Theorem \ref{thmNmon} is the following:
\begin{prop}\label{propdual}
    Let $m\geq 1$ and $k>2$ be two positive integers and let $\SvH$ be an $(m,k)-$triple, with $v=m(1,0,-k)$. Then $$R_s\circ R_{s_1}\in \Im(\PhimkSvH),$$ where $s=(1,0,1),\phantom{.} s_1=(1,0,-1)\in v\ort$ are as in Remark \ref{rmkgenN2}.
\end{prop}
\begin{proof}
    Let us start by considering the reflections $R_s$ and $R_{s_1}$ around the vectors $s=(1,0,1)$ and $s_1=(1,0,-1)$, respectively, defined on the whole Mukai lattice $\Htilde(S,\Z)$. As recalled in Remark \ref{rmkgenN2}, their composition $R_s\circ R_{s_1}$ is an involution satisfying \begin{equation}
        \label{-D_sH} R_s\circ R_{s_1}=-D_S^\H.
    \end{equation}The proof will be easily concluded as soon as we show that $D_S^\H \in \Im(\PhitildemkSvH)$. Firstly, up to conjugation with a morphism $\Phitildemkdef(\overline{\alpha})$, where $$\alpha=(f\colon \SS \to T, \OO_{\SS}, \HH, \gamma, t_1,t_2)$$ is a deformation path from $\SvH$ to $(S', v, H')$, such that $\NS(S')=\Z h'$, with $h'=\c_1(H')$, we can assume that $\NS(S)=\Z h$, where $h=\c_1(H)$. Indeed, if $\psi\in \Aut_{\Gmk}(S',v,H')$ is such that $\Phitildemk(\psi)=D_{S'}^\H$, namely, by (\ref{vduale}), for any $(r,\xi,a)\in \Htilde(S',\Z)$, \begin{equation}
        \label{dual}\Phitildemk(\psi)(r,\xi,a)=D_{S'}^\H(r,\xi,a)=(r,-\xi,a),
    \end{equation} then the latter must satisfy the following commutation relation \begin{equation*}
        \label{commut}D_{S'}^{H}\circ \Phitildemkdef(\overline{\alpha})= \Phitildemkdef(\overline{\alpha})\circ D_S^\H,
    \end{equation*} as the parallel transport $\Phitildemkdef(\overline{\alpha})$ is a graded isomorphism of $\Z-$modules. Hence, it follows that $$\Phitildemk(\overline{\alpha}\ast \psi \ast \overline{\alpha}^{-1})= \Phitildemkdef(\overline{\alpha})^{-1}\circ D_{S'}^\H\circ \Phitildemkdef(\overline{\alpha})=D_S^\H \in \Im(\PhitildemkSvH).$$ 
    Therefore, let us assume that $\NS(S)=\Z h$ and let $p\in \N\Star$ be a positive integer. We recall that, by Lemma \ref{lemtensor} (1), the Fourier-Mukai equivalence $p\HFM\colon \Db(S)\to \Db(S)$ defines a morphism $p\HFM\in \Hom_{\GmkFM}(\SvH,(S,v_{pH},H)$, where $$v_{pH}=(m,mph, \frac{rp^2h^2}{2}+mk).$$ More generally, the cohomological action $p\HFM^\H$ on the whole Mukai lattice $\Htilde(S,\Z)$ is given by the following: for any $(r,\xi,a)\in \Htilde(S,\Z)$, \begin{equation}
        \label{pHH} p\HFM^\H(r,\xi,a)=(r, rph+\xi, r\frac{p^2h^2}{2} + a +p\xi h).
    \end{equation} 
    By Lemma \ref{lemFM} (1), we can choose $p\gg0$ such that both the Fourier-Mukai equivalences $\FM_\PP$ and $\FM_{\PP}\dual$ induce isomorphisms on the corresponding moduli spaces. The composition $\FM_\PP^{-1}\circ \FM_\PP\dual$ defines a morphism $\phi \in \Aut_{\GmkFM}(S,v_{pH},H)$ such that $\PhitildemkFM(\phi)=D_S^\H$, described as in (\ref{dual}). Furthermore, combining (\ref{pHH}) and (\ref{dual}), we get $$\psi:= p\HFM \ast \phi\ast p\HFM \in \Aut_{\Gmk}((S,v,H))$$ and, for any $(r,\xi,a)\in \Htilde(S,\Z))$, \begin{align*}\PhitildemkFM(\psi)&= \PhitildemkFM(p\HFM) \circ 
        \PhitildemkFM(\FM_\PP^{-1})\circ \PhitildemkFM(\FM_\PP\dual) \circ \PhitildemkFM(p\HFM) (r,\xi,a)=\\
        &= \PhitildemkFM(p\HFM) \circ D_S^\H(r, rph+\xi, r\frac{p^2h^2}{2} + a +p\xi h)= \nonumber\\
        &= \PhitildemkFM(p\HFM)(r, -rph-\xi, r\frac{p^2h^2}{2} + a +p\xi h)=(r,-\xi,a) =D_S^\H(r,\xi,a).\nonumber
    \end{align*} In order to conclude, we observe that $$\ori(\PhitildemkFM(\psi))=\ori(p\HFM^\H)+\ori((\FM_\PP\dual)^\H)- \ori(\FM_{\PP}^\H) + \ori(p\HFM^\H)=1,$$ by Corollary \ref{cororientFM}. Consequently, $$\Phimk(\psi)=\Psimk(\PhitildemkFM(\psi))=(-1)^{\ori(\psi^\H)}\psi^\H_{|v\ort}=-D^\H_{S|v\ort}=R_s\circ R_{s_1},$$ concluding the proof.
\end{proof}

We conclude the section with the proof of Theorem \ref{thmNmon}, which has now become a straightforward application of Proposition \ref{propso+v}, Proposition \ref{propdual} and Remark \ref{rmkgenN2}.

\begin{proof}[Proof of Theorem \ref{thmNmon}] Let $\SvH$ an $(m,k)-$triple, with $m\geq 1$ and $k>2$, and let us assume that $S$ is an elliptic Abelian surface with $\NS(S)=\brakett{e,f}$, $v=m(1,0,-k)$ and $h:=\c_1(H)=e+tf$, with $t\gg 0$, as in Proposition \ref{propsurfacemon}. By Proposition \ref{propso+v}, $$\SO^+(\Htilde(S,\Z))_v\subseteq \Im(\Phimk_{\SvH})$$ and, by Proposition \ref{propdual}, the involution $R_s\circ R_{s_1}\in \Im(\PhimkSvH)$. By Remark \ref{rmkgenN2}, those are exactly the generators of the group $\Nn(v\ort)$, hence, inclusion (\ref{Nnimphi}) is proven. The isomorphism of functors $\lambda\colon \Phimk\to \ptmk$ defined in Section \ref{seclambda} (see also Corollary \ref{corlambda}) induces an identification  $$\lambda_{\SvH}^\sharp(\Im(\PhimkSvH))=\Im(\ptmk_{\SvH}),$$ as subgroups of $\monlt(\KvSH)$, by Corollary \ref{corimpt}, where $\lambda_{\SvH}$ is the Hodge isometry of Theorem \ref{pr20thm1.6}. The restriction of $\lambda_{\SvH}^\sharp$ to $\Nn(v\ort)$, together with inclusion (\ref{Nnimphi}), proves the thesis for any $(m,k)-$triple satisfying the current hypotheses. The result is easily extended to any $(m,k)-$triple $(S',v',H')$, as $$\Im(\ptmk_{(S',v',H'))})= \ptmk(\eta)^\sharp(\Im(\ptmk_{\SvH}))$$
for any non-trivial element $\eta \in \Hom_{\Gmk}(\SvH, (S',v',H'))$, whose existence is guaranteed by Remark \ref{rmkhomgmknonempty}.  
\end{proof}

\subsection{The main results}\label{secmainresults}
We conclude the discussion by proving the main result of the paper, which provides a complete description of the locally trivial monodromy group of a moduli space of sheaves on an Abelian surface, and dealing with some straightforward consequences of the latter. We start by recalling that, if $\SvH$ is an $(m,k)-$triple, with $m>1$ and $k>2$ two positive integers, then, by Corollary \ref{corinjmon}, there exists an injective morphism of groups $$i^{\sharp}_{w,m}\colon \monlt(\Kv(S,H)) \longrightarrow \mon^{2}(\Kw(S,H)),$$ induced by the closed embedding $i_{w,m}\colon \KwSH\to \KvSH$ as one of the connected components of the most singular locus of $\KvSH$. We will now complete the picture by proving surjectivity of the morphism $i^{\sharp}_{w,m}$, in which the inclusion of groups $$\Nn(\KvSH)\subseteq \monlt(\KvSH)$$ of Theorem \ref{thmNmon} will play a fundamental role. The outcome is the following, which extends Markman's and Mongardi's result (see Theorem \ref{thmmonkummer}) to the non primitive case. 
\begin{thm}\label{mainthm}
    Let $m,k$ be two positive integers, with $m>1$ and $k>2$, and let $\SvH$ be an $(m,k)-$triple. Then the isomorphism $i^\sharp_{w,m}$ restricts to an isomorphism of groups $$i^{\sharp}_{w,m}\colon \monlt(\Kv(S,H)) \longrightarrow \mon^{2}(\Kw(S,H)).$$ 
    In particular, we get an identification \begin{equation}\label{monNn}
        \monlt(\KvSH)= \Nn(\KvSH).
    \end{equation}
\end{thm}
\begin{rmk} We wish to point out that Theorem \ref{mainthm} is not a generalization of Theorem \ref{thmmonkummer}, as the latter will be used in our proof. For this reason, the case $m=1$ has been omitted from the statement.
\end{rmk}

\begin{proof}[Proof of Theorem \ref{mainthm}] The key ingredient of the proof is the possibility to compare the two groups $\Nn(\KvSH)$ and $\Nn(\KwSH)$ under the identification $$\O(\H^{2}(\Kv(S,H),\Z)) \simeq \O(\H^{2}(\Kw(S,H),\Z))$$ given by $(\lambda_{\SvH}^{-1})^\sharp\circ \lambda_{\SwH}$ (see Theorem \ref{pr20thm1.6}), in a way that is natural with respect to $i^\sharp_{w,m}$. This naturality is precisely guaranteed by Lemma \ref{lemisharp}, which states the identity $$i^{\sharp}_{w,m}=\lambda^\sharp_{\SwH}\circ (\lambda^\sharp_{\SvH})^{-1},$$ under the identification $v\ort=w\ort$. By Theorem \ref{thmNmon}, we have an inclusion of groups $\Nn(\KvSH)\subseteq \monlt(\KvSH)$ and the injective morphism $i^{\sharp}_{w,m}\colon \monlt(\Kv(S,H)) \to \mon^{2}(\Kw(S,H))$ provides the following identification \begin{equation}
    \label{Nn} i^\sharp_{w,m}(\Nn(\KvSH))=\Nn(\KwSH),
\end{equation} as both the groups involved are only defined lattice theoretically, and as both $\lambda_{\SvH}$ and $\lambda_{\SwH}$ are both isometries. The right hand side of equality (\ref{Nn}) is precisely $\mon^2(\KwSH)$, by Theorem \ref{thmmonkummer}, which is, on the other hand, the image of the morphism $i^\sharp_{w,m}$. In other words, the commutativity of the diagram \begin{center}
      \begin{tikzcd}[ampersand replacement=\&, row sep=large]
          \monlt(\Kv(S,H))\arrow[rrr, hook, "i^{\sharp}_{w,m}", "\text{Corollary \ref{corinjmon}}" swap] \&\&\& \mon^{2}(\Kw(S,H))\arrow[d, equal, "\text{Theorem \ref{thmmonkummer} }" swap]\\
          \Nn(\Kv(S,H))\arrow[u, , " \rotatebox{90}{\(\subseteq\)}", "\text{ Theorem \ref{thmNmon}}" swap, hook]\arrow[rrr, bend right=15, "i^{\sharp}_{w,m}"]\arrow[r, "\sim" swap, "(\lambda_{\SvH}^{\sharp})^{-1}"] \& \Nn(v\ort) \arrow[r,equal]\& \Nn(w\ort)\arrow[r, "\sim" swap, "\lambda_{\SwH}^{\sharp}"] \&  \Nn(\Kw(S,H))
      \end{tikzcd}
  \end{center}
is shown, concluding the proof.
\end{proof}

\begin{rmk}
    An immediate consequence of Theorem \ref{mainthm} is that, if $\SvH$ is an $(m,k)-$triple with $m>1$ and $k>2$, whereas the locally trivial deformation class of the moduli space $\KvSH$ is determined by both $m$ and $k$ (see Theorem \ref{PR18defolt}), the isomorphism class of its locally trivial monodromy group $\monlt(\KvSH)$ is uniquely determined by $k$. Indeed, if $m\neq m'$ are two positive integers and $v=mw$ and $v'=m'w$, then $i_{w,m}$ and $i_{w,m'}$ induce a natural isomorphism of groups $$(i_{w,m'})^{-1}\circ i^\sharp_{w,m}\colon \monlt(\KvSH) \longrightarrow \monlt(K_{v'}(S,H)).$$
\end{rmk}

As the locally trivial monodromy group is a locally trivial deformation invariant, we can extend Theorem \ref{mainthm} to any irreducible symplectic variety $X$ which is locally trivial deformation equivalent to $\KvSH$, refining Corollary \ref{corinjmondefo}.

\begin{cor}\label{maincor}
    Let $X$ be an irreducible symplectic variety locally trivial deformation equivalent to $\Kv(S,H)$, where $\SvH$ is an $(m,k)-$triple with $m>1$ and $k>2$, and let $Z$ be a connected component of the most singular locus of $X$. Then $Z$ is an irreducible holomorphic symplectic manifold deformation equivalent to $\Kw(S,H)$ and its closed embedding $i_{Z,X}\colon Z \to X$ induces an isomorphism of groups $$i^\sharp_{Z,X}\colon \monlt(X) \longrightarrow \mon^2(Z).$$ Furthermore, we have an identification of groups $$\monlt(X)=\Nn(X).$$
\end{cor}
\begin{proof}The only part of the statement that remains to be proved is surjectivity of the morphism $i_{Z,X}^\sharp$, which now follows from Theorem \ref{mainthm} and commutativity of diagram (\ref{cdmonX}).\end{proof}

Finally, as a consequence of Remark \ref{indexN}, we immediately deduce the following
\begin{cor}{\cite[Lemma 4.1]{Mar10}}\label{corindex}
    Let $m,k$ be two positive integers, with $m>1$ and $k>2$, and let $\SvH$ be an $(m,k)-$triple. If $X$ is an irreducible symplectic variety locally trivial deformation to $\KvSH$, then $$[\O^+(\H^2(X,\Z))\colon \monlt(X)]=2^{\rho(k)},$$ where $\rho(k)$ is the number of distinct primes occurring in the factorization of $k$.
\end{cor}


%% file: sec5.tex
\label{secSYZ}In this Section we present a geometric application of the main results of this paper, providing a proof of the SYZ conjecture for singular moduli spaces of sheaves on Abelian surfaces. The SYZ conjecture plays a significant role in the study of irreducible symplectic varieties, as it establishes a deep connection between algebraic and symplectic geometry and has important implications for their classification. In simple terms, it predicts that nef isotropic line bundles are related to the existence of Lagrangian fibrations (see Section \ref{seclagr}). In the smooth case of irreducible holomorphic symplectic manifolds, the latter has been established for any known deformation class. For singular symplectic varieties the problem is widely open, but a substantial contribution has been made in the case of primitive symplectic varieties by the work of \cite{OO25}, where the conjecture is proven for singular moduli spaces of sheaves on K3 surfaces. We devote this Section to provide an adaptation of that proof to the Abelian case (see Section \ref{secproofSYZ}), relying on the monodromy description achieved in the previous Section.

\subsection{Lagrangian fibrations} \label{seclagr}In this Section, we collect some essential notions and results concerning a special class of fibrations that is particularly relevant to the study and classification of symplectic varieties, namely Lagrangian fibrations. We then formulate the SYZ conjecture and provide an overview of the current state of the art.
\begin{defn}
    Let $n\in \N\setminus\{0\}$, let $(X,\sigma)$ be a primitive symplectic variety (see Definition \ref{defsv} (3)) of dimension $2n$ and let $X_\reg$ be its smooth locus. \begin{enumerate}
        \item A subvariety $Z\subseteq X$ of dimension $n$ is called \textit{Lagrangian} if $Z\cap X_\reg \neq \emptyset$ and $\sigma|_{X_\reg\cap Z_\reg}=0$.
        \item A \textit{Lagrangian fibration on $X$} is a surjective morphism $f\colon X \to B$ onto a normal Kähler space, with connected fibers and such that the general fiber of $f$ is a Lagrangian subvariety of $X$.
    \end{enumerate}
\end{defn}
A description of the main properties of Lagrangian fibrations on primitive symplectic varieties, together with a sufficient characterization, is given by the following Theorem, which generalizes to the singular setting several foundational results due to Matsushita. 
\begin{thm}{(\cite[Theorem 3]{Sch20}, \cite[Theorem 2.8]{KL25})} \label{thmlagr}Let $X$ be a primitive symplectic variety of dimension $2n$ and let $f\colon X \to B$ be a surjective morphism with connected fibers onto a normal Kähler space. Then $f\colon X \to B$ is a Lagrangian fibration. Furthermore: \begin{enumerate}
    \item The base $B$ is a $\Q-$factorial projective klt variety of Picard rank $1$. If, additionally, $X$ is irreducible symplectic, then $B$ is Fano and, if $B$ is smooth, then $B\simeq \proj^n$.
    \item The general fiber of $f$ is an Abelian variety of dimension $n$ completely contained in $X_\reg$.
    \item The map $f$ is equidimensional and all irreducible components of each fiber are Lagrangian subvarieties.
\end{enumerate}
\end{thm}

We recall that primitive symplectic varieties, being compact Kähler spaces, have a well defined Kähler cone, which is an open cone in $\H^{1,1}(X,\R)$ (see Section \ref{secsingularsymplectc}). A class $\alpha \in \H^{1,1}(X,\R)$ is said to be \textit{nef} if it belongs to the closure of the Kähler cone. 

Let $X$ be a primitive symplectic variety and let $q_X$ denote its Beauville-Bogomolov-Fujiki form (see Remark \ref{rmkPSV}). If $f\colon X \to B$ is a Lagrangian fibration on $X$ and $b=f\Star\OO_B(1)$ is the class of a fiber of $f$, then $b$ is nef and $q_X(b)=0$. The SYZ conjecture claims that the converse holds.

\begin{conj}{(SYZ Conjecture for primitive symplectic varieties)} Let $X$ be a primitive symplectic variety and let $L$ be a line bundle on $X$. If $L$ is nef and $q_X(L)=0$, then there exists a Lagrangian fibration $f\colon X \to B$ such that $L=f\Star\OO_B(1)$.
\end{conj}

In that case, we will say that $L$ \textit{induces a Lagrangian fibration} on $X$.

In the smooth case, the SYZ conjecture has been proven true for any known deformation class, thanks to the work of \cite{BM14}, \cite{Mar14}, \cite{Mat17}, \cite{Wie16}, \cite{Yos16}, \cite{Wie18}, \cite{MR19} and \cite{MO22}. In the singular setting, the first distinguished class of primitive symplectic varieties for which the latter has been established is the one of singular moduli spaces of sheaves on K3 surfaces, by \cite{OO25}. In the next Section, we will show that the SYZ conjecture holds also in the Abelian case, by applying the general theory developed by \cite{OO25} for primitive symplectic varieties and exploiting the monodromy description provided in Section \ref{secmonodromy}.

\subsection{The SYZ conjecture for singular moduli spaces of sheaves on Abelian surfaces} \label{secproofSYZ}
In this Section we will provide a proof for the SYZ conjecture for irreducible symplectic varieties that are locally trivial deformations of moduli spaces of sheaves $\Kv(S,H)$, for any $(m,k)-$triple $\SvH$, with $m>1$ and $k>2$. 

For those deformation classes for which the SYZ conjecture has been established, a description of the monodromy orbit of primitive isotropic elements of the Beauville-Bogomolov-Fujiki lattice has turned out to be crucial. Building on the strategy developed by \cite{OO25} for moduli spaces of sheaves on K3 surfaces, we use the relation between $\monlt(\Kv(S,H))$ and $\mon^2(\KwSH)$ to reduce this problem to the above-mentioned classification in the smooth case, provided by \cite[Section 5]{Wie18}. 

The goal of this section is the following result:
\begin{thm}\label{SYZAb}Let $m>1$ and $k>2$ be two integers, let $\SvH$ be an $(m,k)-$triple and let $X$ be an irreducible symplectic variety locally trivial deformation equivalent to $\Kv(S,H)$. If $L$ is a nef line bundle on $X$ such that $q_X(L)=0$, then there exists a Lagrangian fibration $f\colon X \to B$ such that $L=f\Star\OO_B(1)$.
\end{thm}
\subsubsection{Deformations of Lagrangian fibrations} The first key ingredient for the proof of Theorem \ref{SYZAb} is a result of \cite{OO25} concerning locally trivial deformations of Lagrangian fibrations on primitive symplectic varieties.

\begin{defn}\label{defnlteqLagrangian}\begin{enumerate}
    \item A \textit{locally trivial family of Lagrangian fibrations} is a locally trivial family $p\colon \XX \to T$ of primitive symplectic varieties (see Definition \ref{defltfamily} (3)) such that there exists a commutative diagram \begin{center}
        \begin{tikzcd}
            \XX \arrow[rr,"f"] \arrow[dr, "p" swap]&& \BB \arrow[dl,"s"]\\
            & T
        \end{tikzcd}
    \end{center} where $f$ is a $T-$morphism, $s$ is projective and, for any $t\in T$, the restriction $f_t\colon \XX_t \to \BB_t$ is a Lagrangian fibration. We will denote by $p\colon \XX/\BB \to T$ a locally trivial deformation of Lagrangian fibrations as above.
    \item If $X_1$ and $X_2$ are two primitive symplectic varieties, two Lagrangian fibrations $f_i\colon X_i \to B_i$ on $X_i$, for $i=1,2$, are said to be \textit{locally trivial deformation equivalent} if there exists a locally trivial deformation of Lagrangian fibrations $p\colon \XX/\BB \to T$ and two points $t_1,t_2\in T$ such that $f_{t_i}=f_i$ for $i=1,2$.
\end{enumerate}
\end{defn}

The following generalization to the singular setting of several results due to Matsushita will play a central role in the proof of Theorem \ref{SYZAb}.

\begin{thm}{(\cite[Theorem 3.1]{OO25})}\label{thmOO} For $i=1,2$, let $X_i$ be a $\Q-$factorial and terminal primitive symplectic variety and let $L_i\in \NS(X_i)$ be a nef divisor such that $q_{X_i}(L_i)=0$. Suppose that $L_1$ induces a Lagrangian fibration on $X_1$. If there exists a locally trivial parallel transport operator $g \in \PTlt(X_1,X_2)$ such that $g(L_1)=L_2$, then \begin{enumerate}
    \item $L_2$ induces a Lagrangian fibration on $X_2$;
    \item $X_1$ and $X_2$ are locally trivial deformation equivalent as Lagrangian fibrations.
\end{enumerate}
\end{thm}
Theorem \ref{thmOO} will be used to construct a Lagrangian fibration on any irreducible symplectic variety as in Theorem \ref{SYZAb} as locally trivial deformation of a \textit{Beauville-Mukai system}.

\subsubsection{Beauville-Mukai systems} Let $m>1$ and $k>2$ be two positive integers and let us consider an $(m,k)-$triple $\SvH$, with $S$ an Abelian surface such that $\NS(S)=\Z h$, with $h$ the class of a polarization $H$ of degree $h^2=2k$ and $v=(0, mh, 0)$ a Mukai vector. Mapping each sheaf $F \in \Kv(S,H)$ to its Fitting support (see \cite[Section V.1.3]{Eis95}) defines a surjective morphism (see \cite[Section 3.1.2]{PR18}) \begin{equation}
    \label{BMsyst}f_v\colon \Kv(S,H) \longrightarrow \abs{mH}\simeq \proj^{m^2k-1}.
\end{equation} For any smooth and irreducible curve $C \in \abs{mH}$, the fiber $f_v^{-1}(C)$ coincides with an Abelian subvariety of the Jacobian $\Jac^d(C)$, with $d=m^2k$ (see \cite[Section 6.20]{Wie18}). By Stein's factorization Theorem, the morphism $f_v$ has connected fibers and, by Theorem \ref{thmlagr}, it is a Lagrangian fibration.

\begin{defn}\label{defBMsyst} Let $\SvH$ be an $(m,k)-$triple as above.  A Lagrangian fibration $f_v\colon \KvSH\to \abs{mH}$ defined as in (\ref{BMsyst}) will be called \textit{Beauville-Mukai system of $(m,k)-$type}.
\end{defn}

If $m=1$, then $v=w=(0,h,0)$ yields a Lagrangian fibration $f_w\colon \KwSH \to \abs{H}$ on an irreducible holomorphic symplectic manifold of type $Kum^{k-1}$, known as \textit{Beauville-Mukai system of generalized Kummer type} (see \cite[Section 6]{Wie18}).

\begin{rmk}It is straightforward from Theorem \ref{PR18defolt} that, for any $(m,k)-$triple $\SvH$ with $m>1$ and $k>2$, any irreducible symplectic variety $X$ that is locally trivial deformation of a moduli space $\Kv(S,H)$ is again locally trivial deformation equivalent to a Beauville-Mukai system of $(m,k)-$type, in the sense of Definition \ref{defltfamily} (4).
\end{rmk}

If $\SvH$ is an $(m,k)-$triple as above, with $v=(0,mh,0)$, then the class $b=(0,0,1)\in\Htilde(S,\Z)$ naturally belongs to $v\ort$. Its image $\lambda_v(b)\in \H^{1,1}(\Kv(S,H),\Z) \simeq \Pic(\KvSH)$ via the Hodge isometry $\lambda_v:= \lambda_{\SvH}\colon v\ort \to \H^2(\Kv(S,H))$ (see Theorem \ref{pr20thm1.6}) is an isotropic class which defines the Lagrangian fibration $f_v$. This is the content of the following Lemma, which is an adaptation to the Abelian case of Lemma 4.3 of \cite{OO25}.

\begin{lem}\label{lemfiber}Let $f_v\colon \Kv(S,H) \to \abs{mH}\simeq \proj^n$ be a Beauville-Mukai system of $(m,k)-$type, with $n=m^2k-1$, and set $b=(0,0,1)\in v\ort$. Then $$f_v\Star\OO_{\proj^n}(1)=\lambda_v(b).$$
\end{lem}
\begin{proof}The considered Beauville-Mukai system of $(m,k)-$type fits in the following commutative diagram \begin{center}
    \begin{tikzcd}
        \KwSH \arrow[r, hook, "i_{w,m}"]\arrow[d, "f_w" swap] & \KvSH \arrow[d, "f_v"]\\
        \proj^{k-1} \simeq \abs{H} \arrow[r, "\nu_m" swap] & \abs{mH}\simeq \proj^{n},
    \end{tikzcd}
\end{center}where $i_{w,m}$ is the embedding of $\KwSH$ into $\Kv(S,H)$ defined in Section \ref{secemb}, the morphism $f_w$ is the Beauville-Mukai system of Kummer type associated to the $(1,k)-$\linebreak[4]triple $(S,w,H)$ and $\nu_m$ is the map induced by the $m$-th Veronese embedding. By commutativity of the diagram, we get $$(i_{w,m}\Star \circ f_v\Star) (\OO_{\proj^n}(1))= (f_w\Star \circ \nu_m\Star)(\OO_{\proj^n}(1)).$$ By definition, we have $\nu_w\Star\OO_{\proj^n}(1)=\OO_{\proj^{k-1}}(m)=m\OO_{\proj^{k-1}}(1)$ and, by \cite[Proposition 6.29 (iii)]{Wie18}, it holds $f_w\Star\OO_{\proj^{k-1}}(1)=\lambda_w(b)$. Putting identities together we get $$(i_{w,m}\Star \circ f_v\Star) (\OO_{\proj^n}(1)) = m\lambda_w(b).$$ By Lemma \ref{istar}, it holds $i_{w,m}\Star= m\lambda_w\circ \lambda_v^{-1}$, hence the claim follows.\end{proof}
\subsubsection{Locally trivial deformations of Beauville-Mukai systems} With the notions collected so far, we are finally in the position to prove Theorem \ref{SYZAb}, by adapting the proof of \cite[Proposition 6.2]{OO25} to our setting. 

\begin{proof}[Proof of Theorem \ref{SYZAb}]Let $m>1$ and $k>2$ be two positive integers and let $X$ be an irreducible symplectic variety deformation equivalent to $\Kv(S,H)$ for an $(m,k)-$triple $\SvH$. Let $Z$ be a connected component of its most singular locus, which, by Corollary \ref{corinjmondefo}, is an irreducible holomorphic symplectic manifold deformation equivalent to $\KwSH$. 

Let $l=\c_1(L)\in \NS(X)$ be the class of an isotropic nef line bundle $L$ on $X$. By Lemma \ref{istar} and Corollary \ref{corinjmondefo}, the pullback $$i_{Z,X}\Star\colon \H^2(X,\Z) \to \H^2(Z,\Z)$$ of the closed embedding $i_{Z,X}\colon Z \hookrightarrow X$ is $m$ times a Hodge isometry. Hence, there exists a unique class $l_0 \in \NS(Z)$ of divisibility $\div(l_0)=\div(l)$ such that $i_{Z,X}\Star(l)=ml_0$. By Proposition 6.29 and Proposition 6.32 of \cite{Wie18}, there exists a Beauville-Mukai system of Kummer type $(1,k)$ $$f_{w'}\colon Z':= K_{w'}(S',H')\to \abs{H'}\simeq \proj^{k-1}$$ and a parallel transport operator $g_0\in \PT^2(Z,Z')$ such that $g_0(h_0)=f_{w'}\Star\OO_{\proj^{k-1}}(1)$. Let us now set $v'=mw'$, $n=m^2k-1$ and let us consider the Beauville-Mukai system of $(m,k)-$type $$f_{v'}\colon X':= K_{v'}(S',H') \longrightarrow \abs{mH'}\simeq \proj^n.$$ 
    We will now construct a locally trivial parallel transport operator $g\in \PTlt(X,X')$ such that the following diagram \begin{equation}\label{cdSYZ}
        \begin{tikzcd}
            \H^2(X,\Z)\arrow[d,"i_{Z,X}\Star" swap] \arrow[r,"g"] &\H^2(X',\Z)\arrow[d, "i_{w',m}\Star"]\\
            \H^2(Z,\Z) \arrow[r,"g_0"] &\H^2(Z',\Z)
        \end{tikzcd}
    \end{equation} is commutative. Let $g' \in \PTlt(X',X)$ be a locally trivial parallel transport operator and let $g_0' \in \PT^2(Z',Z)$ be the parallel transport operator induced by restriction. Let us consider $g_0'\circ g_0 \in \mon^2(Z)$ and $(i_{Z,X}^\sharp)^{-1}(g_0'\circ g_0) \in \monlt(X)$. Then $$g:= (g')^{-1}\circ (i_{Z,X}^\sharp)^{-1}(g_0'\circ g_0) \in \PT(X,X')$$ is such that \begin{align*}
        i\Star_{w',m}\circ g &= (g_0')^{-1}\circ i_{Z,X}\Star \circ g' \circ g =  (g_0')^{-1}\circ i_{Z,X}\Star \circ g' \circ (g')^{-1} \circ (i_{Z,X}^\sharp)^{-1}(g_0'\circ g_0)= \\
        &= (g_0')^{-1}\circ g_0' \circ g_0 \circ i_{Z,X}\Star = g_0\circ i_{Z,X}\Star,
    \end{align*} where the first equality follows from (\ref{isharp2}) and the third from (\ref{isharp1}). Commutativity of diagram (\ref{cdSYZ}) is then proven, and from the latter we deduce that $$i\Star_{w',m}(g(l))= g_0(i\Star_{Z,X}(l))=g_0(m l_0) = m f_{w'}\Star \OO_{\proj^{k-1}}(1) = i_{w',m}\Star (f_{v'}\Star \OO_{\proj^n}(1)),$$ where the last equality follows from Lemma \ref{lemfiber}. By injectivity of $i_{w',m}\Star$, we obtain $$g(l)=f_{v'}\Star \OO_{\proj^n}(1).$$ By Theorem \ref{thmOO}, we conclude that $L$ induces a Lagrangian fibration $f\colon X \to B$ on $X$, that is locally trivial deformation equivalent to the Beauville-Mukai system $f_{v'}\colon X' \to \abs{mH'}$, in the sense of Definition \ref{defnlteqLagrangian}.
\end{proof}

\begin{rmk}
    Notice that the proof of Theorem \ref{SYZAb} can be applied in particular to the case in which $X$ is already a Beauville-Mukai system, to show that, for any fixed pair $(m,k)$, with $m>1$ and $k>2$, all Beauville-Mukai systems of $(m,k)-$type are locally trivial deformation equivalent as Lagrangian fibrations. In light of this, Definition \ref{defBMsyst} precisely identifies a locally trivial deformation equivalence class of Lagrangian fibrations.
\end{rmk}

%% file: appendix.tex
\label{appendix}This Appendix is devoted to recall some lattice theory results that will be useful throughout the paper. In the following, we will let $(\Lambda, \cdot)$ be an \textit{even lattice}, i.e, a finitely generated free $\Z-$module $\Lambda$ equipped with a non-degenerate symmetric bilinear pairing $\cdot\colon \Lambda \times \Lambda \to \Z$, such that the associated quadratic form $q_\Lambda$ takes values in $2\Z$. 

If $\Lambda$ and $\Lambda'$ are two such lattices and the respective bilinear forms are clear from the context and if $v\in \Lambda$, we will use the following notations: \begin{itemize}
    \item $\O(\Lambda,\Lambda')$ for the group of isometries from $\Lambda$ to $\Lambda'$;
    \item $\O(\Lambda)$ for the group of isometries from $\Lambda$ to itself;
    \item $\O(\Lambda)_v$ for the subgroup of $\O(\Lambda)$ of isometries fixing the element $v$. 
\end{itemize}

\subsection{Isometries as kernels of natural group morphisms}In this work we will be interested in special groups of isometries that can be characterized as kernels of some natural group homomorphisms, which we will recall in the following section.
\subsubsection{Determinant character}\label{det} If $g \in \O(\Lambda)$ is an isometry, we define its \textit{determinant} $\det(g)$ as the determinant of the matrix associated to $g$ with respect to any $\Z-$basis of $\Lambda$. This yields a natural \textit{determinant character} $$\det\colon \O(\Lambda) \longrightarrow \{1,-1\}.$$  We will denote by $\SO(\Lambda):= \ker(\det)$ the group of isometries of $\Lambda$ of determinant $1$.

\subsubsection{Orientation character}\label{rmkorient}  The second fundamental character that can be naturally associated to any isometry is related to the notion of \textit{orientation}. For a more precise discussion on this theory, we refer the reader to \cite[Section 4]{Mar11} and \cite[Section 4.1]{Mar08}.

Let $(\Lambdatilde,\cdot)$ be a lattice of signature $sgn(\Lambdatilde)=(p,q)$, with $p,q\geq1$ and let us consider the \textit{big positive cone} $$\widetilde{\CC}:= \{\alpha \in \Lambdatilde\otimes_\Z\R \colon \alpha\cdot \alpha >0\}.$$\begin{itemize}
        \item[(a)] If $p=1$, then $\widetilde{\CC}$ has two connected components and we define an \textit{orientation character} $$\ori\colon \O(\Lambdatilde) \longrightarrow \Z/2\Z,$$ by setting $\ori(g)=0$ if $g$ is an isometry mapping each connected component to itself and $\ori(g)=1$ otherwise. The choice of any connected component of $\widetilde{\CC}$ determines an \textit{orientation on $\Lambdatilde$}.
        \item[(b)] If $p>1$, by \cite[Lemma 4.1]{Mar11}, the cone $\widetilde{\CC}$ is a deformation retract of $\R^p\setminus\{0\}$, hence homotopic to the sphere $S^{p-1}$,
        and the action of $\Lambdatilde$ on the top cohomology group of of $\widetilde{\CC}$ determines an \textit{orientation character} $$\ori\colon \O(\Lambdatilde) \longrightarrow \Z/2\Z,$$ defined as follows: as $\H^{p-1}(\widetilde{\CC},\Z)\simeq \H^{p-1}(S^{p-1},\Z)\simeq \Z$, we can choose a generator $\epsilon$ of the latter and notice that, if $g\in \O(\Lambdatilde)$, then it induces an action $(g_{\R|\widetilde{\CC}})\Star$ on $\H^{p-1}(\widetilde{\CC},\Z)$ that either preserves $\epsilon$ or maps it to its opposite. We then set 
        \begin{align*}\ori(g) = \bigg\{ \begin{array}{ll}
           0  & \text{if } (g_{\R|\widetilde{\CC}})\Star(\epsilon) = \epsilon \\
           1  & \text{if } (g_{\R|\widetilde{\CC}})\Star(\epsilon) = -\epsilon.
        \end{array}
    \end{align*} A choice of a generator $\epsilon$ as above, or, equivalently, of an ordered basis of a positive $p-$dimensional real subspace of $\Lambdatilde\otimes_{\Z}\R$, determines an \textit{orientation on $\Lambdatilde$}.
    \end{itemize}
    In any of the two cases above, we define the group of \textit{orientation preserving isometries of $\Lambdatilde$} as $$\O^+(\Lambdatilde):= \ker(\ori)$$ and, similarly, we define $\SO^+(\Lambdatilde):= \O^+(\Lambdatilde)\cap \SO(\Lambdatilde)$. Analogously, if $(\Lambdatilde_1,\epsilon_1)$ and $(\Lambdatilde_2,\epsilon_2)$ are two oriented lattices as above, we get an \textit{orientation map} \begin{equation}
        \label{orientmap}\ori\colon \O(\Lambdatilde_1,\Lambdatilde_2) \longrightarrow \Z/2\Z
    \end{equation}  and we can define the set of \textit{orientation preserving isometries} as $\O^+(\Lambdatilde_1,\Lambdatilde_2):= \ori^{-1}(0)$.

    We also point out that, if $\Lambda\subseteq\Lambdatilde$ is a sublattice of $\Lambdatilde$ of signature $(p',q')$, with $p',q'\geq 1$, inducing a natural embedding $\O(\Lambda)\subseteq \O(\Lambdatilde)$ by extending to the identity on $\Lambda \ort$, then the orientation character of $\Lambdatilde$ restricts to $\Lambda$ as the natural orientation character of $\Lambda$.

    \subsubsection{Discriminant}\label{secdiscr} The non-degeneracy of the bilinear form $\cdot$ gives a canonical embedding of $\Lambda$ in its \textit{dual lattice} $\Lambda\dual:= \Hom_\Z(\Lambda,\Z)\subset \Lambda \otimes_\Z \Q$. \begin{itemize}
    \item We define the \textit{discriminant group} of $\Lambda$ as the quotient $A_\Lambda:= \Lambda\dual/\Lambda$, which is a finite group of order $\abs{\det(\Lambda,\cdot)}$, where $\det(\Lambda,\cdot)$ is the determinant of the matrix associated to the bilinear form $\cdot$ with respect to any $\Z-$basis of $\Lambda$. 
    \item A lattice $\Lambda$ is called \textit{unimodular} if $A_\Lambda \simeq \{0\}$ or, equivalently, if $\det(\Lambda)=\pm 1$.
    \item Any isometry $g\in \O(\Lambda)$ induces an automorphism $\overline{g} \in \Aut(A_\Lambda)$ that preserves the discriminant quadratic form $\overline{q_{\Lambda}}\colon A_{\Lambda}\to \faktor{\Q}{2\Z}$ induced by $q_{\Lambda}$. This naturally defines a group morphism \begin{equation}
        \label{disc}\disc \colon \O(\Lambda)\longrightarrow \O(A_\Lambda),
    \end{equation} where $\O(A_\Lambda)$ is subgroup of $\Aut(A_\Lambda)$ of automorphisms preserving $\overline{q_{\Lambda}}$.
\end{itemize} 

\subsection{Primitive embeddings}\label{secprimitiveemb}In this Subsection we will collect basic results concerning primitive embeddings of even lattices into even unimodular lattices and some extendibility criteria for isometries that will lay the groundwork for the proof of Proposition \ref{lemrefl2}. For a more detailed discussion, we refer to \cite{Nik79} (see also \cite[Section 1.3.1]{Kon21}).

Let $L$ be an even lattice and let us consider an embedding $S \overset{i}{\hookrightarrow} L$ of lattices. \begin{itemize}
    \item The embedding $i$ (or - respectively - the sublattice $S$, identified with its image via $i$) is called \textit{primitive} if the quotient $\faktor{L}{S}$ is torsion free.
    \item The \textit{saturation} $\overline{S}$ of $S$ in $L$ is the smallest primitive sublattice of $L$ containing $S$, i.e. $\overline{S}:= (S\otimes_\Z \Q) \cap L$. Notice that we have a natural chain of embeddings $S \subseteq \overline{S} \subseteq L$ and that the index $[\overline{S}\colon S]$ is finite. Furthermore, it holds $\rk(S)=\rk(\overline{S})$ and a set of generators for $\overline{S}$ can be provided by a suitable $\Q-$linear combination of a set of generators for $S$. 
\end{itemize} 
The sublattice $S$ is primitive if and only if $S =\overline{S}$. In that case, its orthogonal complement $K:= S \ort$ in $L$ is also a primitive sublattice of $L$. Let us assume that $L$ is unimodular and let us consider the quotient $H= \faktor{L}{S\oplus K}$. Then $H$ is a finite Abelian group and an isotropic subgroup of $A_S \oplus A_K$ with respect to $\overline{q_S}\oplus \overline{q_K}$. Furthermore, it can be checked that the restrictions $p_S|_H\colon H \to A_S$ and $p_K|_H\colon H \to A_K$ of the natural projections are both isomorphisms (see \cite[Lemma 1.31]{Kon21}). In particular, set $$\gamma_{SK}:= p_K\circ (p_S|_H)^{-1}\colon A_S \to A_K,$$ we have the following results.

\begin{thm}{(\cite[Proposition 1.6.1]{Nik79}, \cite[Theorem 1.32]{Kon21})}\label{thmantiso} Let $L$ be an even, unimodular lattice, $S$ a primitive sublattice of $L$ and $K=S\ort$ its orthogonal complement in $L$. Then \begin{center}
    $A_S \underset{\gamma_{SK}}{\simeq} A_K$ \phantom{+} and \phantom{+}$\overline{q_S}=-\overline{q_K}\circ \gamma_{SK}$.
\end{center} 
\end{thm}
A straightforward application of Theorem \ref{thmantiso} provides a criterion for the extendibility of an isometry $\pphi\colon S_1\to S_2$ between two primitive sublattices of $L$ to an isometry of the whole lattice $L$, i.e. the existence of an isometry $\Tilde{\pphi}\in \O(L)$ such that $\Tilde{\pphi}|_{S_1}=\pphi$.

\begin{prop}{(\cite[Proposition 1.6.1]{Nik79}, \cite[Corollary 1.33]{Kon21})} \label{propext} Let $S_1,S_2\subseteq L$ be two primitive sublattices and let $K_i:= S_i\ort$ for $i=1,2$. An isometry $\pphi\in \O(S_1,S_2)$ can be extended to an isometry of $L$ if and only if there exist and isometry $\psi\in \O(K_1,K_2)$ such that $$\overline{\psi}\circ \gamma_{S_1 K_1}= \gamma_{S_2 K_2}\circ \overline{\pphi}.$$
    \end{prop}
    In particular, we get the following folklore result: \begin{cor}\label{corextid}Let $S\subseteq L$ a primitive sublattice of an even unimodular lattice $L$ and let $K:=S\ort$. If $\pphi \in \O(S)$ is an isometry such that $\disc(\pphi)=\pm\id_{A_S}$, then it can be extended to $\Tilde{\pphi}\in \O(L)$ such that $\Tilde{\pphi}|_{K}=\pm \id_{K}$.
    \end{cor}

In order to apply Proposition \ref{propext}, a description of the image of the discriminant map defined in (\ref{disc}) might be crucial. We address this problem in the case of an even indefinite lattice, in relation with the question of its uniqueness up to isomorphism.

\begin{prop}{(\cite[Theorem 1.14.2]{Nik79}, \cite[Proposition 1.37]{Kon21})}\label{propunique} Let $T$ be an indefinite even lattice of signature $(t_+,t_-)$, with discriminant quadratic form $\overline{q}=\overline{q_T}$. Suppose that $$\rk(T)\geq l(A_T)+ 2,$$ where $l(A_T)$ is the minimum number of generators of the finite Abelian group $A_T$. Then any even lattice of signature $(t_+,t_-)$ and discriminant quadratic form $\overline{q}$ is isomorphic to $T$ and the natural discriminant map $\disc\colon \O(T)\to \O(A_T)$ is surjective.\end{prop}

\subsection{Weyl groups of reflections}\label{weylgroups}We conclude this Appendix by introducing two special groups of reflections that will play a central role in Sections \ref{secsurj} and \ref{secmainresults}. Let $k> 2$ and let $\Lambda$ be an even lattice of signature $(3,4)$, isometric to $U^{\oplus 3}\oplus <-2k>$, where $U$ is the unimodular hyperbolic lattice of rank $2$. For any $u\in \Lambda$ such that $u\cdot u=\pm 2$, let \begin{align*}
    R_u\colon \Lambda & \longrightarrow \Lambda\\
    v & \longmapsto v-2\frac{v\cdot u}{u\cdot u}u
\end{align*} be the reflection in $u$ and define the isometry $\rho_u:= -\frac{u\cdot u}{2}R_u$. The group $$\W(\Lambda):= \brakett{\rho_u\colon u\in \Lambda, u\cdot u=\pm2}$$ is a normal subgroup of finite index of $\O(\Lambda)$, often called \textit{Weyl group of reflections}. Notice that the action of the previously defined group morphisms on the generators of $\W(\Lambda)$ is the following (see, \cite[Lemma 4.1]{Mar11}, \cite[Lemma 4.10]{Mar08} \cite[Section 1.1]{Mar18}): \begin{equation}\label{char}
    \ori(\rho_u)=0, \phantom{++} \det(\rho_u)=\frac{u\cdot u}{2}, \phantom{++} \disc(\rho_u)=-\frac{u\cdot u}{2}.
\end{equation} From the first equality of (\ref{char}) we deduce that $\W(\Lambda)\subseteq \O^+(\Lambda)$ and from the last equality we get the following discriminant character $$\disc\colon \W(\Lambda) \longrightarrow \{-1,1\}\subseteq \O(A_\Lambda),$$ where $A_\Lambda$ is a cyclic group of order $2k$. Again by (\ref{char}), we get that, for any generator $\rho_u$ of $\Lambda$, it holds $\det\cdot \disc(\rho_u)=-1$. We define the following group of isometries \begin{equation}
    \label{defNn}\Nn(\Lambda):= \ker(\det\cdot\disc),
\end{equation} which is an order $2$ subgroup of $\W(\Lambda)$, generated by compositions $\rho_{u_1}\circ \rho_{u_2}$, with $u_i\in \Lambda$ such that $u_i\cdot u_i=\pm 2$ for $i=1,2$.
\begin{rmk}\label{rmkgenN} Let $\Lambdatilde=\Htilde(S,\Z)$ be the Mukai lattice of an Abelian surface $S$ (see Section \ref{ab}), isometric to $U^{\oplus4}$, let $v=m(1,0,-k)$, with $m\geq 1$ and $k> 2$, and let $\Lambda=v\ort$. \begin{enumerate}
    \item By \cite[Lemma 4.10]{Mar08}, we get that $\W(\Lambda)$ is isomorphic to the subgroup of $\O^+(\Lambda)$ of isometries that can be extended to the whole Mukai lattice $\Lambdatilde$ or, equivalently, to the image of the stabilizer $\O(\Lambdatilde)_v$ under the homomorphism \begin{align*}
    \psi \colon \O(\Lambda) & \longrightarrow \O^+(\Lambda)\\
    g & \longmapsto (-1)^{\ori(g)}g,
\end{align*} where we used the natural inclusion $\O(\Lambdatilde)_v\subseteq \O(\Lambda)$. Notice that, as $sgn(\Lambda)=(3,4)$, the isometry $-\id_{\Lambda}$ is orientation reversing.
\item Section 5 and the proof of Theorem 1.4 of \cite{Mar18} provide a set of generators for $\Nn(\Lambda)$, namely \begin{equation}
    \label{genN} \Nn(\Lambda)=\brakett{\SO^+(\Lambdatilde)_v, R_s\circ R_{s_1}}, 
\end{equation} where $s=(1,0,1), s_1=(1,0,-1) \in \Lambda$. Indeed, by Lemma 5.3 and Corollary 5.1 of \cite{Mar18}, we have an equality $$\phantom{+++}\SO^+(\Lambdatilde)_v=\brakett{R_{u_1}\circ R_{u_2}, R_{t_1}\circ R_{t_2}\colon u_i,t_i \in \Lambda, u_i^2=2, t_i^2=-2, \text{ for } i=1,2}.$$ Consequently, for any $g \in \SO^+(\Lambdatilde)_v$, it holds $\det\cdot \disc(g)=1 \text{ and }\det(g)=1$, from which we deduce that $\SO^+(\Lambdatilde)_v$ is an index $2$ subgroup of $\Nn(\Lambda)$. On the other hand, $\det\cdot \disc(R_s\circ R_{s_1})=1 \text{ and }\det(R_s\circ R_{s_1})=-1$, showing equality (\ref{genN}).
\end{enumerate} 
\end{rmk}